\documentclass[11pt, reqno,sumlimits]{article}
\usepackage{amssymb,amscd, epsfig, mathrsfs, xypic, amsmath,amsthm, color} 
\usepackage{hyperref} 
\usepackage{blindtext}

\xyoption{all}
\usepackage{enumerate}
\numberwithin{equation}{section}

\newcommand{\Sm}{{\mathbf{Sm}_{\bbF}}}
\newcommand{\Sch}{{\mathbf{Sch}_{\bbF}}}

\newcommand{\hGG}{{G\Gamma}}

\newcommand{\GQ}{{GQ}}
\newcommand{\GP}{{GP}}

\newcommand{\EE}{{E}}
\newcommand{\FF}{{F}}
\newcommand{\ee}{{e}}
\newcommand{\ff}{{f}}

\newcommand{\bbF}{{\mathbb F}}
\newcommand{\bbG}{{\mathbb G}}

\newcommand{\LL}{{\mathbb L}}

\newcommand{\PP}{{\mathbb P}}
\newcommand{\QQ}{{\mathbb Q}}

\newcommand{\ZZ}{{\mathbb Z}}

\newcommand{\sfp }{{\sfp}}

\newcommand{\frakK}{{\mathfrak K}}

\newcommand{\bfz}{{\mathbf z }}

\newcommand{\CH}{{C\!H}}
\newcommand{\CK}{{C\!K}}

\newcommand{\calG}{{\mathcal G}}
\newcommand{\calH}{{\mathcal H}}

\newcommand{\calK}{{\mathcal K}}
\newcommand{\calL}{{\mathcal L}}

\newcommand{\calO}{{\mathcal O}}
\newcommand{\calP}{{\mathcal P}}
\newcommand{\calQ}{{\mathcal Q}}
\newcommand{\calR}{{\mathcal R}}
\newcommand{\calS}{{\mathcal S}}

\newcommand{\scA}{{\mathscr A}}
\newcommand{\tscB}{\widetilde{\mathscr B}}
\newcommand{\scB}{{\mathscr B}}
\newcommand{\tscC}{\widetilde{\mathscr C}}
\newcommand{\scC}{{\mathscr C}}

\newcommand{\scO}{{\mathscr O}}

\newcommand{\tscS}{{\widetilde{\mathscr S}} }
\newcommand{\scS}{{\mathscr S}}

\newcommand{\scX}{{\mathscr X}}

\newcommand{\lan}{{\langle}}
\newcommand{\ran}{{\rangle}}

\newcommand{\inc}{\hookrightarrow}

\newcommand{\Fl}{{Fl}}
\newcommand{\Span}{\operatorname{Span}}

\newcommand{\Gr}{{Gr}}
\newcommand{\SG}{SG}

\newcommand{\codim}{\operatorname{codim}}

\newcommand{\Fun}{\operatorname{Fun}}

\newcommand{\gr}{\operatorname{gr}}

\newcommand{\Spec}{\operatorname{Spec}}

\newcommand{\rk}{{\operatorname{rank}}}

\newcommand{\id}{{\operatorname{id}}}

\newcommand{\Sp}{{Sp}}

\newcommand{\supp}{{\operatorname{supp}}}

\newcommand{\Pf}{{\operatorname{Pf}}}
\newcommand{\GX}{{{GX}}}

\newcommand{\LG}{{{LG}}}

\newcommand{\SP}{\calS\calP}

\newcommand{\GT}{G\Theta}
\newcommand{\tGB}{{\GT^*\!\!}}
\newcommand{\GB}{{\GT'\!}}
\newcommand{\GC}{{\GT}}

\newcommand{\OG}{{OG}}

\newcommand{\vece}{\pmb{e}}

\newtheorem{thm}{Theorem}[section]  
\newtheorem{cor}[thm]{Corollary} 
\newtheorem{lem}[thm]{Lemma}  
\newtheorem{prop}[thm]{Proposition} 
\newtheorem{df-pr}[thm]{Definition-Proposition}

\theoremstyle{definition} 
\newtheorem{defn}[thm]{Definition}
   
\newtheorem{rem}[thm]{Remark}
 
\newtheorem{example}[thm]{Example}

\begin{document}

\title{Degeneracy Loci Classes in $K$-theory \\--- Determinantal and Pfaffian Formula ---}
\author{Thomas Hudson, Takeshi Ikeda, Tomoo Matsumura, Hiroshi Naruse}
\date{}

\maketitle
\begin{abstract}
We prove a determinantal formula and Pfaffian formulas that respectively describe the $K$-theoretic degeneracy loci classes for Grassmann bundles and for symplectic Grassmann and odd orthogonal bundles. The former generalizes Damon--Kempf--Laksov's determinantal formula and the latter generalize Pragacz--Kazarian's formula for the Chow ring. As an application, we introduce the factorial $\GT/\GT'$-functions representing the torus equivariant $K$-theoretic Schubert classes of the symplectic and the odd orthogonal Grassmannians, which generalize the (double) theta polynomials of Buch--Kresch--Tamvakis and Tamvakis--Wilson.
\end{abstract}


\tableofcontents
\section{Introduction}
By a {\it degeneracy locus\/} of a Grassmann bundle, we will mean a subvariety defined by Schubert type conditions relative to a fixed flag of subbundles of the given vector bundle. We also consider symplectic and odd orthogonal degeneracy loci in the corresponding isotropic Grassmann bundles. The goal of this paper is to give explicit closed formulas for the classes of the 
structure sheaves of these degeneracy loci in the Grothendieck ring of coherent sheaves of the Grassmann bundle. 

In this paper, all schemes and varieties are assumed to be quasi-projective over an algebraically closed field $\bbF$ of characteristic zero, unless otherwise stated.

\subsection{Damon--Kempf--Laksov formula for $K$-theory }
Let us first review a classical result on degeneracy loci for cohomology. Let $X$ be a smooth variety. Fix a vector bundle $E$ of rank $n$ over $X$ and consider the Grassmann bundle $\xi: \Gr_d(E)\rightarrow X$, parametrizing rank $d$ subbundles of $E$, together with the tautological subbundle $U$ of $\xi^*E$. Suppose to be given a flag $F^{\bullet}$ of subbundles of $E$
\[
0=F^n\subset \cdots\subset F^1\subset F^0 =E, 
\]
where the superscript of  $F^i$ indicates its corank in $E$. Let $\lambda$ be a partition whose Young diagram fits in the $d\times (n-d)$ rectangle, {\it i.e.} a weakly decreasing sequence of nonnegative integers  $(\lambda_1, \dots, \lambda_d)$ such that $\lambda_1 \leq n-d$. One defines its associated degeneracy locus by 
\[
\Omega_{\lambda}:=\left\{(x,U_x) \in \Gr_d(E) \ |\ \dim (U_x \cap F^{\lambda_i+d-i}_x) \geq i\quad ( i=1,\dots, r)\right\}.
\]
One is interested in a formula for the class of $\Omega_\lambda$ as an element of the cohomology ring $H^*(\Gr_d(E))$ expressed in terms of the Chern classes of the vector bundles appearing in the setup. For arbitrary vector bundles $V$ and $W$, we denote the Chern polynomial by $c(V;u)=\sum_{i=0}^nc_i(V)u^i$ and define the relative Chern classes by $c(V-W;u):=c(V;u)/c(W;u)$. The formula, due to 
Damon (\cite{Damon1973}) and Kempf and Laksov (\cite{KL}), reads 
\begin{equation}\label{eq:KLD}
[\Omega_\lambda]=\det\left(c_{\lambda_i+j-i}(E/F^{\lambda_i+d-i}-U)\right)_{1\leq i,j\leq d},
\end{equation}
where we omit the pullback of vector bundles from the notation. It is worth noting that an important special case, when $\lambda$ has a rectangular shape, had been proven earlier by Porteous \cite{SimplePorteous}. 

The primary result of this paper extends the above Porteous--Damon--Kempf--Laksov formula  to the Grothendieck ring $K(\Gr_d(E))$. For our purpose, it turns out to be more convenient to use \emph{Segre classes} rather than Chern classes. Our definition of $K$-theoretic (relative) Segre classes $\scS_m(V-W)\;(m\in \mathbb{Z})$ for vector bundles $V$ and $W$ is a natural extension of the one given by Fulton in \cite{FultonIntersection}.  In the Grothendieck ring $K(\Gr_d(E))$ of the Grassman bundle, we have
\begin{equation}\label{eq:K-KLD}
[\mathcal{O}_{\Omega_\lambda}]=\det\left(\sum_{k=0}^\infty\binom{i-j}{k}(-1)^k\scS_{\lambda_i+j-i+k}\left(U^\vee-(E/F^{\lambda_i+d-i})^\vee\right)\right)_{1\leq i,j\leq d}.
\end{equation}
Note that this result is a direct $K$-theory analog of (\ref{eq:KLD}). 

It is reasonable to put the above results in the more general framework of $K$-theoretic formulas for degeneracy loci. Consider the full flag bundle $\Fl(E)\rightarrow X$ of the vector bundle $E$ over $X$. For each permutation $w\in S_n$ and a fixed complete flag $F^{\bullet}$ of $E$, one can define the associated degeneracy locus $\Omega_w$ in $\Fl(E)$ by the rank conditions given in terms of $w$, where its codimension is the length of $w$. Fulton and Lascoux \cite{FultonLascoux} proved that
the double Grothendieck polynomials of Lascoux and Sch\"utzenberger
express the degeneracy loci classes $[\scO_{\Omega_w}]\in K(Fl(E))$. This formula extends the corresponding result in cohomology due to Fulton \cite{Fulton2}. 
Later, F\'eher and Rym\'anyi \cite{FeherRymanyi03} reinterpreted the result in the context of Thom-polynomials. 
It is worth pointing out that if $w$ is the Grassmannian permutation with descend at $d$ associated to $\lambda$, then 
its double Grothendieck polynomial coincides with the corresponding {\it factorial Grothendieck polynomials\/}
 studied by McNamara \cite{McNamara}. Thus our formula leads to a determinantal formula for the factorial Grothendieck polynomials ({\it cf.} Remark \ref{remLenart}).
It should be noted that Lenart proved that the (single) Grothendieck 
 polynomial 
 associated to $\lambda$ can be expressed as a {\it flagged\/} Schur polynomials.
This approach leads to a determinantal formula different from 
 (\ref{eq:K-KLD}) (see Lascoux--Naruse \cite{LascouxNaruse} for details).
Let us also remark that 
Buch \cite{BuchQuiver} gave a formula for the structure sheaves of the {\it quiver loci\/} in $Fl(E)$, a family of subvarieties specified by more general rank conditions. Such formula is given as an integral linear combination of products of stable Grothendieck polynomials and it extends Fulton's {\it universal Schubert polynomials\/} (\cite{Fulton1999}) in cohomology. 
 Later, Buch himself together with Kresch, Tamvakis, and Yong \cite{BKTY2004} described more explicitly the coefficients by means of a combinatorial formula characterized by alternating signs. 
\subsection{Symplectic and odd orthogonal degeneracy loci}
The second aim of this paper is to provide analogous results for isotropic Grassmann bundles (type C). 

Let us start from type C, the symplectic case. Let $E$ be a symplectic vector bundle of rank $2n$ over a smooth  variety $X$ and suppose that we are given a flag $F^{\bullet}$ of subbundles of $E$
\[
0=F^n\subset F^{n-1} \subset \cdots \subset F^1 \subset F^0 \subset F^{-1} \subset \cdots \subset F^{-n}=E\,,  
\]
such that  $\rk(F^i)=n-i$ and $(F^i)^{\perp} = F^{-i}$ for all $i$. Fix a non-negative integer $k$, and let $\SG^k(E) \to X$ be the symplectic Grassmann bundle over $X$, \textit{i.e.}  the fiber at $x \in X$ is the Grassmannian $\SG^k(E_x)$ of $(n-k)$-dimensional isotropic subspaces of $E_x$. We denote the tautological vector bundle of $\SG^k(E)$ by $U$.

The combinatorial objects used to index the degeneracy loci in $\SG^k(E)$, called \emph{$k$-strict partitions}, have been developed in \cite{BuchKreschTamvakis1} by Buch, Kresch and Tamvakis.  A partition $\lambda$ is {\it $k$-strict\/} if $\lambda_i>k$ implies $\lambda_{i}>\lambda_{i+1}$. Let $\SP^k(n)$ denote the set of all $k$-strict partitions whose Young diagrams fit in the $(n-k)\times(n+k)$ rectangle. For each $\lambda$ in $\SP^k(n)$ of length $r$, we define the strictly decreasing integer sequence $(\chi_1,\ldots,\chi_{r})$ called the {\it characteristic index} (\cite{IkedaMatsumura}) by 
\begin{equation}
\chi_j=\#\{i\;|\;1\leq i<j,\;\lambda_i+\lambda_j>2k+j-i\}
+\lambda_j-k-j. \label{eq:defchi}
\end{equation}
The symplectic degeneracy loci $\Omega_{\lambda}^C \subset \SG^k(E)$ is then given by
\[
\Omega_{\lambda}^C = \{ (x, U_x)\in \SG^{k}(E) \ |\ \dim (U_x \cap F^{\chi_i}_x) \geq i, \ \ \  i=1,\dots, r\}.
\]
In order to describe its fundamental class $[\mathscr{O}_{\Omega_{\lambda}^C}]$  in $K(\SG^k(E))$, we introduce the classes
\begin{equation*}\label{scCdef}
\scC_m^{(\ell)} := \scS_{m}(U^{\vee}-(E/F^{\ell})^{\vee}), \ \ \ m\in \ZZ, \ \ \ -n\leq \ell \leq n.
\end{equation*}
Furthermore, let $r'$ be the smallest even integer that is larger than or equal to the length $r$ of $\lambda$, and we set
\[
D(\lambda)  := \{ (i,j) \in \mathbb{Z}^2|\  1\leq i < j\leq r',\; \chi_i + \chi_j < 0\}, 
\]
where $\chi_{r+1}$ is defined also by (\ref{eq:defchi}) with  $\lambda_{r+1}=0$. Then our formula has the form
\begin{equation}\label{eq:typeCmain}
[\mathscr{O}_{\Omega_{\lambda}^C}] = \sum_{I \subset D(\lambda)} \Pf\left(\sum_{p,q\in \ZZ \atop{p\geq 0, p+q\geq 0}} \ff_{pq}^{ij,I} \scC_{\lambda_i+d_i^I+p}^{(\chi_i)}\scC_{\lambda_j+d_j^I+q}^{(\chi_j)}\right)_{1\leq i<j\leq r' },
\end{equation}
where $d_i^I$ and $\ff_{pq}^{ij,I}$ are integers explicitly defined in terms of $\lambda$ and $I$ (see Section \ref{secSPf}).
Note that in the formula, $\scC_{m}^{(-n-1)}$ can appear if $r$ is odd and $r=n-k$ holds, since $\chi_{n-k+1}=-n-1$. However one can find that they 
occur only when $m\leq 0$, and in that case we can consistently set $\scC_{m}^{(-n-1)}=1$.

Before we deal with the orthogonal case, it can be worth to say a few words about an alternative setting in which our results can be reinterpreted, that of torus equivariant theories.
The Schubert classes in the torus equivariant cohomology of  isotropic flag varieties are described by the double Schubert polynomials of Ikeda, Mihalcea, and Naruse \cite{IkedaMihalceaNaruse} (see \cite{TamvakisJAlgGeom} for a historical account and a description of other approaches). They are the double (or equivariant) version of Billey--Haiman's polynomials \cite{BilleyHaiman}. The equivariant cohomology analogue of (\ref{eq:typeCmain}) was studied by Wilson in her thesis \cite{WilsonThesis}, where she introduced an explicit family of polynomials, called the \emph{double $\vartheta$-polynomials}, written in terms of Young's raising operators. She conjectured that they coincide with the double Schubert polynomials associated with the Weyl group elements corresponding to the $k$-strict partitions. Note that for the non-equivariant case the corresponding raising operator formula was proved by Buch, Kresch and Tamvakis \cite{BuchKreschTamvakis1}. Wilson's conjecture was proved by Ikeda and Matsumura \cite{IkedaMatsumura} in the form which is obtained from (\ref{eq:typeCmain}) by specialization to cohomology. It is worth noting that the Pfaffian sum expression had already appeared in Proposition 2 in \cite{BKTQGiam}. By translating the technique of left divided difference operators employed in \cite{IkedaMatsumura}, Wilson--Tamvakis \cite{TamvakisWilson} reproved Wilson's conjecture in terms of raising operators and also obtained a ring presentation for $H_T^*(\SG^k(n))$. 

If we set $k=0$, the isotropic Grassmannain $\SG^0(n)$ is known as the Lagrangian Grassiannian $\LG(n)$ and in this case (\ref{eq:typeCmain}) reduces to a single Pfaffian. Ikeda and Naruse \cite{IkedaNaruse} introduced the $K$-theoretic (factorial) $Q$-functions $\GQ_\lambda$, which represent the class of the structure sheaf $[\scO_{\Omega_\lambda}]$ in the torus equivariant $K$-ring $K_T(\LG(n))$. As a consequence, (\ref{eq:typeCmain}) allows one to express $\GQ_\lambda$ as a single Pfaffian. If we further specialize this to the cohomology ring, the formula coincides with Kazarian's degeneracy loci formula \cite{Kazarian} which, in the context of torus equivariant cohomology, was later proved by Ikeda \cite{Ikeda2007}  by employing localization techniques. The detailed comparison between \cite{Kazarian} and \cite{Ikeda2007} was explained in \cite[\S 10]{IkedaMihalceaNaruse}. These results imply that Ivanov's factorial $Q$-functions \cite{Ivanov04} describe the Lagrangian degeneracy loci classes in cohomology and, equivalently, the torus equivariant cohomology Schubert class of the Lagrangian Grassmannian. When the reference flag $F^{\bullet}$ is trivial, one recovers Pragacz's formula for nonequivariant cohomology $H^*(\LG(n))$, which shows that Schur's $Q$-functions represent the cohomology Schubert classes of the Lagrangian Grassmannian. In this case, the Pfaffian formula goes back to \cite{Schur}, the original paper on $Q$-functions by Schur. 

Let us now move on to type B, the odd orthogonal case. Let $E$ be a vector bundle of rank $2n+1$ over $X$, endowed with a nondegenerate symmetric form. Let $\OG^k(E)$ be its Grassmannian of isotropic subbundles of rank $n-k$ and $U$ its tautological bundle. Similarly to the symplectic case, for a $k$-strict partition $\lambda\in \SP^k(n)$, one can define the degeneracy locus $\Omega_\lambda^B\subset \OG^k(E)$ (see Definition \ref{dfOmegaB}). The formula for $[\mathscr{O}_{\Omega_\lambda}] \in K(\OG^n(n))$ is given by the same expression of the symplectic case, except that one replaces $\scC_m^{(\ell)}$ by 
$\scB_m^{(\ell)}$, which is defined by 
\[
\scB_m^{(\ell)} := \begin{cases}
\scS_m(U^\vee- (E/F^{\ell})^{\vee}) & (-n\leq \ell <0)\\
\sum_{s=0}^\infty 2^{-s-1}
\scS_{m+s}(U^\vee- (E/F^{\ell})^{\vee}) & (0\leq \ell \leq n),
\end{cases}
\]
where the reference flag $F^{\bullet}$ in this case is defined by (\ref{eq:flagB}).
The difference between the symplectic case and the odd orthogonal case is far more subtle in $K$-theory than in cohomology. The double Schubert polynomials of type C and B, which represent the equivariant Schubert classes for symplectic and odd orthogonal flag varieties, only differ by a power of $2$ (\cite{BilleyHaiman}, \cite{IkedaMihalceaNaruse},  \textit{cf.} \cite[p.2675]{BergeronSottile}) and by restriction the same holds for the formulas for the Grassmannian loci. This is not the case for the $K$-theory classes: one can see that even the entries of the Pfaffian formula are modified in a non-trivial way from type C to type B.

It would be interesting to extend our method and formula to the $K$-theory of even orthogonal isotropic Grassmannians (type D). An important progress related to this problem was recently made by Tamvakis \cite{Tamvakis1506.04441}, who proved a formula for the Schubert classes in the equivariant cohomology of the type D Grassmannian case by adapting the method of left divided difference operators used in \cite{IkedaMatsumura} and \cite{TamvakisWilson}. 
Another related recent progress in cohomology is due to 
Anderson and Fulton \cite{AndersonFulton},\cite{AndersonFulton2}. By using a resolution of singularities they obtained determinantal and Pfaffian formulas for the degeneracy loci (or equivalently the double Schubert polynomials) which are indexed by (signed) permutations belonging to a class called \emph{Vexillary}, which extends the Grassmannian permutations of type A, B, C and D. 

\subsection{Double Grothendieck polynomials and $\GC$ and $\GB$-functions }
One of the motivations for our work comes from the study of the canonical polynomials representing the structure sheaves 
of the Schubert varieties of homogenous spaces. With this purpose in mind we introduce the functions $\GT_{\lambda}(x,a|b)$ and $\GT'_{\lambda}(x,a|b)$ and show that they describe the equivariant Schubert classes in $K$-theory of, respectively, symplectic and odd orthogonal Grassmannians. In the maximal cases such polynomials respectively coincide with the $\GQ_{\lambda}(x|b)$ and $\GP_{\lambda}(x|b)$ polynomials  introduced by Ikeda--Naruse \cite{IkedaNaruse}. Furthermore, the specialization of our polynomials to cohomology recovers the double $\vartheta$-functions, the corresponding double Schubert polynomials of type C and B.

Recently, Kirillov--Naruse \cite{KrNr} introduced the \textit{double Grothendieck} polynomials for all classical types, hence solving the problem for all flag varieties. As a consequence their result provides information for all other homogenous spaces and in particular for the Grassmannians we consider. To be more precise, they give two different combinatorial descriptions, one in terms of compatible sequences and the other using pipe dreams, while our work provides, in a more restricted setting, a third alternative description by means of closed formulas.

\subsection{Methods of proofs}
Although each type has its own peculiarities, the strategy behind the proofs of the main theorems is essentially the same. The first step is of geometric nature and consists in constructing a resolution of singularities of the degeneracy locus inside a tower of projective bundles. For this we follow Damon and Kempf--Laksov for type A and Kazarian for type B and C. In the second part we compute the desired fundamental class by taking a pushforward of the class of the desingularization along the projective tower. This procedure relies on the key fact that the degeneracy loci have at worst rational singularities and it is performed by making use of connective $K$-theory $\CK^*$, a functor introduced by Levine--Morel in \cite{LevineMorel} which can be specialized to both $K$-theory and cohomology. In order to move from one level of the tower to the next we perform a Gysin computation which is given in terms of Segre classes and the final form is then obtained by induction in a nontrivial way. Dealing with the inductive step requires us to expand Kazarian's manipulation, a kind of \textit{umbral calculus}, which uses formal monomials instead of the indices of the Segre classes. Then it becomes possible to handle complicated expressions in Segre classes by treating them as rational functions. The desired formulas are then obtained by using either the Vandermonde determinant or the Schur Pfaffian identity.
  
It is worth stressing that for non-maximal isotropic Grassmann bundles the natural generalization of Kazarian's geometric construction does not produce an actual smooth resolution of singularities. To bypass this obstacle it becomes necessary to work with Borel--Moore homology theories and more precisely with $\CK_*$, the homological counterpart of $\CK^*$.

\subsection{Beyond $K$-theory: Generalized cohomology theories}
In order to gain a better understanding of the role played by $\CK^*$ in our work, it can be worth to have a look at it from the perspective of generalized Schubert calculus. With the introduction due to Quillen \cite{ElementaryQuillen} of the concept of oriented cohomology theory, it became clear that $H^*$ and $K^0$ were just two examples of an entire family of functors for which the typical problems associated to the study of Schubert varieties could be formulated. In recent years this line of research gained more attention following the construction by Levine--Morel \cite{LevineMorel} of algebraic cobordism $\Omega^*$, the universal oriented cohomology theory in the context of algebraic geometry. To some extent the final goal of generalized Schubert calculus would be to lift every result to $\Omega^*$ and in this perspective our work within connective $K$-theory, which is defined as $\CK^*:=\Omega^*\otimes_{\LL}\ZZ[\beta]$, can be viewed as the second step towards the corresponding universal formulas for $\Omega^*$. Here by universal we mean that they can be specialized to all other theories, in the same way in which the formulas for $\CK^*$ allow one to recover the statements for $K$-theory and cohomology. 

The difficulties one faces in the more general setting are of both computational and theoretical nature. There is in particular one aspect in which most theories substantially differ from $\CK^*$: in general not all Schubert varieties have a well defined fundamental class. This problem can be avoided by 
defining the Schubert classes as the pushforward of some resolution of singularities, but in this way the outcome depends on the choice made.

 When dealing with flag manifolds $G/B$, it is natural to consider Bott--Samelson resolutions as they are available for all Schubert varieties and this is the setting most commonly chosen. Bressle--Evens \cite{BraidBressler,SchubertBressler} were the first to consider this approach in the context of topological cobordism and they identified the correct generalization of divided difference operators, a key tool in dealing with Bott--Samelson classes. For algebraic cobordism, the ring presentation of $\Omega^*(G/B)$ as well as some algorithms for describing its multiplicative structure were independently obtained by Calmes--Petrov--Zainoulline \cite{SchubertCalmes} and Hornbostel--Kiritchenko \cite{SchubertHornbostel}. These results were later expanded to flag bundles (or, equivalently, to the equivariant setting) by Kiritchenko--Krishna \cite{EquivariantKiritchenko} and Calmes--Zainoulline--Zhong \cite{EquivariantCalmes}. In the latter work the authors also used Bott--Samelson resolutions to deal with the equivariant Schubert calculus of more general homogenous spaces. For an example of some special cases of Schubert varieties of flag bundles whose classes can be described without making use of a resolution of singularieties, we refer the reader to   \cite{GeneralisedHudson}.

In spite of being compatible with the action of divided difference operators, the classes obtained from Bott--Samelson resolutions are usually not well behaved from the point of view of stability, \textit{i.e.} the polynomial expressions defining them depend on the rank of the bundle used to define the ambient space in which Schubert varieties live. This is one of the reasons which suggests to experiment with different resolutions, in those contexts where these are available. For Grassmann bundles this is indeed the case and one can choose between the resolutions constructed through towers of projective bundles or those constructed through towers of Grassmann bundles. The first approach, which results in an ideal sequel of this paper, was developed in \cite{SegreHudson} by Hudson--Matsumura, while the second was carried out in \cite{UniversalNakagawa} by Nakagawa--Naruse.
\subsection{Organization of the paper}
In Section \ref{sec:preCK}, we first present some preliminary facts on connective $K$-theory and then introduce the $K$-theoretic Segre classes. 
In Section \ref{sec: type A} we prove our first main result for type A, \textit{i.e.} the degeneracy loci formula for the ordinary Grassmann bundles. In Section \ref{seccombBC} we recall the combinatorics used to describe the degeneracy loci in isotropic Grassmann bundles. 
In Section \ref{sec:BM}, we introduce $\CK_*$, a homological analogue of connective $K$-theory that can be applied to singular schemes. 
In Section \ref{section:Pf non max} we prove the Pfaffian formula for symplectic Grassmann bundles.
In Section \ref{sec:typeB} we first derive a formula for the degeneracy loci of the quadric bundle of an orthogonal vector bundle of odd rank. Then we use it to prove the Pfaffian formula for odd orthogonal Grassmann bundles.
In Section \ref{sec7}, we explain a few basic facts on the torus equivariant connected $K$-theory.
We continue by discussing the stability of Schubert classes when we consider the large rank limit of the corresponding Weyl groups.
In Section \ref{sec8}, we introduce the algebraic framework needed in order to express the $K$-analog of the double Schubert polynomials of \cite{IkedaMihalceaNaruse}. In particular, we introduce the graded ring $\calK_{\infty}$ spanned by the {\it double Grothendieck polynomials\/} and study its GKM description.
In Section \ref{sec: GTheta}, we give the definition of $G\Theta$ and $G\Theta'$-functions as elements of $\calK_{\infty}$ and identify them with the torus equivariant $K$-theoretic Schubert classes of isotropic Grassmannians.
\section{Segre classes in connective $K$-theory}\label{sec:preCK}
\subsection{Preliminaries}\label{sec:preliminaryOnCK}
Connective $K$-theory denoted by $\CK^*$ is an example of oriented cohomology theory built out of the algebraic cobordism of Levine and Morel. It consists of a contravariant functor and pushforwards for projective morphisms which satisfy some axioms. For the detailed construction we refer the reader to \cite{DaiLevine}, \cite{Hudson}, and \cite{LevineMorel}. In this section, we recall some preliminary facts on $\CK^*$, especially on Chern classes.

Let $X$ be a smooth variety. The connective $K$-theory of $X$ interpolates the Grothendieck ring $K(X)$ of the algebraic vector bundles on $X$ and the Chow ring $\CH^*(X)$ of $X$. Connective $K$-theory assigns to $X$  a commutative graded algebra $\CK^*(X)$  over the coefficient ring $\CK^*({pt})$. Such ring is isomorphic to the polynomial ring $\ZZ[\beta]$ by setting $\beta$ to be the class of degree $-1$ obtained by pushing forward the fundamental class of the projective line along the structural morphism $\PP^1 \to {pt}$. The $\ZZ[\beta]$-algebra $\CK^*(X)$ specializes to the Chow ring $\CH^*(X)$ and the Grothendieck ring $K(X)$ by respectively setting $\beta$ equal to $0$ and $-1$. For any closed equidimensional subvariety $Y$ of $X$, there exists an associated fundamental class  $[Y]_{\CK^*}$ in $\CK^*(X)$. In particular,  $[Y]_{\CK^*}$ is specialized to the class $[Y]$ in $\CH^*(X)$ and also to the class of the structure sheaf $\mathcal{O}_Y$ of $Y$ in $K(X)$. In the rest of the paper, we denote the fundamental class of $Y$ in $\CK^*(X)$ by $[Y]$ instead of $[Y]_{\CK^*}$.

As a feature of any oriented cohomology theory, connective $K$-theory admits a theory of Chern classes. For line bundles $L_1$ and $L_2$ over $X$, their 1st Chern classes $c_1(L_i)\in  \CK^1(X)$ satisfy the following equality:
\begin{equation}\label{L tensor L}
c_1(L_1\otimes L_2)=c_1(L_1)+c_1(L_2)+\beta c_1(L_1)c_1(L_2).
\end{equation}
Note that the operation
\[
(u,v)\mapsto u\oplus v:= u+v+\beta uv
\]
is an example of commutative one-dimensional formal group law, which is a key ingredient of oriented cohomology theories. 
The reader should be aware that our sign convention for $\beta$ is opposite to the one used in the references \cite{DaiLevine}, \cite{Hudson}, \cite{LevineMorel}. It follows from (\ref{L tensor L}) that
\begin{equation}\label{c_1 dual}
c_1(L^{\vee}) = \frac{-c_1(L)}{1+\beta c_1(L)}.
\end{equation}
Therefore it is convenient to introduce the following notation for the \emph{formal inverse}: 
\[
u\ominus v = \frac{u-v}{1+\beta v} \ \ \ \ \mbox{ and } \ \ \ \ \bar u := \frac{-u}{1+\beta u}.
\]
It is easy to check that $u \oplus \bar v = u \ominus v$. In general, for a vector bundle $\EE\rightarrow X$ of rank ${\ee}$, the Chern classes $c_i(E)$ can be defined using Grothendieck's argument for the Chow ring. Indeed, consider the dual projective bundle  $\PP^*(\EE)\stackrel{\pi}\rightarrow X$ and the exact sequence of vector bundles
on $\PP^*(\EE)$
\[
0\longrightarrow \calH\longrightarrow\pi^* \EE\longrightarrow\calQ\longrightarrow 0, 
\]
where $\calH$ is the rank $(e-1)$ subbundle of $\pi^* \EE$ whose fiber at $x\in\PP^*(\EE)$ is precisely the hyperplane represented by $x$ itself. We call $\calQ$ the universal quotient line bundle. The ring $\CK^*(\PP^*(\EE))$ is generated by $\tau:=s^*s_*(1)$ as a $\CK^*(X)$-algebra, where $s$ is the zero section of  $\calQ$. One of the axioms of oriented cohomology theories states that there is a relation
\begin{equation}\label{rel chern}
\sum_{i=0}^{{\ee}}(-1)^i c_i(\EE)\tau^{{\ee}-i}=0, 
\end{equation}
for some $c_i(\EE)\in \CK^i(X), \; 0\leq i \leq e$. This relation and the normalization $c_0(E)=1$ uniquely determine the classes $c_i(E)$. Thus we can {\em define\/} these elements  to be the Chern classes of $E$. One can derive the fact that the canonical generator $\tau$ concides with $c_1(\calQ)$.

For computations it is convenient  to combine the Chern classes into a Chern polynomial 
\[
c(\EE;u):=\sum_{i=0}^{\ee} c_i(\EE) u^i.
\]
A formal difference $E-F$ of vector bundles $E$ and $F$ over $X$ defines a \emph{virtual vector bundle}. Two such virtual vector bundles $E-F$ and $E'-F'$ are considered to be identical if they are equal in the Grothendieck ring $K(X)$ of vector bundles over $X$. Let us set $c(\EE-\FF;u):=c(\EE;u)/c(\FF;u)$. This is well-defined because of the Whitney formula: given a short exact sequence of bundles $0\rightarrow \FF \rightarrow \EE \rightarrow W \rightarrow 0$, one has  $c(\EE;u)=c(\FF;u)c(W;u)$. If $x_1,\ldots,x_e$ and $y_1,\ldots,y_f$ are Chern roots of $E$ and $F$ respectively, then the explicit expression of $c_p(E-F)$ is given by
\begin{equation}\label{eq:eh}
c(E-F;u)=\frac{\prod_{i=1}^e(1+x_i u)}{\prod_{j=1}^f(1+y_j u)}=\sum_{p=0}^\infty\sum_{j=0}^p(-1)^je_{p-j}(x_1,\ldots,x_a)h_j(y_1,\ldots,y_b)\cdot u^p,
\end{equation}
where $e_i$ and $h_i$ are, respectively, the elementary and complete symmetric functions of degree $i$.


Finally we conclude this section with the following two lemmas that will be used in the rest of the paper. The first is an elementary expansion formula of the top Chern class of the tensor product of a line bundle with a vector bundle. The second is the $K$-theoretic version of the classical fact that the top Chern class is equal to the Euler class.
\begin{lem}\label{tensor decomp}
Let $E$ be a vector bundle of rank $e$ and $L$ a line bundle over $X$.  Then in $\CK^*(X)$ we have
\[
c_{e}(L\otimes E) = \sum_{p=0}^{e} c_p(E)\sum_{q=0}^p\binom{p}{q}\beta^qc_1(L)^{e-p+q} .
\]
\end{lem}
\begin{proof}
Let $x_1,\dots, x_{\ee}$ be Chern roots of $E$. Let $c_1:=c_1(L)$. Then
\begin{eqnarray*}
c_\ee(L\otimes E) &= & (x_1\oplus c_1) \cdots (x_{\ee} \oplus c_1)
= (x_1(1+\beta c_1)+c_1) \cdots (x_{\ee}(1+\beta c_1) + c_1)\\
&=& \sum_{p=0}^{\ee} e_p(x)c_1^{{\ee}-p}(1+\beta c_1)^p.
\end{eqnarray*}
Thus by expanding the right hand side we have the desired formula.
\end{proof}
\begin{lem}\label{lemG-B}
Let $E$ be a vector bundle of rank $e$ over a scheme $X$. Let $s$ be a section of $E$ over $X$ and $Z$ be its zero scheme. 
\begin{enumerate}[(a)]
\item If $X$ is Cohen--Macaulay and the codimension of $Z$ in $X$ is $e$, then the section $s$ is regular and $Z$ is Cohen--Macaulay.
\item If $X$ is smooth, then $[Z] = c_e(E) \in \CK^e(X)$.
\end{enumerate}
\end{lem}
\begin{proof}
In (a), the regularity follows from Theorem 31 in \cite{MatsumuraComm} and then the Cohen--Macaulayness follows from Theorem 30 in \cite{MatsumuraComm}. The claim (b) is a consequence of Lemma 6.6.7. in \cite{LevineMorel} ({\it cf.} Example 3.2.16 and 14.1.1 in \cite{FultonIntersection}). 
\end{proof}

\begin{rem}
In order to compute the classes of the degeneracy loci in the symplectic and odd orthogonal Grassmann bundles, we must work over the so-called oriented Borel--Moore homology developed by Levine--Morel. Thus we need a statement analogous to Lemma \ref{lemG-B}(b), which also follows from Lemma 6.6.7 \cite{LevineMorel}.
\end{rem}

\subsection{$K$-theoretic Segre classes}\label{secKSeg}
In this section, we define the \emph{Segre classes} of a vector bundle, and then show a formula for their generating function, through which we introduce the \emph{relative} Segre classes. The main result used in the rest of the paper is Proposition \ref{push of tensor}. 

\begin{defn}
Let $\EE$ be a vector bundle over $X$ of rank $\ee$. Let $\pi: \PP^*(\EE)\to X$ be the dual projective bundle of $\EE$ and $\calQ$ its universal quotient line bundle of $\EE$. Let $\tau=c_1(\calQ)$. For each integer $m\in \ZZ$, define
\begin{equation}\label{df:Segre}
\scS_m(\EE) := 
\begin{cases}
\pi_*(\tau^{m+\ee-1}) & (m>0)\\
(-\beta)^{-m} & (m\leq 0).
\end{cases}
\end{equation}
\end{defn}
The class $\scS_m(\EE)$ is a $K$-theoretic version of the Segre class defined in \cite{FultonIntersection} although there is a sign difference. We remark that if $\EE$ is a line bundle, we have $\calQ=E$ and $\pi=\id_X$, {\it i.e.} $\scS_m(\EE)=c_1(\EE)^m$ for all $m\geq 0$. 

The next lemma, which is due to Buch \cite[Lemma 6.6 amd 7.1]{BuchQuiver}, shows that the definition of the Segre class of degree $m$ such that $-e+1\leq m\leq 0$ is consistent with the one for the positive degree. We give another proof, using Vishik's formula \cite{Vishik}.
\begin{lem}\label{push is beta}
We have $\pi_*(\tau^{m+\ee-1})  = (-\beta)^{-m}$ for $-\ee+1 \leq m \leq 0$.
\end{lem}
\begin{proof}
Let $z_1,\dots,z_{\ee}$ be the Chern roots of $\EE$. From Vishik's formula in algebraic cobordism \cite[Proposition 5.29, p.548]{Vishik}, we have
\begin{eqnarray*}
&&\pi_*(\tau^{-m+\ee-1}) 
= \sum_{i=1}^{\ee}  \frac{z_i^{-m+{\ee}-1} }{\prod_{j\not=i} (z_i\ominus z_j)}
= \sum_{i=1}^{\ee}  \frac{z_i^{-m+{\ee}-1} \prod_{ j\not=i}(1+\beta z_j)}{\prod_{j\not=i} (z_i-z_j)}.
\end{eqnarray*}
The right hand side is a polynomial in $\beta$ of degree at most $e-1$. Let us denote it by $F(\beta)$. It is easy to see that $F(-1/z_i)$ is equal to $z_i^{-m}$ for $1\leq i\leq e$. This property uniquely determines $F(\beta)$ by degree reason. Obviously $(-\beta)^m$ also satisfies the same conditions, so we have $F(\beta)=(-\beta)^m$.
\end{proof}
\begin{rem}\label{rem S negative}
By using Vishik's formula, we can also show that $\scS_m(E) = \scS_m(E \oplus O_X^n)$ for any integer $n \geq 1$, where $O_X$ is the trivial line bundle over $X$. Thus, together with Lemma \ref{push is beta}, the definition of the Segre classes of negative degree in (\ref{df:Segre}) is consistent with the one for positive degrees.
\end{rem}
By setting
\[
\scS(E;u):=\sum_{m\in \ZZ} \scS_m(\EE) u^{m},
\]
we can show the following theorem.
\begin{thm}\label{thm:Gysin} 
We have
\begin{equation}\label{Segre comp}
\scS(E;u)= \frac{1}{1 + \beta u^{-1}} \frac{c(\EE;\beta)}{c(\EE;-u)},
\end{equation}
where $\dfrac{1}{1 + \beta u^{-1}}$ is expanded in the form $\sum_{i=0}^\infty (-\beta)^i u^{-i}$. 
\end{thm}
\begin{proof}
For $m\geq 1$, we multiply (\ref{rel chern}) by $\tau^{m-1}$, and push it forward to obtain 
\begin{equation}\label{eq:rec-Segre}
\sum_{i=0}^e (-1)^i c_i(E)\pi_*(\tau^{m-i+e-1})=0.
\end{equation}
One observes that by using Lemma \ref{push is beta} together with (\ref{eq:rec-Segre}), we can  successively determine $\pi_{*}(\tau^k)$ for $k\geq e$ uniquely. This recursion relation is equivalent to requiring that $c(E;-u)\scS(E;u)$ has only non-positive powers in $u$. Thus the series $c(E;-u)\scS(E;u)$ equals to
\[
\sum_{m\leq 0}\left(\sum_{i=0}^e(-1)^i c_i(E)(-\beta)^{i-m})\right)u^m=\!\left(\sum_{m\leq 0}(-\beta)^{-m}u^m\right)\!\!\left(\sum_{i=0}^ec_i(E)\beta^{i}\right).
\]
Hence we have $\scS(E;u)=(1+\beta u^{-1})^{-1}c(E;\beta)/c(E;-u)$. 
This proves the theorem.
\end{proof}
\begin{rem}
Let $G_m(z_1,\dots,z_d)$ be the Grothendieck polynomials associated to a partition of length one (see (\ref{eq:ratio})). We can rewrite Buch's formula (\cite[Lemma 6.6]{BuchQuiver} as the following generating function
\[
\sum_{m\in \ZZ} G_m(z_1,\dots,z_d) u^m = \frac{1}{1+\beta u^{-1}} \prod_{i=1}^d\frac{1+\beta z_i}{1-z_i u},
\]
where we set $G_m(z_1,\dots,z_d)=(-\beta)^{-m}$ for $m\leq 0$. Therefore, by comparison with (\ref{eq:rec-Segre}), we find that the Segre class $\scS_m(\EE)$ of a vector bundle $E$ of rank $d$ coincides with $G_m(z_1,\dots,z_d)$ if $z_1,\dots,z_d$ are viewed as Chern roots of $E$. Buch \cite[Lemma 7.1]{BuchQuiver} proved this fact by a different method.
\end{rem}
\begin{rem}\label{rem beta+u}
It follows from a simple identity $\frac{1+\beta x}{1 - u x} = \frac{1}{1+(u+\beta) \bar x}$ that
\[
\frac{c(\EE;\beta)}{c(\EE;-u)} = \frac{1}{c(\EE^{\vee}; u+\beta)}.
\]
\end{rem}
Theorem \ref{thm:Gysin} allows us to extend the definition of the $K$-theoretic Segre classes to virtual vector bundles as follows.
\begin{defn} 
We define the \emph{Segre class} $\scS_{m}(E-F)$ for a virtual bundle $E-F$ by the following generating function:
\begin{equation}\label{segre vir}
\scS(\EE-\FF;u):=\sum_{m\in \ZZ} \scS_{m}(\EE-\FF) u^{m}= \frac{1}{1 + \beta u^{-1}} \frac{c(\EE - \FF;\beta)}{c(\EE-\FF;-u)}.
\end{equation}
Since the right hand side is written in terms of Chern classes, the notion is well-defined. 
\end{defn}
By Remark \ref{rem beta+u}, we can observe that
\[
\scS(\EE-\FF;u)=\scS(\EE;u)c(\FF^{\vee}; u+\beta).
\] 
Thus we have
\begin{equation}\label{formula0}
\scS_{m}(\EE-\FF) =\sum_{p=0}^{\rk(F)} c_p(\FF^{\vee}) \sum_{q=0}^p \binom{p}{q}\beta^q\scS_{m-p+q}(\EE).
\end{equation}
Finally, we prove that we can obtain $\scS_m(\EE-\FF)$ also as the pushforward of certain Chern classes. 
\begin{prop}\label{push of tensor} 
Let $\pi: \PP^*(E)\rightarrow X$ be the dual projective bundle of a rank $e$ vector bundle $E\to X$ and let $\tau$ be the first Chern class of its tautological quotient line bundle $\calQ$. Let $F$ be a vector bundle over $X$ of rank $f$ and denote its pullback to $\PP^*(E)$ also by $F$. We have 
\[
\pi_*\left(\tau^sc_f(\calQ \otimes F^{\vee})\right) = \scS_{s+f-e+1}(E-F).
\]
\end{prop}
\begin{proof}
By using Lemma \ref{tensor decomp} and Theorem \ref{thm:Gysin}, we can compute
\begin{eqnarray*}
\pi_*\left(\tau^sc_{f}(\calQ \otimes F^{\vee})\right)
&=&\pi_*\left(\sum_{p=0}^f c_p(F^{\vee}) \sum_{q=0}^{p}\binom{p}{q}\beta^q\tau^{s+f-p+q}\right)\\
&=&\sum_{p=0}^{f} c_p(F^{\vee}) \sum_{q=0}^{p}\binom{p}{q}\beta^q\scS_{s+f-e+1-p+q}(E).
\end{eqnarray*}
Thus (\ref{formula0}) implies the claim.
\end{proof}
\begin{rem}\label{remLine}
As noted before, when $E$ is a line bundle we have $\scS_m(E) = c_1(E)^m$ for all $m\geq 0$. This, together with Lemma \ref{tensor decomp} and (\ref{formula0}), implies that $c_1(E)^sc_{f}(E \otimes F^{\vee}) = \scS_{f+s}(E-F)$.
\end{rem}
\section{Determinantal formula for Grassmann bundles}\label{sec: type A}
The goal of this section is to prove the determinantal formula for the degeneracy loci in a Grassmann bundle (Theorem \ref{MainA}). Following \cite{Damon1973} and \cite{KempfLaksov}, for each locus we construct a resolution of singularities in a tower of projective bundles. We show that the fundamental classes of the resolution is a product of Chern classes and that  it gives the degeneracy locus class when it is pushed forward along the tower (Section \ref{sec: resolution A}). In Section \ref{sec:pushlemA}, we obtain a lemma that describes the pushforward of a single Chern class along each stage of the tower. In Section \ref{sec: umbral}, we develop a calculus to compute the pushforward of the product of Chern classes by generalizing the technique used by Kazarian in \cite{Kazarian}. This technique will be used in later sections as well.

\subsection{Degeneracy loci}\label{sec: deg loci type A}
Let $E$ be a vector bundle  of rank $n$ over a smooth  variety $X$. Let $\xi: \Gr_d(E) \to X$ be the Grassmann bundle parametrizing rank $d$ subbundles of $E$. Let $U$ be the tautological subbundle of $\xi^*E$ over $\Gr_d(E)$. We denote an element of $\Gr_d(E)$ by $(x,U_x)$ where $x\in X$ and $U_x$ is a $d$-dimensional subspace of $E_x$.  

Suppose to be given a complete flag of $E$
\[
0=F^n\subset \cdots \subset F^1\subset E,
\]
where the superscript indicates the corank, \textit{i.e.} $\rk(F^k) = n-k$. We denote the pullback of  $E$ and $F^i$ along $\xi$ also by $E$ and $F^i$. 

Let $\calP_d$ be the set of all partitions $(\lambda_1,\dots,\lambda_d)$ with at most $d$ parts. Let $\calP_d(n)$ be the subset of $\calP_d$ consisting of partitions $\lambda$ such that $\lambda_i\leq n-d$ for all $i=1,\dots,d$. Define the \emph{length} of $\lambda$ to be the number of nonzero parts in $\lambda$.

For each $\lambda\in \calP_d(n)$ of length $r$, define \emph{the type A degeneracy loci} $\Omega_{\lambda}$ in $\Gr_d(E)$ by
\begin{equation}\label{dfOmegaA}
\Omega_{\lambda}:=\{(x,U_x) \in \Gr_d(E) \ |\ \dim (U_x \cap F^{\lambda_i+d-i}_x) \geq i, i=1,\dots, r\}.
\end{equation}
Note that for $i > r$  the condition $\dim (F^{\lambda_i+d-i}_x \cap U_x) \geq i$ is vacuous. 
\subsection{Resolution of singularities}\label{sec: resolution A}
Over $\Gr_d(E)$, we consider the $r$-step flag bundle $\Fl_r(U)$ of $U$, where the fiber at $(x,U_x) \in \Gr_d(E)$ consists of flags of subspaces $(D_1)_x \subset \cdots \subset (D_r)_x$ of $U_x$ with $\dim (D_i)_x=i$. Let $D_1\subset \cdots \subset D_r$ be the corresponding flag of tautological subbundles over $\Fl_r(U)$. We let $D_0=0$. We denote an element of $\Fl_r(U)$ by $(x,U_x, (D_{\bullet})_x)$.
We can also realize this bundle as a tower of projective bundles: 
\begin{eqnarray}
&&\pi: \Fl_r(U)=\PP(U/D_{r-1}) \stackrel{\pi_r}{\longrightarrow} 
\PP(U/D_{r-2}) \stackrel{\pi_{r-1}}{\longrightarrow} \cdots \ \ \ \ \ \ \ \ \ \nonumber\\\label{P tower}
&&\ \ \ \ \ \ \ \ \ \ \ \ \ \ \ \ \ \ \ \ \ \ \cdots\stackrel{\pi_3}{\longrightarrow} \PP(U/D_1) \stackrel{\pi_2}{\longrightarrow} \PP(U)  \stackrel{\pi_1}{\longrightarrow} \Gr_d(E).
\end{eqnarray}
We regard $D_j/D_{j-1}$ as the tautological line bundle of $\PP(U/D_{j-1})$. As mentioned in the introduction we will omit the pullback of vector bundles from the notation. 
\begin{defn}\label{def:Y_j(A)}
Define a sequence of subvarieties $Y_r \subset Y_{r-1} \subset \cdots \subset Y_1 \subset \Fl_r(U)$ as follows. For each $j=1,\dots,r$, let
\[
Y_j := \{ (x,U_x, (D_{\bullet})_x) \in  \Fl_r(U) \ |\ (D_i)_x \subset F^{\lambda_i+d-i}_x, i=1,\dots,j\}.
\]
We set $Y_0:=\Fl_r(U)$ and $Y_{\lambda}:=Y_r$. Let $\iota_j: Y_{j} \to Y_{j-1}$ be the canonical inclusion. 
\end{defn}
It is easy to see that $Y_{\lambda}$ is birational to $\Omega_{\lambda}$ along $\pi$. Thus we obtain the following description of $[\Omega_{\lambda}]$.
\begin{lem}\label{prod chern}
The class of $\Omega_{\lambda}$ in $\CK^*(\Gr_d(E))$ is given by
\[
[\Omega_{\lambda}] =\pi_*([Y_{\lambda}]).
\]
\end{lem}
\begin{proof}
When $X$ is  a point,  it is well-known that the degeneracy loci $\Omega_{\lambda}$ has at worst rational singularities \cite[Section 2.2]{BrionLecturenotes}. In general,  the claim holds by a standard argument (see for example,  
\cite[Proof of Theorem 3]{FultonLascoux}). Thus $\pi|_{Y_{\lambda}}:Y_{\lambda} \to \Omega_{\lambda}$ is a rational resolution, \textit{i.e.} the higher direct images of the structure sheaf of $Y_{\lambda}$ vanish. Therefore it follows that in connective $K$-theory, the pushforward of $[Y_{\lambda}]$ is $[\Omega_{\lambda}]$ (see \cite{DaiLevine} or \cite[Lemma 2.2]{Hudson}). 
\end{proof}
Next we describe the class $[Y_{\lambda}]$ as a product of Chern classes.
\begin{lem}\label{A Y_r prod}
For each $j=1,\dots,r$, let $\alpha_j:=c_{\lambda_j+d-j}((D_j/D_{j-1})^{\vee}\otimes E/F^{\lambda_j+d-j})$. The class of $Y_{\lambda}$ in $\CK^*(\Fl_r(U))$ is given by
\begin{equation}
[Y_{\lambda}] =\alpha_1 \cdots \alpha_r.
\end{equation}
\end{lem}
\begin{proof}
Over $Y_{j-1}$, there is a bundle map $D_j/D_{j-1} \to E/U$ and one can see that $Y_j$ is the locus where this bundle map has rank $0$. Since $Y_j$ is smooth of codimension $\lambda_j+d-j$ in $Y_{j-1}$, it follows from Lemma \ref{lemG-B} that the fundamental class of $Y_{j}$ in $\CK^*(Y_{j-1})$ is given by
\begin{equation}\label{formulaEuler}
\iota_{j*}(1)=(\iota_{1}\cdots \iota_{j-1})^*\alpha_j.
\end{equation}
A repeated application of (\ref{formulaEuler}) together with the projection formula implies the claim. For example, if $\lambda=(\lambda_1,\lambda_2)$, then $[Y_{\lambda}] = (\iota_1\iota_2)_* (1) = \iota_{1*} \iota_1^*\alpha_2= \alpha_1\alpha_2$.
\end{proof}
\subsection{Pushforward formula}\label{sec:pushlemA}
\begin{lem}\label{rth first} 
For each $j=1,\dots,r$, let $\tau_j:=c_1((D_j/D_{j-1})^{\vee})$. For each $m\in \ZZ$ and $\ell=0,\dots, n$, let
\[
\scA_{m}^{(\ell)}:=\scS_{m}(U^{\vee}-(E/F^{\ell})^{\vee}).
\]
Then, for each integer $s\geq 0$, we have
\begin{equation}\label{eq:push A}
\pi_{j*}\big(\tau_j^s\alpha_j\big)=\sum_{p= 0}^{j-1} c_p(D_{j-1}) \sum_{q=0}^p\binom{p}{q} \beta^q  \scA_{\lambda_j+s-p+q}^{(\lambda_j+d-j)}.
\end{equation}
\end{lem}
\begin{proof}
We apply Proposition \ref{push of tensor} to $\PP^*((U/D_{i-1})^{\vee})=\PP(U/D_{i-1})$ with $\calQ=(D_i/D_{i-1})^{\vee}$ as the universal quotient bundle. Then the left hand side is equal to 
\begin{eqnarray*}
\scS_{\lambda_j+s}((U/D_{j-1} - E/F^{\lambda_j+d-j})^{\vee})
&=&\scS_{\lambda_j+s}((U- E/F^{\lambda_j+d-j})^{\vee}-D_{j-1}^{\vee}).
\end{eqnarray*}
Now the claim follows from (\ref{formula0}).
\end{proof}
\begin{rem}\label{rem d}
Extend the tower (\ref{P tower}) by $\pi_j: \PP(U/D_{j-1}) \to \PP(U/D_{j-2})$ for $j=r+1,\dots,d$ and let $D_j/D_{j-1}$ be the tautological line bundle of $\PP(U/D_{j-1})$. With the same notation for $\tau_j$ and $\alpha_j$ Lemma \ref{rth first} still holds for $j=r+1,\dots,d$. With this extension, we find that if $\lambda_j=0$, then $\pi_{j*}(\alpha_j)=1$. Indeed, by (\ref{eq:push A}) and the fact that $\scA_m^{(\ell)}=(-\beta)^m$ for all $m\leq 0$, we have
\begin{eqnarray*}
\pi_{i*}(\alpha_i)&=&\sum_{p= 0}^{i-1} c_p(D_{i-1}) \sum_{q=0}^p\binom{p}{q} \beta^q (-\beta)^{-p+q}
=\sum_{p=0}^{{i}-1} c_p(D_{i-1})  (-\beta)^p\sum_{q=0}^p\binom{p}{q}(-1)^q.
\end{eqnarray*}
Since $\sum_{q=0}^p\binom{p}{q}(-1)^q=0$ for all $p>0$, we have $\pi_{i*}(\alpha_i)=c_0(D_{i-1})=1$. Therefore we can conclude that $[Y_{\lambda}] = \pi_{r+1*}\cdots \pi_{d*}(\alpha_1\cdots \alpha_d)$. 
\end{rem}
\subsection{How to compute $[\Omega_{\lambda}]$ in an example}
By the pushforward formula (\ref{eq:push A}),  we can calculate the class $[\Omega_\lambda]$ in $\CK^*(\Gr_d(E))$. For example, let us consider the case $r=2$. The class $[\Omega_{(\lambda_1,\lambda_2)}]$ is given by 
\[
\pi_*([Y_{(\lambda_1,\lambda_2)}])
=\pi_{1*}(\alpha_1\cdot\pi_{2*} (\alpha_2))
=\pi_{1*}\left(\alpha_1\scA_{\lambda_2}^{(k_2)}+\bar \tau_1\alpha_1(\scA_{\lambda_2-1}^{(k_2)}+\beta \scA_{\lambda_2}^{(k_2)})\right),
\]
where we denote $k_i=\lambda_i+d-i$. Using the expansion $\bar \tau_1=-(\tau_1-\beta \tau_1^2+\beta^2 \tau_1^3-\cdots)$, we have 
\[
[\Omega_{(\lambda_1,\lambda_2)}]=\scA_{\lambda_1}^{(k_1)}\scA_{\lambda_2}^{(k_2)}
-\left(\scA_{\lambda_1+1}^{(k_1)}-\beta \scA_{\lambda_1+2}^{(k_1)}+\beta^2\scA_{\lambda_1+3}^{(k_1)}-\cdots\right)
(\scA_{\lambda_2-1}^{(k_2)}+\beta \scA_{\lambda_2}^{(k_2)}).
\]
At this point one notices that the result can be calculated by the following formal Laurent series in $t_1$ and $t_2$:
\[
f( t_1, t_2)= t_1^{\lambda_1}\left( t_2^{\lambda_2}+\bar  t_1( t_2^{\lambda_2-1}+\beta  t_2^{\lambda_2})\right).
\]
We can write it more compactly as
\[
f( t_1, t_2)= t_1^{\lambda_1} t_2^{\lambda_2}(1-\bar  t_1/\bar  t_2)
=\det\begin{pmatrix}
 t_1^{\lambda_1}& t_1^{\lambda_1}\bar  t_1\\
 t_2^{\lambda_1}\bar  t_2^{-1} & t_2^{\lambda_2}
\end{pmatrix}.
\]
Now we obtain the class $[\Omega_{(\lambda_1,\lambda_2)}]$ by replacing $ t_i^m$ with $\scA_m^{(k_i)}$ in $f( t_1, t_2)$. Thus, in principle, we can calculate $[\Omega_\lambda]=\pi_*([Y_{\lambda}])$ by successive applications of the pushforward formula in Lemma \ref{rth first} to the product formula for $[Y_{\lambda}]$ in Lemma \ref{A Y_r prod}. In order to obtain the determinantal formula, we make use of formal Laurent series by generalizing Kazarian's method \cite[Appendix C]{Kazarian}) in next section.
\subsection{Umbral calculus}\label{sec: umbral}
In this section, following Kazarian \cite{Kazarian}, we develop a technique to compute the pushforward along the tower. We call it ``umbral calculus'' borrowing the name from Roman-Rota \cite{RomanRota} and Roman \cite{RomanBook}. Here we mean a method that allows us to compute a complicated series in characteristic classes by replacing them by the corresponding powers of some formal variables.

Consider a graded $\ZZ[\beta]$-algebra $R=\oplus_{i\in \ZZ} R_i $ where $\deg \beta =-1$. Let $t_1,\ldots,t_{d}$ be indeterminates of degree $1$. We use the multi-index notation $t^{\pmb{s}}:=t_1^{s_1}\cdots t_{d}^{s_{d}}$ for $\pmb{s}=(s_1,\dots,s_{d})\in \ZZ^{d}$. A formal Laurent series 
\[
f(t_1,\ldots,t_{d})=\sum_{\pmb{s}\in\ZZ^{d}}a_{\pmb{s}}t^{\pmb{s}}
\]
is {\em homogeneous of degree} $m\in \ZZ$ if $a_{\pmb{s}}$ is zero unless $a_{\pmb{s}}\in R_{m-|\pmb{s}|}$ with $|\pmb{s}|=\sum_{i=1}^{d} s_i$. Let $\supp f = \{\pmb{s} \in \ZZ^{d} \ |\ a_{\pmb{s}}\not=0\}$. 
A series $f(t_1,\ldots,t_{d})$ is a {\em power series} if $\supp f  \subset (\ZZ_{\geq 0})^{d}$ . Let $R[[t_1,\ldots,t_r]]_{m}$ denote the set of all power series in $t_1,\dots, t_{d}$ of degree $m\in \ZZ$. We define
\[
R[[t_1,\ldots,t_{d}]]_{\gr}:=\bigoplus_{m\in \ZZ}R[[t_1,\ldots,t_{d}]]_{m}.
\]
Then $R[[t_1,\ldots,t_{d}]]_{\gr}$ is a graded $\ZZ[\beta]$-algebra and we call it \emph{the ring of graded formal power series}.
\begin{defn}\label{def fls} 
For each $m \in \ZZ$, define $\calL^{R}_m$ to be the space of all formal Laurent series of homogeneous degree $m$ such that there exists $\pmb{n}\in \ZZ^{d}$ for which $\pmb{n} + \supp f$ is contained in the cone in $\ZZ^{d}$ defined by
\[
s_1\geq0, \; s_1+s_2\geq 0, \;\cdots, \; s_1+\cdots + s_{d} \geq 0.
\]
Then $\calL^{R}:=\bigoplus_{m\in \ZZ} \calL^{R}_m$ is a graded ring over $R$ with the obvious product. For each $i=1,\dots, {d}$, let $\calL^{R,i}$ be the $R$-subring of $\calL^R$ consisting of series that do not contain any negative powers of $ t_1,\dots, t_{i-1}$.  In particular, $\calL^{R,1}=\calL^{R}$. 
\end{defn}
In the rest of this section, we set $R:=\CK^*(\Gr_d(E))$.
\begin{defn}
Define a graded $R$-module homomorphism $\phi_1: \calL^{R} \to \CK^*(\Gr_d(E))$ by 
\[
\phi_j( t_1^{s_1}\cdots  t_{d}^{s_d})= \scA_{s_1}^{(\lambda_1+d-1)} \cdots \scA_{s_d}^{(\lambda_d+d-d)}.
\]  
Similarly, for $j\geq 2$, define a graded $R$-module homomorphism $\phi_{j}: \calL^{R,j} \to \CK^*(\PP(U/D_{j-2}))$ by
\[
\phi_j( t_1^{s_1}\cdots  t_{d}^{s_d})= \tau_1^{s_1}\cdots  \tau_{j-1}^{s_{j-1}}\scA_{s_j}^{(\lambda_j+d-j)} \cdots \scA_{s_d}^{(\lambda_d+d-d)}.
\]
It is known that $\CK^*(\Gr_d(E))$ is bounded above, \textit{i.e.} $\CK^m(\Gr_d(E)) =0$ for all $m>\dim \Gr_d(E)$. Therefore $\scA_m^{(\ell)}$ is zero for all $m$ sufficiently large. This ensures that the above map is well-defined.

Furthermore, we have the commutative diagram
\begin{equation}\label{comd1}
\xymatrix{
\calL^{R,j} \ar[r]\ar[d]_{\phi_j}& \calL^{R,j-1}\ar[d]_{\phi_{j-1}}\\
\CK^*(\PP(U/D_{j-2})) \ar[r]_{\pi_{j-1*}}& \CK^*(\PP(U/D_{j-3}))
}
\end{equation}
where the top horizontal arrow is the obvious inclusion map. 
\end{defn}
\begin{lem}\label{p and phi} 
For every choice of $j=1,\dots,d$ and of an integer $s\geq 0$, we have
\[
\pi_{j*}( \tau_j^s\alpha_j)=\phi_j\left( t_j^{\lambda_j+s}\prod_{i=1}^{j-1}(1-\bar  t_i/\bar  t_j)\right).
\]
\end{lem}
\begin{proof}
Since $c_p(D_{j-1}) = e_p(\bar \tau_1,\dots,\bar \tau_{j-1}) $, the pushforward formula (\ref{eq:push A}) allows us to compute the left hand side as follows:
\begin{eqnarray*}
\pi_{j*}( \tau_j^s\alpha_j)
&=&\sum_{p=0}^{j-1} e_p(\bar \tau_1,\dots,\bar \tau_{j-1}) \sum_{q=0}^p\binom{p}{q} \beta^q  \scA_{\lambda_j+s-p+q}^{(\lambda_j+d-j)} \\
&=&\phi_j\left( \sum_{p=0}^{j-1} e_p(\bar t_1,\dots,\bar t_{j-1}) \sum_{q=0}^p\binom{p}{q} \beta^q  t_j^{\lambda_j+s-p+q}\right)\\  
&=&\phi_j\left(t_j^{\lambda_j+s}\sum_{p=0}^{j-1} e_p(\bar t_1,\dots,\bar t_{j-1})  t_j^{-p}(1+\beta t_j)^p\right)\\    
&=&\phi_j\left(t_j^{\lambda_j+s}\prod_{i=1}^{j-1}\left( 1-\bar t_i/\bar t_j\right)\right).   
\end{eqnarray*}
\end{proof}
\begin{prop}\label{prop1A}
We have 
\[
\pi_{1*}\cdots \pi_{d*}(\alpha_1\cdots\alpha_d)=\phi_1\left( t_1^{\lambda_1}\cdots t_{d}^{\lambda_{d}} \prod_{1\leq i<j\leq {d}}(1-\bar  t_i/\bar  t_j)\right).
\]
\end{prop}
\begin{proof}
By the commutativity of the diagram (\ref{comd1}) and by Lemma \ref{p and phi},  we can compute the left hand side as follows:
\begin{eqnarray*}
&&\pi_{1*}\cdots \pi_{d*}(\alpha_1\cdots\alpha_d)\\
&=&\pi_{1*}\cdots \pi_{d-1*}\left(\alpha_1\cdots\alpha_{d-1}\phi_{d}\left( t_{d}^{\lambda_{d}}\prod_{s=1}^{{d}-1}( 1-\bar t_s/\bar t_{d})\right)\right)\\
&=&\pi_{1*}\cdots \pi_{d-2*}\left(\alpha_1\cdots\alpha_{d-2}\phi_{{d}-1}\left( t_{{d}-1}^{\lambda_{{d}-1}}\prod_{s=1}^{{d}-2}( 1-\bar t_s/\bar t_{{d}-1})\cdot t_{d}^{\lambda_{d}}\prod_{s=1}^{{d}-1}( 1-\bar t_s/\bar t_{d})\right)\right)\\
&=&\cdots = \phi_1\left( t_1^{\lambda_1}\cdots t_{d}^{\lambda_{d}}
{\prod_{1\leq i<j\leq {d} }( 1-\bar t_i/\bar t_{j})}\right).
\end{eqnarray*}
\end{proof}
\subsection{Main Theorem for type A}\label{sec: proof A}
For each nonnegative integer $k$ and an integer $m$, let $\displaystyle\binom{m}{k}$ be the generalized binomial coefficient, \textit{i.e.} $\displaystyle\binom{m}{k} = \frac{m(m-1)\cdots (m-k+1)}{k!}$ or equivalently $(1 +  x)^m = \sum_{k=0}^{\infty} \displaystyle\binom{m}{k} x^k$.
\begin{thm}\label{MainA}
For each $\lambda \in \calP_d(n)$, let $\Omega_{\lambda}$ be its associated degeneracy locus in $\Gr_d(E)$. Let 
\begin{equation}\label{def scA}
\scA_{m}^{(\ell)}:=\scS_{m}(U^{\vee}-(E/F^{\ell})^{\vee}) \ \ \ \in \CK^*(\Gr_d(E)).
\end{equation} 
Then the class associated to $\Omega_{\lambda}$ in $\CK^*(\Gr_d(E))$ is given by
\begin{equation}\label{type A det formula}
[\Omega_{\lambda}] = \det\left(  \sum_{s=0}^{\infty} \binom{i-j}{s}\beta^s \scA_{\lambda_i+j-i+s}^{(\lambda_i+d-i)} \right)_{1\leq i,j \leq d}.
\end{equation}
\end{thm}
\begin{proof}
By the combined application of Lemma \ref{prod chern}, Lemma \ref{A Y_r prod}, Remark \ref{rem d}, and Proposition \ref{prop1A}, we have
\begin{equation}\label{eq1183}
[\Omega_\lambda]=\pi_*([Y_{\lambda}])=\pi_{1*}\cdots \pi_{d*}(\alpha_1\cdots\alpha_d)=\phi_1\left( t_1^{\lambda_1}\cdots t_{d}^{\lambda_{d}}
{\prod_{1\leq i<j\leq d }( 1-\bar t_i/\bar t_{j})}\right).
\end{equation}
On the other hand the Vandermonde determinantal formula yields
\[
 t_1^{\lambda_1}\cdots t_{d}^{\lambda_{d}} \prod_{1\leq i<j\leq {d}}(1-\bar  t_i/\bar  t_j)=\det\left( t_i^{\lambda_i}\bar t_i^{\;j-i}\right)_{1\leq i,j\leq {d}}, 
\]
and each entry of the determinant can be computed as
\[
t_i^{\lambda_i}\bar t_i^{j-i} =(-1)^{j-i} t_i^{\lambda_i+j-i}\left(1+\beta  t_i\right)^{i-j}=(-1)^{j-i}\sum_{k=0}^{\infty} \binom{i-j}{k} \beta^k t_i^{\lambda_i+j-i+k}.
\]
Thus, the right hand side of (\ref{eq1183}) equals to
\[
\det\left(\sum_{k=0}^\infty (-1)^{j-i}\binom{i-j}{k}\beta^k\scA_{\lambda_i+j-i+k}^{(\lambda_i+d-i)} \right)_{1\leq i, j \leq {d}}.
\]
Finally, the sign $(-1)^{j-i}$ can be removed from each entry by performing elementary operations on the determinant. 
\end{proof}
\subsection{A determinantal formula of Grothendieck polynomials}
Fix a positive integer $d$. Let $z=(z_1,\ldots,z_d)$ be a sequence of $d$ variables, and $b=(b_1,b_2,\ldots)$ an infinite sequence of variables. Let $\mathcal{P}_d$ denote the set of all partitions with at most $d$ parts, \textit{i.e.} all non-increasing sequences of $d$ nonnegative integers. Let $\lambda=(\lambda_1, \dots, \lambda_d)$ be such a partition. The \emph{factorial Grothendieck polynomials} associated to $\lambda$ is defined by 
\begin{equation}\label{eq:ratio}
G_{\lambda}(z|b)=\frac{\det\left([z_i|b]^{\lambda_j+d-j}(1+\beta z_i)^{j-1}\right)_{1\leq i,j\leq d}}{\prod_{1\leq i<j\leq d}(z_i-z_j)},
\end{equation}
where $[z_i|b]^k=(z_i\oplus b_1)(z_i\oplus b_2)\cdots (z_i\oplus b_k)$ for an integer $k\geq 0$.
One can show that $G_{\lambda}(z|b)$ is a symmetric polynomial in the variables $z_1,\ldots,z_d$ with coefficients in $\ZZ[\beta][b_1,b_2,\dots ]$. 
\begin{rem}
In the original paper by Lascoux-Sch\"utzenberger, $G_\lambda(z):=G_{\lambda}(z|b)|_{b=0}$ is defined in terms of isobaric divided difference operators. Buch proved a combinatorial expression in terms of {\em set-valued semistandard tableaux}. A proof of the coincidence of the bi-alternant formula above and the tableaux formula is given in \cite{IkedaShimazaki}. McNamara \cite{McNamara} introduced the factorial Grothendieck polynomials in terms of set-valued semistandard tableaux. One can also show that his definition coincides with (\ref{eq:ratio}).
\end{rem}
These factorial Grothendieck polynomials represent the degeneracy loci classes in the connective $K$-theory of Grassmann bundles in the sense that if $\lambda \in \calP_d(n)$, then $G_{\lambda}(z|b)$ coincides with $[\Omega_{\lambda}]$ if we regard $z_1,\dots,z_d$ as roots of $U^{\vee}$ and $b_i=c_1((F^{i-1}/F^i)^{\vee})$. From this fact, we can derive the following determinantal formula of $G_{\lambda}(z|b)$ from Theorem \ref{MainA}. First we define the polynomials $G_m^{(\ell)}(z|b)$ corresponding to the Segre classes $\scA_m^{(\ell)}$ under the identification above.
\begin{defn}
For each $m\in \ZZ$ and a non-negative integer $\ell$, define the function $G_m^{(\ell)}(z|b)$ by the following generating function.
\[
\sum_{m\in \ZZ} G_m^{(\ell)}(z|b) u^m = \frac{1}{1+\beta u^{-1}} \prod_{i=1}^d\frac{1+\beta z_i}{1-z_i u} \prod_{i=1}^{\ell}\frac{1-b_i u} {1+\beta b_i}.
\]
\end{defn}
Now we have the following theorem.
\begin{thm}\label{thmGrothe}
We have the following determinantal formula of the factorial Grothendieck polynomials:
\begin{equation}\label{G det formula}
G_{\lambda}(z|b)= \det\left(  \sum_{s=0}^{\infty} \binom{i-j}{s}\beta^s G_{\lambda_i+j-i+s}^{(\lambda_i+d-i)} (z|b)\right)_{1\leq i,j \leq d}.
\end{equation}
\end{thm}
We omit the proof since it is parallel to the analogous statements that we describe for symplectic and odd orthogonal Grassman bundles in Section \ref{sec7}, \ref{sec8}, and \ref{sec: GTheta}. 
\begin{rem}\label{remLenart}
Lenart \cite[Theorem 2.1]{Lenart2000} proved a different determinantal formula of (non-factorial) Grothendieck polynomials. Each entry of his formula is a finite sum of complete symmetric functions, while in our formula each upper triangular entry is an infinite sum of Grothendieck polynomials associated to partitions of length one .
\end{rem}
\section{Combinatorics of type B and C}\label{seccombBC}
In this section, we recall the Weyl group of type $\mathrm{B}_{\infty}$ and $\mathrm{C}_\infty$. We also discuss the notations for the $k$-strict partitions and the characteristic indices that will be used to describe the symplectic and odd orthogonal Grassmannian degeneracy loci.
\subsection{Weyl group of $\mathrm{B}_{\infty}$ and $\mathrm{C}_\infty$}
Let $W_{\infty}$ be the infinite hyperoctahedral group which is defined by the generators  $s_i, i=0,1,\ldots$, and the relations 
\begin{equation}\label{coxeter relation}
\begin{array}{c}
s_i^2=e \ \ \ \ (i\geq 0) ,\ \ \ \ \ \ \ s_is_j=s_js_i\ \ \ \ (|i-j|\geq 2),\\
s_0s_1s_0s_1=s_1s_0s_1s_0, \ \  s_is_{i+1}s_i=s_{i+1}s_is_{i+1}\ \ \ \ (i\geq 1).\\
\end{array}
\end{equation}
We identify $W_{\infty}$ with the group of all permutations $w$ of $\ZZ \backslash \{0\}$ such that $w(i)\not= i$ for only finitely many $i\in \ZZ \backslash \{0\}$, and $\overline{w(i)}=w(\bar{i})$ for all $i$. In this context $\bar{i}$ stands for $-i$. The generators, often referred to as \emph{simple reflections}, are identified with the transpositions $s_{0}=(1,\bar{1})$ and $s_{i}=(i+1,i)(\overline{i},\overline{i+1})$ for $i\geq 1$. The \emph{one-line notation} of an element $w\in W_\infty$ is the sequence $w=(w(1)w(2)w(3)\cdots)$.  The \textit{length} of $w\in W_\infty$ is denoted by $\ell(w)$.


For each nonnegative integer $k$, let $W_{(k)}$ be the subgroup of $W_{\infty}$ generated by all $s_i$ with $i\not=k$. Let $W_\infty^{(k)}$ be the set of minimum length coset representatives for $W_{\infty}/W_{(k)}$, it is given by
\[
W_\infty^{(k)} = \{ w \in W_\infty \ |\ \ell(ws_i)>\ell(w) \mbox{ for all } i\not=k\}.
\]
An element of $W_\infty^{(k)}$ is called \emph{$k$-Grassmannian} and it is given by  the following one-line notation:  
\begin{equation}\label{oneline for k grass}
\begin{array}{c}
w=(v_1\cdots v_k | \overline{\zeta_1}\cdots \overline{\zeta_s} u_1u_2\cdots);\\
0<v_1<\cdots <v_k, \ \ \ \overline{\zeta_1}<\cdots< \overline{\zeta_s} <0<u_1<u_2<\cdots.
\end{array}
\end{equation}
We insert a vertical line after $v_k$ to indicate that $w$ is regarded as a $k$-Grassmannian element. For example, $(13|\bar4 256\cdots)$ is a $2$-Grassmannian element in $W_{\infty}$. 

Upon a choice of an integer $n\geq 0$, we let $W_n$ be the subgroup of $W_{\infty}$ generated by $s_0, s_1,\dots, s_{n-1}$. Or, equivalently, it consists of the elements $w \in W_{\infty}$ such that $w(i)=i$ for all $i>n$. We write the one-line notation of $w \in W_n$ as the finite sequence $(w(1)w(2)\cdots w(n))$. We set $W_{n,(k)}:=W_n \cap W_{(k)}  $ and $W_n^{(k)}:=W_n \cap W_\infty^{(k)}$ so that $W_n^{(k)}\cong W_n/W_{n,(k)}$.
\subsection{$k$-strict partitions}
Let $\SP^{k}$ be the set of all $k$-strict partitions, \textit{i.e.} $\lambda\in \SP^{k}$ is an infinite sequence $(\lambda_1,\lambda_2,\cdots)$ of non-increasing nonnegative integers such that all but finitely many $\lambda_i$'s are zero, and such that $\lambda_i>k$ implies $\lambda_i>\lambda_{i+1}$.  Let $\SP_r^{k}$ be the subset of $\SP^{k}$ consisting of all $k$-strict partitions of length at most $r$. If $\lambda \in \SP_r^{k}$, then we write 
$\lambda=(\lambda_1,\dots,\lambda_r)$. 

There is a bijection between $W_{\infty}^{(k)}$ and $\SP^{k}$. The map from $W_{\infty}^{(k)}$ to $ \SP^{k}$ is given as follows. Let $w \in W_\infty^{(k)}$ be an element which can be written with the one-line notation (\ref{oneline for k grass}). Let $\nu=(\nu_1,\nu_2,\dots)$ be a partition given by $\nu_i=\sharp\{p \ |\ v_p>u_i\}$. Then we define a $k$-strict partition $\lambda$ by setting $\lambda_i=\zeta_i+k$ if $1\leq i\leq s$ and  $\lambda_i=\nu_{i-s}$ if  $s+1\leq i$. See Buch--Kresch--Tamvakis \cite{BuchKreschTamvakis1} for details.

We also consider the subset $\SP^k(n)$ of $\SP_{n-k}^{k}$, which consists of all $k$-strict partitions in $\SP_{n-k}^{k}$ such that $\lambda_1\leq n+k$. Then the above bijection can be restricted to $W_n^{(k)}$: 
\[
W_n^{(k)} \cong \SP^k(n).
\] 

\subsection{Characteristic index}\label{sec:chi}
\begin{defn}
For each $k$-Grassmannian element $w=(v_1\cdots v_k | \overline{\zeta_1}\cdots \overline{\zeta_s} u_1u_2\cdots)\in W_{\infty}^{(k)}$, define the associated characteristic index $\chi$ as
\[
\chi=(\chi_1,\chi_2,\dots ):=(\zeta_1-1,\zeta_2-1, \dots, \zeta_s-1, -u_1,-u_2,\dots).
\]
\end{defn}
\begin{rem}
For $k=0$ one has $\chi_i=\lambda_i-1$ for all $i=1,\dots, r$,  where $r$ is the length of $\lambda$. 
\end{rem}
\begin{defn}
Let $\lambda$ be the $k$-strict partition associated to $w \in W_{\infty}^{(k)}$ under the bijection $W_{\infty}^{(k)} \cong \SP^{k}$. Set
\begin{eqnarray*}
C(\lambda)&:=&\{(i,j) \ |\ 1\leq  i < j, \ \ \chi_i + \chi_j \geq 0\}, \\
\gamma_j&:=& \sharp\{ i \ |\ 1\leq  i < j, \ \ \chi_i + \chi_j \geq 0\}\ \ \ \ \mbox{for each $j>0$}.
\end{eqnarray*}
If $\lambda\in \SP^k_r$, we also define
\[
D(\lambda)_r  := \{ (i,j) \in \Delta_r\ |\  \chi_i + \chi_j < 0\} = \Delta_r \backslash C(\lambda),
\]
where $\Delta_r := \{(i,j) \ |\ 1\leq i < j\leq r\}$.
\end{defn}

\begin{rem}\label{chilambda}
In terms of $\lambda$ the characteristic index $\chi$ is given by $\chi_j = \lambda_j - j + \gamma_j-k$. Furthermore, we have $\chi_i+\chi_j \geq 0$ if and only if $\lambda_i+\lambda_j > 2k+ j-i$ (see \cite[Lemma 3.3]{IkedaMatsumura}). 
\end{rem}
\section{Segre classes for oriented Borel--Moore homology}\label{sec:BM}
In this section, we discuss the Segre classes for a regularly embedded subscheme $Z$ of a smooth  variety $X$ in order to apply it to the computation of degeneracy loci classes for isotropic Grassmann bundles associated to symplectic and odd orthogonal vector bundles. The main difference from the previous sections is that we must work with the theory of \emph{oriented Borel--Moore (BM) homology} developed by Levine--Morel \cite[Chapter 5]{LevineMorel}.

Recall that our base field $\bbF$ is algebraically closed and of characteristic zero. Let $\Sm$ be the category of smooth quasi-projective $\bbF$-schemes. Let $\Sch$ be the category of separated schemes of finite type over $\Spec(\bbF)$.  By a scheme (resp. smooth scheme), we mean an object of $\Sch$ (resp. $\Sm$). 
\subsection{Chern class operators for oriented Borel--Moore homology}
We let $\mathbf{Ab}_*$ denote the category of graded abelian groups. An {\it oriented Borel--Moore homology theory\/} on $\Sch$  is given by a covariant functor $A_*$ from $\Sch$ to $\mathbf{Ab}_*$,  together with pullback maps of {\it local complete intersection morphisms \/} (l.c.i. morphisms for short), and an associative commutative graded binary operation $A_*(X)\otimes A_*(Y) \to A_*(X\times Y)$ called the {\it external product\/}. Notice that, in particular, $A_*(\Spec(\bbF))$ is a commutative graded ring. We will not recall the axioms imposed on this data, however we will discuss some implications. 

A morphism of schemes $f: X\to Y$ is an {\it l.c.i morphism\/} if it admits a factorization $f=q\circ i$ where $i: X \to P$ is a regular embedding  and a smooth quasi-projective morphism $q: P \to Y$. For all such morphisms one has a pullback map of the given oriented Borel--Moore homology theory. In particular, if $X$ is an l.c.i. scheme, \textit{i.e.} the structure morphism $p:X \to \Spec(\bbF)$ is an l.c.i morphism, then its {\it fundamental class\/} $1_X\in A_*(X)$ is defined by $1_X:=p^*(1)$, where $1 \in A_0(\Spec(\bbF))$ (\cite[Definition 5.18]{LevineMorel}). 

For each vector bundle $E$ of rank $e$ over a scheme $X$, there are homomorphisms
\begin{equation}
\tilde{c}_i(E): A_*(X)\rightarrow A_{*-i}(X),
\end{equation}
with $0\leq i\leq e,$ and $\tilde{c}_0(E)=\mathrm{id}_{A_*(X)}$, which are called the {\it Chern class operators\/}. Let $\calQ$ denote the universal quotient line bundle of $\pi: \PP^*(E)\rightarrow X$.  Then we have the following operator version of the relation (\ref{rel chern}):
\begin{equation}
\sum_{i=0}^e(-1)^i \tilde{c}_1(\calQ)^{e-i}\circ\pi^* \circ \tilde{c}_i(E)=0. 
\end{equation}
For any oriented Borel--Moore homology $A_*$, there exists an associated formal group law $F_{A_*}(u,v)$ such that 
\begin{equation}\label{eq:extFGL}
F_{A_*}(\tilde{c_1}(L_1),\tilde{c_1}(L_2)) =\tilde{c_1}(L_1\otimes L_2)
\end{equation}
for any line bundles $L_1,L_2$ over $X$ (Remark 5.2.9 \cite{LevineMorel}). This is an extended (operator version of) formal group law axiom. 

Algebraic cobordism $\Omega_*$ is the universal oriented Borel--Moore homology on $\Sch$ (Theorem 7.1.1 \cite{LevineMorel}). The coefficient ring $\Omega_*(\Spec(\bbF))$ is isomorphic to the Lazard ring $\LL$. In \cite{DaiLevine}, Dai-Levine considered the oriented Borel--Moore homology $\CK_*:=\Omega_*\otimes_{\LL}\ZZ[\beta]$ and proved that $\CK_{\dim X}(X) \cong K_0(X)[\beta,\beta^{-1}]_{\dim X}$ for all equidimensional schemes $X$. Here $K_0$ stands for the Grothendieck group of coherent sheaves. Under this identification the fundamental class $1_X$ coincides with the class $[\calO_X]$ of the structure sheaf of $X$. Furthermore, the formal group law $F_{\CK_*}$ is given by $F_{\CK_*}(u,v)=u+v+\beta uv$. Once again, let us stress that in this paper the sign of $\beta$ is opposite from \cite{DaiLevine}.

If $X$ is smooth, then it is possible to compare $\CK_*$ with the connective $K$-theory introduced in \S \ref{sec:preliminaryOnCK}. The relationship is given by $\CK^*(X)=\CK_{\dim X-*}(X)$ (see also Proposition 5.2.1 \cite{LevineMorel}). With this identification, we have $c_i(E)=\tilde{c}_i(E)(1_X)$. Moreover, if $f: Y \to X$ is an l.c.i. morphism, then the pullback $f^*: \CK_*(X) \to \CK_*(Y)$ gives to $\CK_*(Y)$ the structure of a $\CK^*(X)$-module.

By the extended formal group law axiom (\ref{eq:extFGL}), we can reprove Lemma \ref{tensor decomp} for $\CK_*$ as follows.
\begin{lem}\label{lem2}
Let $X$ be a scheme, and consider a vector bundle $E$ of rank $e$ and a line bundle $L$. 
Then in $\CK_*(X)$, we have
\begin{eqnarray*}
\tilde{c}_e(L \otimes E) 
&=& \sum_{p=0}^{e} \sum_{q=0}^p\binom{p}{q}\beta^q\tilde{c}_p(E)\circ \tilde{c}_1(L)^{e-p+q}.
\end{eqnarray*}
\end{lem}
We conclude this section with the following lemma that follows from Lemma 6.6.7 \cite{LevineMorel}.
\begin{lem}\label{lemCKGB} 
Let $X$ be a scheme and $E$ be a vector bundle of rank $e$ over $X$. Suppose that $E$ has a section $s: X \to E$ such that the zero scheme of $s$, denoted by $i:Z\to X$, is a regularly embedded closed subscheme of codimension $e$. We have
\[
\tilde{c}_e(E) = i_*\circ i^*.
\]
In particular, if $X$ is an l.c.i. scheme, we have 
\[
\tilde{c}_e(E)(1_X) = i_*(1_Z).
\]
\end{lem}
\subsection{Segre class operators}
\begin{defn}
Let $X$ be a scheme. For vector bundles $E$ and $F$ over $X$, define the relative Segre class operators $\tscS_m(E-F)$ for $\CK_*(X)$ by
\begin{equation}\label{dfSegX1}
\sum_{m\in \ZZ} \tscS_m(E-F) u^m = \frac{1}{1+\beta u^{-1}} \frac{\tilde{c}(E;\beta)}{\tilde{c}(E;-u)}\frac{\tilde{c}(F;-u)}{\tilde{c}(F;\beta)}.
\end{equation}
In particular, we have
\begin{equation}\label{eq1}
\tscS_m(E-F) =\sum_{p=0}^{\infty} \sum_{q=0}^p \binom{p}{q}\beta^q\tilde{c}_p(F^{\vee})\circ\tscS_{m-p+q}(E).
\end{equation}

\end{defn}
\begin{rem}
If $X$ is a smooth scheme, under the identification $\CK^*(X)=\CK_{\dim X- *}(X)$, we have $\scS_m(E)=\tscS_m(E)(1_X)$ in $\CK_*(X)$. Thus by the definition (\ref{df:Segre}) and Lemma \ref{push is beta}, we have
\begin{equation}\label{dfBMSegre}
\tscS_m(E)(1_X)  = \pi_*\circ \tilde{c}_1(\calQ)^{m+e-1}(1_{P_X}), \ \ \ m+e-1\geq 0
\end{equation}
where $P_X:=\PP^*(E)$ is the dual projective bundle of $E$ with the projection $\pi:P_X \to X$ and $\calQ$ is its universal quotient line bundle. 

Thus,  if $Z \inc X$ is a regular embedding, by the axiom (BM2) of oriented Borel--Moore homology (p.144 \cite{LevineMorel}), the axiom (A4) of Borel--Moore functors (p.18 \cite{LevineMorel}) and the definition of the fundamental classes (p.146 \cite{LevineMorel}), we can pullback (\ref{dfBMSegre}) to $Z$ and we have
\begin{equation}\label{eq2}
\tscS_m(E|_Z)(1_Z)  = (\pi_Z)_*\circ \tilde{c}_1(\calQ_Z)^{m+e-1}(1_{P_Z}), \ \ \ m+e-1\geq 0
\end{equation}
where $E|_Z\to Z$ and $\pi_Z: P_Z\to Z$ are, respectively, the restriction of $E$ and $P_X$ to $Z$ and $\calQ_Z$ is the universal quotient line bundle of $P_Z$.
\end{rem}
\begin{thm}\label{thmlcipush}
Let $X $ be a smooth scheme and $i: Z \inc X$ a regular embedding. Let $E$ be a vector bundle over $X$ and $F$ a vector bundle over $Z$. Let $\pi_Z: \PP^*(E|_Z) \to Z$ be the dual projective bundle of the restriction $E|_Z$ and $\calQ_Z$ its universal quotient line bundle. In $\CK_*(Z)$, we have
\[
\pi_{Z*}\circ \tilde{c}_1(\calQ_Z)^s \circ \tilde{c}_f(\calQ_Z \otimes F^{\vee})(1_{\PP^*(E|_Z)}) = \tscS_{s+f - e+1}(E|_Z - F)(1_Z).
\]
where the pullback of $F$ to $\PP^*(E|_Z)$ is denoted also by $F$. 
\end{thm}
\begin{proof}
By Lemma \ref{lem2}, (\ref{eq2}) and (\ref{eq1}), we have
\begin{eqnarray*}
&&\pi_{Z*}\circ \tilde{c}_1(\calQ_Z)^s \circ\tilde{c}_f(\calQ|_Z \otimes F^{\vee})(1_{\PP^*(E|_Z)})\\
&=& \sum_{p=0}^{f} \sum_{q=0}^p\binom{p}{q}\beta^q\pi_{Z*}\circ \tilde{c}_p(F^{\vee})\circ \tilde{c}_1(\calQ|_Z)^{s+f-p+q}(1_{\PP^*(E|_Z)})\\
&=& \sum_{p=0}^{f} \sum_{q=0}^p\binom{p}{q}\beta^q\tilde{c}_p(F^{\vee}) \circ \pi_{Z*} \circ \tilde{c}_1(\calQ|_Z)^{s+f-p+q}(1_{\PP^*(E|_Z)})\\
&=& \sum_{p=0}^{f} \sum_{q=0}^p\binom{p}{q}\beta^q\tilde{c}_p(F^{\vee})\circ\tscS_{s+f-p+q-e+1}(E|_Z)(1_Z)\\
&=&\tscS_{s+f-e+1}(E|_Z - F)(1_Z).
\end{eqnarray*}
\end{proof}
\section{Pfaffian formula for symplectic Grassmann bundles}\label{section:Pf non max}
In this section we fix nonnegative integers $n$ and $k$ such that $0 \leq k <n$, unless otherwise stated. The case when $k =0$ is the Lagrangian case. 
\subsection{Symplectic degeneracy loci}\label{subsection Theorem B}
Let $E$ be a symplectic vector bundle of rank $2n$ over a smooth variety $X$. Suppose to be given a complete flag  $F^{\bullet}$ of subbundles of $E$
\[
0=F^n\subset F^{n-1} \subset \cdots \subset F^1 \subset F^0 \subset F^{-1} \subset \cdots \subset F^{-n}=E,  
\]
such that  $\rk(F^i)=n-i$ and $(F^i)^{\perp} = F^{-i}$ for all $i$. Let $\SG^k(E) \to X$ be the symplectic Grassmann bundle over $X$ such that the fiber at $x \in X$ is the Grassmannian $\SG^k(E_x)$ of $(n-k)$-dimensional isotropic subspaces of $E_x$. As before, we suppress from the notation the pullback of vector bundles. Let $U$ be the tautological vector bundle over $\SG^k(E)$. 
\begin{defn}\label{dfOmegaC}
Let $\lambda\in \SP^k(n)$ of length $r$ and $\chi$ its type C characteristic index. Define the symplectic degeneracy loci $\Omega_{\lambda} \subset \SG^k(E)$ by
\[
\Omega_{\lambda}^C = \{ (x, U_x)\in \SG^{k}(E) \ |\ \dim (U_x \cap F^{\chi_i}_x) \geq i, \ \ \  i=1,\dots, r\}.
\]
In this section, we write $\Omega_{\lambda}^C$ by $\Omega_{\lambda}$.
\end{defn}
\subsection{Resolution of singularities}
Let $\lambda\in \SP^k(n)$ with characteristic index $\chi$ and length $r$. Consider the $r$-step flag bundle $\Fl_r(U)$ over $\SG^k(E)$ whose fiber at $(x,U_x)$ consists of the flag $(D_{\bullet})_x=\{(D_1)_x \subset \cdots (D_r)_x\}$ of subspaces of $U$ with $\dim (D_i)_x = i$. The flag of tautological bundles of $\Fl_r(U)$ is denoted by $D_1 \subset \cdots \subset D_r$. The bundle $\Fl_r(U)$ can be constructed as the following tower of projective bundles
\begin{equation}\label{C P tower}
\pi: \Fl_r(U)=\PP(U/D_{r-1}) \stackrel{\pi_r}{\longrightarrow} \PP(U/D_{r-2}) \stackrel{\pi_{r-1}}{\longrightarrow} \cdots \stackrel{\pi_3}{\longrightarrow} \PP(U/D_1) \stackrel{\pi_2}{\longrightarrow} \PP(U)  \stackrel{\pi_1}{\longrightarrow} \SG^k(E).
\end{equation}
We regard $D_j/D_{j-1}$ as the tautological line bundle of $\PP(U/D_{j-1})$ where we let $D_0=0$. For each $j=1,\dots,r$, let $\tilde{\tau}_j:=\tilde{c}_1((D_j/D_{j-1})^{\vee})$ be the Chern class operator of $(D_j/D_{j-1})^{\vee}$ on $\CK_*(\PP(U/D_{j-1}))$.
\begin{defn}\label{def:Z_j}
For each $j=1,\dots,r$, we define a subvariety $Z_j$ of $\PP(U/D_{j-1})$ by 
\[
Z_j := \{ (x,U_x, (D_1)_x, \dots, (D_j)_x) \in \PP(U/D_{j-1}) \ |\ (D_i)_x \subset F^{\chi_i}_x, \ i=1,\dots,{j}\}.
\]
We set $Z_0:=\SG^k(E)$ and $Z_{\lambda}:=Z_r$. 

Let $...$ and consider the projection $...$ together with the obvious inclusion $...$.
Let $P_{j-1}:=\pi_j^{-1}(Z_{j-1})$ and consider the projection $\pi_j': P_{j-1} \to Z_{j-1}$ with the obvious inclusion $\iota_j: Z_j \to P_{j-1}$. Set $\varpi_j:=\pi_j'\circ\iota_j$. We have the following commutative diagram
\[
\xymatrix{
\PP(U/D_{j-1}) \ar[r]_{\pi_j} & \PP(U/D_{j-2})\\
P_{j-1} \ar[r]_{\pi_j'}\ar[u] & Z_{j-1}\ar[u]\\
Z_j\ar[u]_{\iota_j}\ar[ru]_{\varpi_j} &
}.
\]
Finally, we consider the composition $\varpi:=\varpi_1\circ\cdots\circ\varpi_r: Z_{\lambda} \to \SG^k(E)$.
\end{defn}
\begin{rem}
The construction given in Definition \ref{def:Z_j} is similar to the one 
used by Kazarian in \cite{Kazarian} for the Lagrangian case. However, in that setting all the $Z_j$'s happen to be smooth, while in our more general setup it is not necessarily the case.
\end{rem}
\begin{lem}\label{Ylem1C}
The variety $Z_{\lambda}$ is irreducible and has at worst rational singularity. Furthermore, $Z_{\lambda}$ is birational to $\Omega_{\lambda}$ through $\varpi$.
\end{lem}
\begin{proof}
Consider the following $r$-step isotropic partial flag bundle $\Fl^{isot}_r(E)$ over $X$: the fiber at $x\in X$ consists of flags $(C_{\bullet})_x$ of isotropic subspaces $(C_1)_x\subset \cdots \subset (C_r)_x \subset E_x$ such that $\dim (C_{j})_x = {j}$. Let $Z$ be the following $\SG^k(\bbF^{2(n-r)})$-bundle over $\Fl^{isot}_r(E)$ 
\[
Z= \{ (x,(C_{\bullet})_x, V_x) \ |\ V_x \in \SG^k((C_r^{\perp}/C_r)_x)\}.
\]
Let $W_{\lambda}$ be the degeneracy locus in $\Fl^{isot}_r(E)$ defined by
\[
W_{\lambda} :=\{ (x, (C_{\bullet})_x) \in \Fl^{isot}_{r}(E) \ |\  \dim ( F^{\chi_i}_x \cap(C_i)_x) \geq i, \ \ \ i=1,\dots,r\}.
\]
Consider the total space $Z|_{W_{\lambda}}$ of the restriction of the bundle $Z$ to $W_{\lambda}$:
\[
Z|_{W_{\lambda}} = \{(x,(C_{\bullet})_x, V_x) \in Z \ |\ \dim ( F^{\chi_i}_x \cap(C_i)_x) \geq i, \ \ \ i=1,\dots,r \}.
\]
Note that the conditions imply $(C_i)_x \subset F^{\chi_i}_x$ for each $i$. We can show that the variety $Z|_{W_{\lambda}}$ is isomorphic to $Z_{\lambda}$. Indeed, recall that 
\[
Z_{\lambda}=\{(x,U_x, (D_{\bullet})_x) \in \Fl_r(U)\ |\ (D_i)_x \subset F^{\chi_i}_x, i=1,\dots, r\}.
\]
The isomorphism $Z|_{W_{\lambda}} \to Z_{\lambda}$ is given by
\[
(x, V_x, (C_{\bullet})_x) \mapsto (x, \widetilde{V_x}, (C_{\bullet})_x),
\]
where the $(n-k)$-dimensional isotropic subspace $\widetilde{V_x} \subset E_x$ is defined as the preimage of $V_x$ under the quotient map  $(C_r^{\perp})_x \to (C_r^{\perp}/C_r)_x$. It follows from a well-known fact about Schubert varieties that the variety $W_{\lambda}$ is irreducible and has at worst rational singularities (\textit{cf.} \cite[Section 8.2.2. Theorem (c), p.274]{Kumar}). Therefore $Z_{\lambda}$ is irreducible and has at worst rational singularity as well. 

For the latter claim, first note that $\varpi$ is the restriction of $\pi$ to $Z_{\lambda}$. Let $\Omega_{\lambda}^{\circ}$ be an open set of $\Omega_{\lambda}$, consisting of points $(x, U_x)\in \SG^{k}(E)$ such that $\dim (F^{\chi_{i}}_x \cap U_x) = {i}$ and $\dim (F^{\chi_{i}+1}_x\cap U_x)= {i}-1$ for all ${i}=1,\dots,r$.  Let $Z_{\lambda}^{\circ}$ be the preimage of $\Omega_{\lambda}^{\circ}$ by $\varpi: Z_{\lambda} \to \Omega_{\lambda}$. Then it is easy to see that $\varpi|_{Z_{\lambda}^{\circ}}: Z_{\lambda}^{\circ} \to \Omega_{\lambda}^{\circ}$ is bijective and in the view of the irreducibility of $Z_{\lambda}$ and $\Omega_{\lambda}$, this implies that $\varpi$ is birational. 
\end{proof}
\begin{lem}\label{lem1C}
The fundamental class of $\Omega_{\lambda}$ in $\CK_*(\SG^k(E))$ is given by 
\[
[\Omega_{\lambda}] = \varpi_*(1_{Z_{\lambda}}).
\]
\end{lem}
\begin{proof}
Hironaka's theorem (see \cite[Appendix]{LevineMorel}) ensures that there is a projective birational map $\varpi': \tilde{Z} \to Z_{\lambda}$ such that $\tilde{Z}$ is smooth. Since $Z_{\lambda}$ has at worst rational singularities by Lemma \ref{Ylem1C}, we have $\varpi'_*(1_{\tilde{Z}})=1_{Z_{\lambda}}$. On the other hand we know that $\Omega_{\lambda}$ has at worst rational singularities (\cite[Section 8.2.2. Theorem (c), p.274]{Kumar}), and therefore $[\Omega_{\lambda}]=\varpi_*\circ\varpi'_*(1_{\tilde{Z}})$. Thus we can conclude that $\varpi_*(1_{Z_{\lambda}})=[\Omega_{\lambda}]$ ({\it cf.} \cite[Lemma 2.2]{Hudson}).
\end{proof}
\begin{lem}\label{Ylem2C}
For  each $j=1,\dots, r$, we have the following.
\begin{itemize}
\item[(1)] Both of the inclusions $Z_j\inc P_{j-1}$ and $P_{j-1}\inc \PP(U/D_{j-1})$ are regular embeddings.
\item[(2)] In $\CK_*(P_{j-1})$, we have
\begin{equation}\label{eq:iota1}
\iota_{j*}(1_{Z_j})=\tilde{c}_{\lambda_j + n - k - j}\big((D_j/D_{j-1})^{\vee}\otimes D_{\gamma_j}^{\perp}/F^{\chi_j}\big)(1_{P_{j-1}}).
\end{equation}
\end{itemize}
\end{lem}
\begin{proof} (1) 
By Lemma \ref{lemG-B} (a) and induction on $j$, it suffices to prove the assertions that $Z_j$ is the zero-scheme defined by a section of a bundle over $P_{j-1}$ and that its codimension agrees with the rank of the bundle. 
In fact, if $j=1$, the regularity of $Z_1\inc P_0$ follows from the claims since $P_0=\PP(U)$. If $j\geq 2$, we can assume that $Z_{j-1}\inc P_{j-2}$ and  $P_{j-2}\inc \PP(U/D_{j-2})$ are regular embeddings so that their composition $Z_{j-1}\inc \PP(U/D_{j-2})$ is also a regular embedding. It follows that its pullback $P_{j-1} \inc \PP(U/D_{j-1})$ along $\pi_j$ is a regular embedding. This implies that $P_{j-1}$ is Cohen--Macaulay by Lemma \ref{lemG-B} (a) since $\PP(U/D_{j-1})$ is smooth and hence Cohen--Macaulay. Thus the above assertions imply that $Z_j\inc P_{j-1} $ is a regular embedding again by Lemma \ref{lemG-B} (a).
 
In order to prove the assertions, we first compute the codimension. For $1\leq j \leq r$, Lemma \ref{Ylem1C} applied to the $k$-strict partition $(\lambda_1,\dots, \lambda_j)$ implies that $Z_j$ is birational to the degeneracy loci $\Omega_{(\lambda_1,\dots,\lambda_j)}$ in $\SG^k(E)$. Since the codimension of $\Omega_{(\lambda_1,\dots,\lambda_j)}$ in $\SG^k(E)$ is $\sum_{i=1}^j \lambda_i$ and the dimension of the fiber of the projection $\pi_j \circ \cdots \circ \pi_1$ is $\sum_{i=1}^j (n-k-j)$, the codimension of $Z_j$ in $\PP(U/D_{j-1})$ is $\sum_{i=1}^{j} (\lambda_i  + n-k-i)$. This also implies that $P_{j-1}$ has codimension $\sum_{i=1}^{j-1} (\lambda_i  + n-k-i)$ in $\PP(U/D_{j-1})$ since $P_{j-1}$ is the preimage of $Z_j$ along $\pi_j$. Thus the codimension of $Z_j$ in $P_{j-1}$ is $\lambda_j + n - k - j$. 

Next we describe $Z_j$ as a zero scheme of a bundle. Over $P_{j-1}$, we have $D_i \subset F^{\chi_i} \subset F^{\chi_j}$ for all $i<j$. Furthermore, if $\chi_i + \chi_j \geq 0$ with $i<j$, then $F^{\chi_j} \subset (F^{\chi_i})^{\perp} \subset D_i^{\perp}$. Thus we have the bundle map $D_j/D_{j-1} \to D_{\gamma_j}^{\perp}/D_{j-1} \to D_{\gamma_j}^{\perp}/F^{\chi_j}$ over $P_{j-1}$, where we recall that $\gamma_j=\sharp\{ i \ |\ 1\leq i<j, \chi_i + \chi_j \geq 0\}$. Now, observe that $Z_j$ is the locus where this bundle map has rank $0$, \textit{i.e.} $Z_j$ is the zero scheme of the corresponding section of $(D_j/D_{j-1})^{\vee}\otimes D_{\gamma_j}^{\perp}/F^{\chi_j}$. 

Finally by the relation between $\chi$ and $\lambda$ mentioned in Remark \ref{chilambda}, we find that
\begin{equation}
\rk(D_{\gamma_j}^{\perp}/F^{\chi_j}) =2n - \gamma_j - (n - \chi_j) = \lambda_j + n - k - j = \codim(Z_j \subset P_{j-1}).\label{eq:rk=codim}
\end{equation}
(2) From (1), we know that $P_{j-1}$ is Cohen--Macaulay. Since, as we saw in the proof of (1), $Z_j$ is the zero scheme of a section of the bundle $(D_j/D_{j-1})^{\vee}\otimes D_{\gamma_j}^{\perp}/F^{\chi_j}$ over $P_{j-1}$, Lemma \ref{lemCKGB} together with (\ref{eq:rk=codim}) implies (\ref{eq:iota1}).
\end{proof}
\subsection{Pushforward formula and umbral calculus}
\begin{defn}\label{dfCclass}
For each $m\in \ZZ$ and $-n\leq \ell \leq n$, define 
\[
\tscC_m^{(\ell)} := \tscS_{m}(U^{\vee}-(E/F^{\ell})^{\vee}).
\]
In $\CK^*(\SG^k(E))$, we denote $\scC_m^{(\ell)} := \tscC_m^{(\ell)} (1_{\SG^k(E)})$.
\end{defn}
\begin{lem}\label{lempushC}
In $\CK_*(Z_{j-1})$, we have
\[
\varpi_{j*}\circ  \tilde{\tau}_j^s(1_{Z_j}) = \sum_{p= 0}^{\infty} \sum_{q=0}^p\binom{p}{q}\beta^q\tilde{c}_p(D_{j-1} - D_{\gamma_j}^{\vee})\circ\tscC_{\lambda_j+s-p+q}^{(\chi_j)}(1_{Z_{j-1}}),  \ \ \ s\geq 0.
\]
\end{lem}
\begin{proof}
By Lemma \ref{Ylem2C}, we have
\begin{eqnarray*}
\varpi_{j*}\circ  \tilde{\tau}_j^s(1_{Z_j}) 
=\pi'_{j*}\circ\iota_{j*}\circ  \tilde{\tau}_j^s(1_{Z_j}) 
=\pi'_{j*}\circ  \tilde{\tau}_j^s\circ\iota_{j*}(1_{Z_j}) 
=\pi'_{j*}\circ  \tilde{\tau}_j^s\circ\tilde{\alpha}_j(1_{P_{j-1}}),
\end{eqnarray*}
where $\tilde{\alpha}_j:=\tilde{c}_{\lambda_j + n - k - j}((D_j/D_{j-1})^{\vee}\otimes D_{\gamma_j}^{\perp}/F^{\chi_j})$.  On the other hand, by Theorem \ref{thmlcipush} we have
\begin{eqnarray*}
\pi_{j*}'\circ \tilde{\tau}_j^s \circ \tilde{\alpha}_j(1_{P_{j-1}}) 
&=& \tscS_{s+\lambda_j}\big((U/D_{j-1})^{\vee} - (D_{\gamma_j}^{\perp}/F^{\chi_j})^{\vee}\big)(1_{Z_{j-1}})\\
&=& \tscS_{s+\lambda_j}\big(U^{\vee}- (E/F^{\chi_j})^{\vee} - (D_{j-1}- D_{\gamma_j}^{\vee})^{\vee}\big)(1_{Z_{j-1}}),
\end{eqnarray*}
where we have used $D_{\gamma_j}^{\perp} = E - D_{\gamma_j}^{\vee}$. Now the claim follows from (\ref{eq1}).
\end{proof}
Set $R:=\CK^*(\SG^k(E))$ and let $\calL^R$ be the ring of formal Laurent series with indeterminates $t_1,\dots, t_{r}$ defined in Definition \ref{def fls}.
\begin{defn}
Define a graded $R$-module homomorphism $\phi_1: \calL^{R} \to \CK_*(\SG^k(E))$ by 
\[
\phi_1( t_1^{s_1}\cdots  t_r^{s_r})= \tscC_{s_1}^{(\chi_1)} \circ\cdots \circ\tscC_{s_r}^{(\chi_r)}(1_{\SG^k(E)}).
\]  
Similarly, for $j \geq 2$, define a graded $R$-module homomorphism $\phi_{j}: \calL^{R,j} \to \CK_*(Z_{j-1})$ by
\[
\phi_j( t_1^{s_1}\cdots  t_r^{s_r})= \tilde{\tau}_1^{s_1}\circ\cdots  \circ\tilde{\tau}_{j-1}^{s_{j-1}}\circ\tscC_{s_j}^{(\chi_j)} \circ\cdots \circ\tscC_{s_r}^{(\chi_r)}(1_{Z_{j-1}}).
\]
\end{defn}
\begin{rem}\label{remHomotoCoho}
By regarding $\CK^*(\SG^k)=\CK_{\dim \SG^k(E) - *}(\SG^k(E))$, we have
\[
\phi_1( t_1^{s_1}\cdots  t_r^{s_r})= \scC_{s_1}^{(\chi_1)} \circ\cdots \circ\scC_{s_r}^{(\chi_r)}.
\]  
\end{rem}
\begin{lem}\label{lemp-phiC}
We have
\[
\varpi_{j*}\circ  \tilde{\tau}_j^s(1_{Z_j})= \phi_j\left( t_j^{\lambda_j+s}\frac{\prod_{i=1}^{j-1}( 1-\bar t_i/\bar t_j)}{\prod_{i=1}^{\gamma_j}(1 -  t_i/\bar t_j)}\right)
\]
for all $s\geq 0$.
\end{lem}
\begin{proof}
Define $H_m(t_1,\dots, t_{j-1})$ by
\[
\frac{\prod_{i=1}^{j-1}( 1+\bar t_i u)}{\prod_{i=1}^{\gamma_j}(1 +  t_i u)} =  \sum_{m=0}^{\infty} H_m(t_1,\dots, t_{j-1}) u^m
\]
Then by Lemma \ref{lempushC}, we obtain
\begin{eqnarray*}
\varpi_{j*}\circ  \tilde{\tau}_j^s(1_{Z_j})
&=&\sum_{p= 0}^{\infty} \sum_{q=0}^p\binom{p}{q}\beta^q\tilde{c}_p(D_{j-1} - D_{\gamma_j}^{\vee})\circ\tscC_{\lambda_j+s-p+q}^{(\chi_j)}(1_{Z_{j-1}})\\
&=&\phi_j\left(\sum_{p= 0}^{\infty} \sum_{q=0}^p\binom{p}{q}\beta^q
H_p(t_1,\dots, t_{j-1}) t_j^{\lambda_j+s-p+q}
\right)\\
&=&\phi_j\left( t_j^{\lambda_j+s}\sum_{p= 0}^{\infty} 
H_p(t_1,\dots, t_{j-1}) t_j^{-p} \left(\sum_{q=0}^p\binom{p}{q}\beta^q t_j^{q} \right)
\right)\\
&=&\phi_j\left( t_j^{\lambda_j+s}\sum_{p= 0}^{\infty} 
H_p(t_1,\dots, t_{j-1}) (-\bar t_j)^{-p} 
\right).
\end{eqnarray*}
Now the claim follows from the definition of $H_m$.
\end{proof}
\begin{prop}\label{propmainC}
We have
\[
 \varpi_{1*}\circ\cdots\circ \varpi_{r*}(1_{Z_{\lambda}}) = \phi_1 \left( t_1^{\lambda_1}\cdots t_r^{\lambda_r}\frac{\prod_{(i,j)\in \Delta_r}(1- \bar t_i/\bar t_j)}{\prod_{(i,j) \in C(\lambda)} (1- t_i/\bar t_j)}\right).
\]
\end{prop}
\begin{proof}
By repeatedly applying Lemma \ref{lemp-phiC}, we obtain
\begin{eqnarray*}
 \varpi_{1*}\circ\cdots\circ \varpi_{r*}(1_{Z_{\lambda}})
&=&\varpi_{1*}\cdots \varpi_{r-1*}\phi_{r}\left( t_r^{\lambda_r}\frac{\prod_{i=1}^{r-1}( 1-\bar t_i/\bar t_r)}{\prod_{i=1}^{\gamma_r}(1 -  t_i/\bar t_r)}\right)\\
&=&\varpi_{1*}\cdots \varpi_{r-2*}\phi_{r-1}\left( t_{r-1}^{\lambda_{r-1}} t_r^{\lambda_r}\frac{\prod_{i=1}^{r-2}( 1-\bar t_i/\bar t_{r-1})}{\prod_{i=1}^{\gamma_{r-1}}(1 -  t_i/\bar t_{r-1})}\frac{\prod_{i=1}^{r-1}( 1-\bar t_i/\bar t_r)}{\prod_{i=1}^{\gamma_r}(1 -  t_i/\bar t_r)}\right)\\
&=&\cdots = \phi_1\left( t_1^{\lambda_1}\cdots t_r^{\lambda_r}\frac{\prod_{(i,j)\in \Delta_r }( 1-\bar t_i/\bar t_{j})}{\prod_{(i,j)\in C(\lambda) }( 1- t_i/\bar t_j)}\right).
\end{eqnarray*}
\end{proof}
\subsection{Extension of Schur-Pfaffian}\label{secSPf}
Let $r$ be a nonnegative integer. Let us first introduce some combinatorial quantities: for a subset $I$ of $\Delta_r$, define 
\[
a_i^I:=\sharp\{j \ |\ (i,j)\in I\},  \ \ \   c_j^I:=\sharp\{i \ |\ (i,j)\in I\}, \ \ \ \mbox{and}\ \ \ d_i^I:=a_i^I-c_i^I.
\]
Let $r'$ be the smallest even integer greater than or equal to $r$. Define the following rational function of $t_i$ and $t_j$:
\begin{eqnarray*}
F_{i,j}^I(t)&:=&\frac{1}{(1 + \beta  t_i)^{{r'}-i-c_i^I-1}} \frac{1}{(1 + \beta  t_j)^{{r'}-j-c_j^I}} \frac{1 - \bar  t_i/\bar t_j}{1 -  t_i/\bar  t_j},\\
F_{i}^I(t)&:=&\frac{1}{(1 + \beta  t_i)^{{r'}-i-c_i^I-1}}.
\end{eqnarray*}
\begin{lem}\label{lem4C}
Let $(\lambda_1,\lambda_2, \dots, \lambda_r)$ be a $k$-strict partition of length $r$. Let $r'$ be the smallest even integer greater than or equal to $r$. Assume that $\lambda_{r'}=0$ if $r$ is odd. Then we have
\begin{eqnarray}\label{t=Pf}
t_1^{\lambda_1}\cdots t_r^{\lambda_r} \frac{\prod_{(i,j) \in \Delta_r} (1- \bar t_i/\bar t_j)}{\prod_{(i,j) \in C(\lambda)} (1- t_i/\bar t_j)}
=\sum_{I\subset D(\lambda)_r}\Pf  \left(\Lambda_{i,j}^I(t)\right)_{1 \leq i < j \leq  {r'}},
\end{eqnarray}
where, if $r$ is even,
\begin{eqnarray*}
\Lambda_{i,j}^I(t)&:=&t_i^{\lambda_i+d_i^I}t_j^{\lambda_j+d_j^I}F_{i,j}^I(t), 
\end{eqnarray*}
and if $r$ is odd, 
\begin{eqnarray*}
\Lambda_{i,j}^I(t)&:=&\begin{cases}
t_i^{\lambda_i+d_i^I}t_j^{\lambda_j+d_j^I}F_{i,j}^I(t) & 1\leq i<j\leq r\\
t_i^{\lambda_i+d_i^I}F_i^I(t) & 1 \leq i <j=r+1.
\end{cases}
\end{eqnarray*}
\end{lem}
\begin{proof}
Since $D(\lambda)_r = \Delta_r\backslash C(\lambda)$, we have
\[
\frac{\prod_{(i,j) \in \Delta_r} (1- \bar t_i/\bar t_j)}{\prod_{(i,j) \in C(\lambda)} (1- t_i/\bar t_j)} 
=\sum_{I\subset D(\lambda)_r} \prod_{(i,j)\in I}(-t_i/\bar t_j) \prod_{(i,j)\in \Delta_r} \frac{1- \bar t_i/\bar t_j}{1- t_i/\bar t_j}.
\]
Note that $\prod_{(i,j)\in I}  t_i(-\bar t_j)^{-1}  = \prod_{i=1}^{r} t_i^{d_i^I}(1+\beta t_i)^{c_i^I}$ and
\begin{eqnarray}\label{eq1533}
\frac{\bar  t_i - \bar  t_j}{\bar  t_i\oplus \bar  t_j}
&=&-(1+\beta t_i)\frac{1-\bar  t_i/\bar  t_j}{1-t_i/\bar  t_j}.
\end{eqnarray}
By Lemma 2.4 \cite{IkedaNaruse}, we have 
\[
\Pf\left(\dfrac{\bar  t_i - \bar  t_j}{\bar  t_i\oplus \bar  t_j}\right)_{1\leq i<j\leq m} = \prod_{1\leq i<j\leq m} \dfrac{\bar  t_i - \bar  t_j}{\bar  t_i\oplus \bar  t_j},
\]
for any even integer $m$. From this, a direct computation shows that 
\begin{eqnarray}
&& t_1^{\lambda_1}\cdots t_{r'}^{\lambda_{r'}} \left(\prod_{(i,j)\in I} (- t_i/\bar t_j) \right)\left(\prod_{(i,j) \in \Delta_{r'}} \frac{1- \bar t_i/\bar t_j}{1- t_i/\bar t_j}\right)\nonumber\\
&=& \prod_{i=1}^{r'}\frac{(-1)^{{r'}-i}t_i^{\lambda_i+d_i^I}}{(1 + \beta  t_i)^{{r'}-i-c_i^I}} 
\Pf\left( \frac{\bar  t_i - \bar  t_j}{\bar  t_i\oplus \bar  t_j}\right)_{1 \leq i < j \leq  {r'}} \nonumber\\
&=&\Pf \left(  \frac{(-1)^{{r'}-i}t_i^{\lambda_i+d_i^I}}{(1 + \beta  t_i)^{{r'}-i-c_i^I}} 
\frac{(-1)^{{r'}-j}t_j^{\lambda_j+d_j^I}}{(1 + \beta  t_j)^{{r'}-j-c_j^I}} 
\frac{\bar  t_i - \bar  t_j}{\bar  t_i\oplus \bar  t_j}\right)_{1 \leq i < j \leq  {r'}}\nonumber\\
&=&\Pf  \left( (-1)^{i+j+1}\frac{t_i^{\lambda_i+d_i^I}}{(1 + \beta  t_i)^{ {r'}-i-c_i^I-1}} 
\frac{t_j^{\lambda_j+d_j^I}}{(1 + \beta  t_j)^{ {r'}-j-c_j^I}} 
\frac{1 - \bar  t_i/\bar t_j}{1 -  t_i/\bar  t_j}\right)_{1 \leq i < j \leq  {r'}}.\label{eqt=Pf}
\end{eqnarray}
Since multiplying each entry by $(-1)^{i+j+1}$ does not modify change the Pfaffian,  one obtains the formula for the case when $r$ is even, {\it i.e.} $r=r'$. 

If $r$ is odd, {\it i.e.} $r'=r+1$, we can observe  that
\[
t_1^{\lambda_1}\cdots t_r^{\lambda_r}\frac{\prod_{(i,j)\in \Delta_r}(1- \bar t_i/\bar t_j)}{\prod_{(i,j) \in C(\lambda)} (1- t_i/\bar t_j)}
= \left.t_1^{\lambda_1}\cdots t_{{r+1}}^{\lambda_{{r+1}}}\frac{\prod_{(i,j)\in \Delta_{{r+1}}}(1- \bar t_i/\bar t_j)}{\prod_{(i,j) \in C(\lambda)} (1- t_i/\bar t_j)}\right|_{t_{{r+1}}=-\beta^{-1}}.
\]
By applying (\ref{eqt=Pf}) to the right hand side, we obtain
\[
t_1^{\lambda_1}\cdots t_r^{\lambda_r}\frac{\prod_{(i,j)\in \Delta_r}(1- \bar t_i/\bar t_j)}{\prod_{(i,j) \in C(\lambda)} (1- t_i/\bar t_j)}=\sum_{ I \subset D(\lambda)_{r+1}} \Pf\left[   \left. t_i^{\lambda_i+d_i^I}t_j^{\lambda_j+d_j^I}F_{i,j}^I(t)\right]_{1\leq i<j\leq {r+1}}\right|_{t_{{r+1}}=-\beta^{-1}}
\]
For any subset $I \subset D(\lambda)_{{r+1}}$, we have $F_{i,{r+1}}^I(t)|_{t_{{r+1}}=-\beta^{-1}}=0$ if $(i,{r+1}) \in I$. Thus only the terms with $I \subset D(\lambda)_r$ survive. Furthermore, if $(i,{{r+1}}) \not\in I$, then $\lambda_{{r+1}}+d_{{r+1}}^I=0$ and $F_{i,{r+1}}^I(t)|_{t_{{r+1}}=-\beta^{-1}}=F_{i}^I(t)$. This proves the odd case.
\end{proof}

For the reader's convenience, we provide the specialization of Lemma \ref{lem4C} to the Lagrangian case $k=0$.
\begin{lem}\label{lem4CLag}
 We have
\begin{eqnarray}\label{t=Pf}
t_1^{\lambda_1}\cdots t_r^{\lambda_r} \prod_{(i,j) \in \Delta_r} \frac{1- \bar t_i/\bar t_j}{1- t_i/\bar t_j}
=\Pf  \left(\Lambda_{i,j}(t)\right)_{1 \leq i < j \leq  {r'}},
\end{eqnarray}
where, if $r$ is even,
\begin{eqnarray*}
\Lambda_{i,j}(t)&:=&\frac{t_i^{\lambda_i}}{(1 + \beta  t_i)^{{r}-i-1}} \frac{t_j^{\lambda_j}}{(1 + \beta  t_j)^{{r}-j}} \frac{1 - \bar  t_i/\bar t_j}{1 -  t_i/\bar  t_j}, 
\end{eqnarray*}
and if $r$ is odd, 
\begin{eqnarray*}
\Lambda_{i,j}(t)&:=&\begin{cases}
\dfrac{t_i^{\lambda_i}}{(1 + \beta  t_i)^{{r}-i}} \dfrac{t_j^{\lambda_j}}{(1 + \beta  t_j)^{{r+1}-j}} \dfrac{1 - \bar  t_i/\bar t_j}{1 -  t_i/\bar  t_j} & 1\leq i<j\leq r\\
\dfrac{t_i^{\lambda_i}}{(1 + \beta  t_i)^{{r}-i}} & 1 \leq i <j=r+1.
\end{cases}
\end{eqnarray*}
\end{lem}
\subsection{Main theorem for type C}\label{secmainC}
Consider the expansion of $F_{i,j}^I(t)$ and $F_{i}^I(t)$ as Laurent series in $\calL^{R}$:
\begin{eqnarray*}
F_{i,j}^I(t)= \sum_{p\geq 0, \atop{p+q\geq 0}} \ff_{pq}^{ij,I} t_i^pt_j^q, \ \ \ \ \ 
F_i^I(t)= \sum_{p \geq 0} \ff_p^{i,I} t_i^p.
\end{eqnarray*}
\begin{thm}\label{mainC}
Let $\lambda$ be a $k$-strict partition in $\SP^k(n)$ of length $r$ and $\chi$ its characteristic index. In $\CK^*(\SG^k(E))$, the fundamental class of the degeneracy locus $\Omega_{\lambda}$ is given by 
\[
[\Omega_{\lambda}] = \sum_{I \subset D(\lambda)_r} \Pf\left(\sum_{p,q\in \ZZ \atop{p\geq 0, p+q\geq 0}} \ff_{pq}^{ij,I} \scC_{\lambda_i+d_i^I+p}^{(\chi_i)}\scC_{\lambda_j+d_j^I+q}^{(\chi_j)}\right)_{1\leq i<j\leq r' },
\]
where $r'$ is the smallest integer greater than or equal to $r$ and we set $\scC_{-i}^{(-n-1)}:=(-\beta)^i$ for $i\leq 0$. In particular, if $r$ is odd, then $(i,r+1)$-entry of the Pfaffian reduces to $\sum_{p \geq 0} \ff_p^{i,I} \scC_{\lambda_i+d_i^I+p}^{(\chi_i)}$.
\end{thm}
\begin{proof}
By Lemma \ref{lem1C} and Proposition \ref{propmainC}, we have 
\begin{equation*}
[\Omega_{\lambda}] = \phi_1 \left( t_1^{\lambda_1}\cdots t_r^{\lambda_r}\frac{\prod_{(i,j)\in \Delta_r}(1- \bar t_i/\bar t_j)}{\prod_{(i,j) \in C(\lambda)} (1- t_i/\bar t_j)}\right).
\end{equation*}
Then Lemma \ref{lem4C} implies the formula. See also Remark \ref{remHomotoCoho} for the cohomological notation.
\end{proof}
As a special case of Theorem \ref{mainC}, we obtain the Pfaffian formula of the degeneracy loci classes for the Lagrangian Grassmannian $\LG(E)=\SG^0(E)$. 
\begin{thm}\label{mainLag}
Let $\lambda \in \SP(n)$ be of length $r$. Let $\ff_{pq}^{ij}=\ff_{pq}^{ij,\varnothing}$ and $\ff_{p}^{i}=\ff_{p}^{i,\varnothing}$.  In $\CK^*(\LG(E))$, The fundamental class of the degeneracy locus $\Omega_{\lambda}$ is given as follows: if $r$ is even,
\[
[\Omega_{\lambda}]= \Pf\left(\sum_{p,q\in \ZZ \atop{p\geq 0, p+q\geq 0}} \ff_{pq}^{ij} \scC_{\lambda_i+p}^{(\lambda_i-1)}\scC_{\lambda_j+q}^{(\lambda_j-1)}\right)_{1\leq i<j\leq r},
\]
if $r$ is odd, 
\[
[\Omega_{\lambda}] =\sum_{s=1}^{r}(-1)^{s+r}\left(\sum_{p\in \ZZ\atop{p \geq 0}} \ff_p^{s} \scC_{\lambda_s+p}^{(\lambda_s-1)}\right)\Pf\left(\sum_{p,q\in \ZZ \atop{p\geq 0, p+q\geq 0}} \ff_{pq}^{ij} \scC_{\lambda_i+p}^{(\lambda_i-1)}\scC_{\lambda_j+q}^{(\lambda_j-1)}\right)_{1\leq i<j\leq r \atop{i,j\not=s}}.
\]
\end{thm}
\begin{rem}
In the theorem, in order to deal with the odd case, we have applied the cofactor expansion of Pfaffian to the last column of the Pfaffian.
\end{rem}
\section{Pfaffian formula for odd orthogonal Grassmann bundles}\label{sec:typeB}
In this section, we fix nonnegative integers $n$ and $k$ such that $0 \leq k <n$, unless otherwise stated. 
\subsection{Odd orthogonal degeneracy loci}
Let $E$ be a vector bundle over $X$ of rank $2n+1$ with a non-degenerate symmetric form. We assume to be given a complete flag of $E$
\begin{equation}\label{eq:flagB}
F^{n-1} \subset \cdots \subset F^1 \subset F^0 \subset (F^0)^{\perp} \subset F^{-1} \subset \cdots \subset F^{-n}=E,
\end{equation}
such that $\rk(F^i)=n-i$ for $i\geq 0$ and $(F^{i})^{\perp}=F^{-i}$ for all $i \geq 1$. Note that $\rk(F^{i}) = n-i+1$ for $i \leq -1$. Let $\OG^k(E)$ be the Grassmannian of isotropic subbundles of rank $n-k$ in $E$. We denote an element of  $\OG^k(E)$ by $(x,U_x)$ where $x\in X$ and $U_x$ is an $(n-k)$-dimensional isotropic subspace of $E_x$.
\begin{defn}\label{dfOmegaB}
Let $\lambda \in \SP^k(n)$ of length $r$ with $\chi$ its characteristic index. The associated degeneracy loci $\Omega_{\lambda}^B$ is defined by
\[
\Omega_{\lambda}^B = \{(x,U_x)\in \OG^k(n) \ |\ \dim(U_x \cap F^{\chi_i})\geq i, \ i=1,\dots,r\}.
\]
In this section, we write $\Omega_{\lambda}^B$ simply by $\Omega_{\lambda}$.
\end{defn}
\subsection{Quadric bundle}
We describe the degeneracy loci classes of $\OG^{n-1}(E)$, which has a structure of a fibration of quadric hypersurfaces and is denoted by $Q(E)$. In this section we do not assume that $X$ is smooth, it suffices that it is regularly embedded in a smooth variety. Later this will allow us to apply the results of this section to the study of the classes of the resolutions of $\Omega_{\lambda}$.

For each point $x\in X$, the set of all isotropic lines of $E_x$ forms a closed subvariety of $\PP(E_x)$. One can choose homogeneous coordinates $(x_1,\ldots,x_n,y_1,\ldots,y_n,z)$ of $\PP(E_x)$ such that the subvariety is the quadric hypersurface defined by the equation
\[
x_1y_1+\cdots+x_ny_n+z^2=0.
\]
Let $S$ be the tautological line bundle of $Q(E)$. We consider the subvatieties in $\PP(E)$ corresponding the components of the flag (\ref{eq:flagB}). We may assume 
\[
\PP(F^0): y_1=\cdots=y_n=z=0,\quad \PP((F^0)^\perp):  y_1=\cdots=y_n=0,
\]
and for $1\leq i\leq n$,
\[
\PP(F^i): y_1=\cdots=y_n=z=0, \quad x_1=\cdots=x_i=0,\quad \PP(F^{-i}): y_{i+1}=\cdots=y_n=0
\]
in the local coordinates. Since for $0\leq i< n$ the subbundle $F^i$ is isotropic, we have that $\PP(F^i)$ is naturally a subvariety of $Q(E)$. For $-n\leq i< 0$, one sees that the scheme-theoretic intersection $Q(E)\cap\PP(F^i)$ is reduced. We denote these subvarieties of $Q(E)$ by $X_i$:
\begin{equation}\label{Xi}
X_i=\begin{cases}
Q(E)\cap\PP(F^i) & (-n\leq i<0)\\
\PP(F^i) & (0\leq i< n).
\end{cases}
\end{equation}
Note that the scheme-theoretic intersection $Q(E)\cap\PP((F^0)^\perp)$ in $\PP(E)$ is not reduced and the defining ideal is generated by $z^2,y_1,\ldots,y_n$. The corresponding scheme with the reduced structure is $X_0=\PP(F^0)$. The following result on the fundamental class of a non-reduced closed subscheme allows us to calculate $[Q(E)\cap\PP((F^0)^\perp)]$ and hence $[X_i]$ for $0\leq i < n$ as elements in $\CK_*(Q(E))$.
\begin{lem}\label{lem:divisorclass}(\cite[Section 7.2.1]{LevineMorel})
Let $W$ be a smooth scheme and $D$ a smooth prime divisor on $W$. Consider the divisor $E=2D$ and let $|E|$ be the closed subscheme of $W$ defined by $E$. If $L$ is the line bundle corresponding to $D$ and $\iota: D\rightarrow |E|$ is the natural morphism, then  we have
\[
1_{|E|}=\iota_*(2+\beta \tilde{c}_1(L|_D)(1_D))\quad \mbox{in}\; \CK_*(|E|).
\]
where $L|_D$ is the restriction of $L$ to $D$. 
\end{lem}
Now we calculate the classes $[X_i]$.
\begin{lem}
In $\CK_*(Q(E))$ the class of the subvariety $X_i$ for $-n\leq i<0$ is given by 
\[
[X_i] =\tilde{c}_{n+i}(S^{\vee}\otimes E/F^i)(1_{Q(E)}).
\]
For $0\leq i< n$, the class of $X_i$ in $\CK_*(Q(E))$  satisfies the following identity
\begin{equation}\label{eqQEX}
\big(2+\beta \tilde{c}_1(S^{\vee}\otimes (F^0)^{\perp}/F^0)\big)([X_i]) = \tilde{c}_{n+i}\big(S^{\vee}\otimes (E/(F^0)^{\perp}\oplus F^0/F^i)\big)(1_{Q(E)}).
\end{equation}
\end{lem}
\begin{proof}
For $i<0$, the formula  is a simple consequence of Lemma \ref{lemG-B}.  We show the case when $i=0$ by computing the class $[X_0]$ in $\CK_*(Q(E))$ in two different ways. The variety $X_0=\PP(F^0)$ is a divisor in $\PP((F^0)^{\perp})$, corresponding the line bundle $S^{\vee}\otimes (F^0)^{\perp}/F^0$ 
and hence by Lemma \ref{lemG-B} one has $[X_0]=\tilde{c} _1(S^{\vee}\otimes (F^0)^{\perp}/F^0)(1_{\PP((F^0)^{\perp})})$ in $\CK_*(\PP((F^0)^{\perp}))$. As noted above, the scheme theoretic intersection $Q(E)\cap \PP((F^0)^{\perp})$ is not reduced and it defines the divisor $2 X_0$ on $\PP((F^0)^{\perp})$. From Lemma \ref{lem:divisorclass}, we have
\begin{equation}\label{eq:nr}
1_{Q(E) \cap \PP((F^0)^{\perp})}=\iota_*(2 + \beta \tilde{c} _1(S^{\vee}\otimes (F^0)^{\perp}/F^0)(1_{X_0}))
\end{equation}
where $\iota: X_0 \to Q(E)\cap \PP((F^0)^{\perp})$ is the natural morphism. Since the class $2 + \beta \tilde{c} _1(S^{\vee}\otimes (F^0)^{\perp}/F^0)$ is defined over $Q(E)$, by pushing forward (\ref{eq:nr}) to $Q(E)$ and by the projection formula, we obtain 
\[
[Q(E) \cap \PP((F^0)^{\perp})]=(2 + \beta \tilde{c} _1(S^{\vee}\otimes (F^0)^{\perp}/F^0))([X_0]).
\]
On the other hand,  by Lemma \ref{lemG-B} we have $[Q(E)\cap \PP((F^0)^{\perp})]=\tilde{c}_n(S^{\vee}\otimes E/(F^0)^{\perp})(1_{Q(E)})$. This proves the case $i=0$. 

For $i>0$, we have $[X_i] = \tilde{c}_{i}(S^{\vee}\otimes F^0/F^i)(1_{X_0})$ in $\CK_*(X_0)$. By pushing it forward to $Q(E)$, we obtain $[X_i] = \tilde{c}_{i}(S^{\vee}\otimes F^0/F^i)([X_0])$ in $\CK_*(Q(E))$. By applying $2 + \beta \tilde{c} _1(S^{\vee}\otimes (F^0)^{\perp}/F^0)$ and the identity (\ref{eqQEX}) for $i=0$, we obtain the identity:
\begin{eqnarray*}
(2 + \beta \tilde{c} _1(S^{\vee}\otimes (F^0)^{\perp}/F^0))([X_i])
&=&(2 + \beta \tilde{c} _1(S^{\vee}\otimes (F^0)^{\perp}/F^0))\circ \tilde{c}_{i}(S^{\vee}\otimes F^0/F^i)([X_0])\\
&=& \tilde{c}_n(S^{\vee}\otimes E/(F^0)^{\perp})\circ \tilde{c}_{i}(S^{\vee}\otimes F^0/F^i)(1_{Q(E)})\\
&=& \tilde{c}_{n+i}(S^{\vee}\otimes (E/(F^0)^{\perp}\oplus F^0/F^i))(1_{Q(E)}),
\end{eqnarray*}
in $\CK_*(Q(E))$.
\end{proof}
The non-degenerate symmetric form of $E$ induces an isomorphism between the trivial line bundle and $(F^0)^{\perp}/F^0 \otimes (F^0)^{\perp}/F^0$. Therefore $c_1((F^0)^{\perp}/F^0)=0$ in $\CK^*(Q(E))\otimes \ZZ[1/2]$. Thus we have the following corollary.
\begin{cor}\label{corXi}
In $\CK_*(Q(E))\otimes_{\ZZ}\ZZ[1/2]$, we have 
\[
[X_i] = \begin{cases}
\tilde{c}_{n+i}(S^{\vee}\otimes E/F^i)(1_{Q(E)}) & (-n\leq i < 0)\\
\left(\displaystyle\frac{1}{(2+\beta \tilde{c}_1(S^{\vee}))} \tilde{c}_{n+i}(S^{\vee}\otimes E/F^i)\right)(1_{Q(E)}) & (0\leq  i \leq n-1).
\end{cases}
\]
\end{cor}
\begin{rem}
By Remark \ref{remLine}, we can rewrite the formula in Corollary \ref{corXi} in $\CK^*(Q(E))\otimes_{\ZZ}\ZZ[1/2]$ as
\[
[X_i] = \begin{cases}
\scS_{n+i}((S- E/F^i)^{\vee}) & (-n\leq i < 0)\\
\displaystyle\frac{1}{2} \sum_{s\geq 0} \left(\displaystyle\frac{-\beta}{2}\right)^s \scS_{n+i+s}((S- E/F^i)^{\vee}) & (0\leq  i \leq n-1).
\end{cases}
\]
\end{rem}
\subsection{Resolution of singularities}
Consider the $r$-step flag bundle $\pi: \Fl_r(U) \to \OG^k(E)$ as before. We let $D_1\subset \cdots \subset D_r$ be the tautological flag. Recall that $\Fl_r(U)$ can be constructed as the tower of projective bundles
\begin{equation}\label{BPtower}
\pi: \Fl_r(U)=\PP(U/D_{r-1}) \stackrel{\pi_r}{\longrightarrow} \cdots \stackrel{\pi_{3}}{\longrightarrow}\PP(U/D_1) \stackrel{\pi_2}{\longrightarrow}  \PP(U) \stackrel{\pi_1}{\longrightarrow} \OG^k(E)
\end{equation}
We regard $D_j/D_{j-1}$ as the tautological line bundle of $\PP(U/D_{j-1})$ where we let $D_0=0$. For each $j=1,\dots,r$, let $\tilde{\tau}_j:=\tilde{c}_1((D_j/D_{j-1})^{\vee})$ be the Chern class operator of $(D_j/D_{j-1})^{\vee}$ on $\CK_*(\PP(U/D_{j-1}))$.
\begin{defn}
For each $j=1,\dots,r$, we define a subvariety $Z_j$ of $\PP(U/D_{j-1})$ by 
\[
Z_j := \{ (x,U_x, (D_1)_x, \dots, (D_j)_x) \in \PP(U/D_{j-1}) \ |\ (D_i)_x \subset F^{\chi_i}_x, \ i=1,\dots,{j}\}.
\]
We set $Z_0:=\OG^k(E)$ and $Z_{\lambda}:=Z_r$. Let $P_{j-1}:=\pi_j^{-1}(Z_{j-1})$ and consider the projection $\pi_j': P_{j-1} \to Z_{j-1}$ and the obvious inclusion $\iota_j: Z_j \to P_{j-1}$. Let $\varpi_j:=\pi_j'\circ\iota_j$. We have the commutative diagram
\[
\xymatrix{
\PP(U/D_{j-1}) \ar[r]_{\pi_j} & \PP(U/D_{j-2})\\
P_{j-1} \ar[r]_{\pi_j'}\ar[u] & Z_{j-1}\ar[u]\\
Z_j\ar[u]_{\iota_j}\ar[ru]_{\varpi_j} &
}
\]
Let $\varpi:=\varpi_1\circ\cdots\circ\varpi_r: Z_{\lambda} \to \OG^k(E)$.
\end{defn}
Now we prove the key lemma to compute the pushforward of the fundamental class of $Z_{\lambda}$.
\begin{lem}\label{Ylem2B}
For each $j=1,\dots, r$, the variety $Z_j$ is regularly embedded in $P_{j-1}$ and $P_{j-1}$ is regularly embedded in $\PP(U/D_{j-1})$. Moreover, in $\CK_*(P_{j-1})$ we have
\[
\iota_{j*}(1_{Z_j})=\tilde{\alpha}_j(1_{P_{j-1}}),
\]
where
\[
\tilde{\alpha}_j = \begin{cases}
\tilde{c}_{\lambda_j + n - k - j}((D_j/D_{j-1})^{\vee}\otimes (D_{\gamma_j}^{\perp}/F^{\chi_j})) & (-n\leq \chi_j<0)\\
\displaystyle\frac{1}{(2+\beta \tilde{c}_1((D_j/D_{j-1})^{\vee}))} \tilde{c}_{\lambda_j + n - k - j}((D_j/D_{j-1})^{\vee}\otimes (D_{\gamma_j}^{\perp}/F^{\chi_j}))
& (0\leq \chi_j< n).
\end{cases}
\]
\end{lem}
\begin{proof}
First we consider the case when $j=1$. Let $Q(E) \to P_0=\PP(U)$ be the quadric bundle where $E$ is regarded as a bundle over $P_0$ by pullback. Since $U$ is isotropic, there is an obvious regular embedding $s_1: P_0 \to Q(E)$. If $S_1$ is the tautological line bundle of $Q(E)$, then $s_1^*S_1 = D_1$. Let $X_i$ be the subvariety of $Q(E)$ defined at (\ref{Xi}), then we find that $Z_1 = s_1^{-1}(X_{\chi_1})$. This implies that $\iota_1: Z_1 \to P_0$ is a regular embedding. Moreover, by pulling back the formula in Lemma \ref{corXi}, we obtained the claim. 

To prove the claim for $j\geq 2$, we use induction on $j$. In particular, the regularity of $Z_i \inc P_{i-1}$ for all $i<j$ implies the regularity of $P_{j-1} \inc \PP(U/D_{j-1})$ since $P_{j-1}=\pi_j^{-1}(Z_{j-1})$.

Over $P_{j-1}$ we know that $D_i$ is a subbundle of $F^{\chi_i}$ for all $i<j$. First suppose that $\chi_j\geq 0$. Then $\gamma_j=j-1$. Consider the quadric bundle $Q(D_{j-1}^{\perp}/D_{j-1}) \to P_{j-1}$. The line bundle $D_j/D_{j-1}$ over $P_{j-1}$ defines a section $s: P_{j-1} \to Q(D_{j-1}^{\perp}/D_{j-1})$. Then it is easy to see that we have a fiber diagram
\begin{equation}\label{diagQY}
\xymatrix{
P_{j-1} \ar[r]^{s\  \ \ \ \ \ \ \  } & Q(D_{j-1}^{\perp}/D_{j-1})\\
Z_j \ar[u] \ar[r]& \PP(F^{\chi_j}/D_{j-1})\ar[u]
}
\end{equation}
Since $\PP(F^{\chi_j}/D_{j-1})$ is regularly embedded in $Q(D_{j-1}^{\perp}/D_{j-1})$, it follows that $Z_j$ is regularly embedded in $P_{j-1}$. Moreover, by the diagram (\ref{diagQY}), we have 
\[
[Z_j] =s^*[\PP(F^{\chi_j}/D_{j-1})].
\]
in $\CK_*(P_{j-1})$. Thus by Lemma \ref{corXi}, we obtain the desired formula.

Next suppose $\chi_j<0$. Similarly to the proof of Lemma \ref{Ylem1C}, we find that $F^{\chi_j} \subset (F^{\chi_{\gamma_j}})^{\perp}\subset D_{\gamma_j}^{\perp}$ over $P_{j-1}$. Thus we have the bundle map $D_j/D_{j-1}\to D_{\gamma_j}^{\perp}/F^{\chi_j}$ over $P_{j-1}$. Thus it defines a section $\sigma$ of $(D_j/D_{j-1})^{\vee} \otimes D_{\gamma_j}^{\perp}/F^{\chi_j}$. We can show from the local equation that the zero scheme of $\sigma$ is reduced and it coincides with $Z_j$. Since the codimension of $Z_j$ and the rank of the bundle $(D_j/D_{j-1})^{\vee} \otimes D_{\gamma_j}^{\perp}/F^{\chi_j}$ agree, we can conclude that $Z_j$ is regularly embedded in $P_{j-1}$. Moreover by Lemma \ref{lemCKGB}, we obtain the desired formula.
\end{proof}
\begin{lem}\label{Ylem1B}
The variety $Z_{\lambda}$ is irreducible and has at worst rational singularity. Furthermore, $Z_{\lambda}$ is birational to $\Omega_{\lambda}$ through $\varpi$.
\end{lem}
\begin{proof}
Consider the following $r$-step isotropic partial flag bundle $\Fl^{isot}_r(E)$ over $X$: the fiber at $x\in X$ consists of flags $(C_{\bullet})_x$ of isotropic subspaces $(C_1)_x\subset \cdots \subset (C_r)_x \subset E_x$ such that $\dim (C_{j})_x = {j}$. Let $Z$ be an $\OG^k(\bbF^{2(n-r)})$-bundle over $\Fl^{isot}_r(E)$ defined by
\[
Z= \{ (x,(C_{\bullet})_x, V_x) \ |\ V_x \in \OG^k((C_r^{\perp}/C_r)_x)\}.
\]
Let $W_{\lambda}$ be the degeneracy locus in $\Fl^{isot}_r(E)$ defined by
\[
W_{\lambda} :=\{ (x, (C_{\bullet})_x) \in \Fl^{isot}_{r}(E) \ |\  \dim ( F^{\chi_i}_x \cap(C_i)_x) \geq i, \ \ \ i=1,\dots,r\}.
\]
Consider the total space $Z|_{W_{\lambda}}$ of the restriction of the bundle $Z$ to $W_{\lambda}$:
\[
Z|_{W_{\lambda}} = \{(x,(C_{\bullet})_x, V_x) \in Z \ |\ \dim ( F^{\chi_i}_x \cap(C_i)_x) \geq i, \ \ \ i=1,\dots,r \}.
\]
Note that the conditions imply $(C_i)_x \subset F^{\chi_i}_x$ for each $i$. We can show that the variety $Z|_{W_{\lambda}}$ is isomorphic to $Z_{\lambda}$. Indeed, recall that 
\[
Z_{\lambda}=\{(x,U_x, (D_{\bullet})_x) \in \Fl_r(U)\ |\ (D_i)_x \subset F^{\chi_i}_x, i=1,\dots, r\}.
\]
The isomorphism $Z|_{W_{\lambda}} \to Z_{\lambda}$ is given by
\[
(x, V_x, (C_{\bullet})_x) \mapsto (x, \widetilde{V_x}, (C_{\bullet})_x),
\]
where the $(n-k)$-dimensional isotropic subspace $\widetilde{V_x} \subset E_x$ is defined as the preimage of $V_x$ under the quotient map  $(C_r^{\perp})_x \to (C_r^{\perp}/C_r)_x$. It follows from a well-known fact about Schubert varieties that the variety $W_{\lambda}$ is irreducible and has at worst rational singularity (\textit{cf.} \cite[p.274, 8.2.2. Theorem (c)]{Kumar}). Therefore $Z_{\lambda}$ is irreducible and has at worst rational singularity as well. 

For the latter claim, first note that $\varpi$ is the restriction of $\pi$ to $Z_{\lambda}$. Let $\Omega_{\lambda}^{B\circ}$ be an open set of $\Omega_{\lambda}$, consisting of $(x, U_x)\in \SG^{k}(E)$ such that $\dim (F^{\chi_{i}}_x \cap U_x) = {i}$ and $\dim (F^{\chi_{i}+1}_x\cap U_x)= {i}-1$ for all ${i}=1,\dots,r$.  Let $Z_{\lambda}^{\circ}$ be the preimage of $\Omega_{\lambda}^{\circ} $ by $\pi|_{Z_{\lambda}}: Z_{\lambda} \to \Omega_{\lambda}$. Then it is easy to see that $\pi|_{Z_{\lambda}^{\circ}}: Z_{\lambda}^{\circ} \to \Omega_{\lambda}^{\circ}$ is bijective. Thus in the view of the irreducibility of $Z_{\lambda}$ and $\Omega_{\lambda}$, this implies that $\pi|_{Z_{\lambda}}$ is birational. 
\end{proof}
Similarly to Lemma \ref{lem1C}, Lemma \ref{Ylem1B} implies the following lemma.
\begin{lem}\label{lem1B}
The fundamental class of $\Omega_{\lambda}$ in $\CK_*(\OG^k(E))$ is given by 
\[
[\Omega_{\lambda}] = \varpi_*(1_{Z_{\lambda}}).
\]
\end{lem}
\subsection{Pushforward formula and umbral calculus}
\begin{defn}\label{dfBclass}
For each $m\in \ZZ$, we define the Segre class operators $\tscB_m^{(\ell)}$ for $\CK_*(\OG^k(E))\otimes_{\ZZ}\ZZ[1/2]$ by the following generating function
\[
\sum_{m\in \ZZ}\tscB_m^{(\ell)} u^m = 
\begin{cases}
\tscS((U- E/F^{\ell})^{\vee};u) & (-n\leq \ell <0)\\
\displaystyle\frac{1}{2+ \beta u^{-1}} \tscS((U- E/F^{\ell})^{\vee};u) & (0\leq \ell \leq n).
\end{cases}
\]
Or equivalently, 
\[
\tscB_m^{(\ell)} := \begin{cases}
\tscS_m((U- E/F^{\ell})^{\vee}) & (-n\leq \ell <0)\\
\displaystyle\frac{1}{2} \sum_{s\geq 0} \left(\displaystyle\frac{-\beta}{2}\right)^s\tscS_{m+s}((U- E/F^{\ell})^{\vee}) & (0\leq \ell \leq n).
\end{cases}
\]
In $\CK^*(\OG^k(E))\otimes_{\ZZ}\ZZ[1/2]$, we denote $\scB_m^{(\ell)}:=\tscB_m^{(\ell)}(1_{\OG^k(E)})$.
\end{defn}
\begin{example}
Let $\lambda=(\lambda_1) \in \SP^k(n)$. The corresponding degeneracy loci is denoted by $\Omega_{\lambda_1}$:
\[
\Omega_{\lambda_1} = \{(x,U_x)\in \OG^k(n) \ |\ \dim(U_x \cap F^{\chi_1}) \geq 1\}.
\]
By Proposition \ref{push of tensor}, Lemma \ref{lem1B} and \ref{Ylem2B}, we have $[\Omega_{\lambda_1}] = \scB_{\lambda_1}^{(\chi_1)}$ in $\CK^*(\OG^k(E))\otimes_{\ZZ}\ZZ[1/2]$.
\end{example}
\begin{lem}\label{lempushB}
In $\CK_*(Z_{j-1})\otimes_{\ZZ}\ZZ[1/2]$, we have
\[
\varpi_{j*}\circ  \tilde{\tau}_j^s(1_{Z_j}) = \sum_{p= 0}^{\infty} \sum_{q=0}^p\binom{p}{q}\beta^q\tilde{c}_p(D_{j-1} - D_{\gamma_j}^{\vee})\circ\tscB_{\lambda_j+s-p+q}^{(\chi_j)}(1_{Z_{j-1}}), \ \ \ s\geq 0.
\]
\end{lem}
\begin{proof}
By Lemma \ref{Ylem2B}, we have
\begin{eqnarray*}
\varpi_{j*}\circ  \tilde{\tau}_j^s(1_{Z_j}) 
=\pi'_{j*}\circ\iota_{j*}\circ  \tilde{\tau}_j^s(1_{Z_j})
=\pi'_{j*}\circ  \tilde{\tau}_j^s\circ\iota_{j*}(1_{Z_j}) 
=\pi'_{j*}\circ  \tilde{\tau}_j^s\circ\tilde{\alpha}_j(1_{P_{j-1}}).
\end{eqnarray*}
Suppose that $\chi_j<0$. By Theorem \ref{thmlcipush},  the right hand side equals 
\begin{eqnarray*}
\tscS_{\lambda_j+s}((U/D_{j-1} - D_{\gamma_j}^{\perp}/F^{\chi_j})^{\vee})(1_{Z_{j-1}})
=\tscS_{\lambda_j+s}((U- E/F^{\chi_j} -D_{j-1} + D_{\gamma_j}^{\vee})^{\vee})(1_{Z_{j-1}})
\end{eqnarray*}
where $D_{\gamma_j}^{\perp} = E - D_{\gamma_j}^{\vee}$. Then (\ref{eq1}) proves the formula. Similarly, if $0\leq \chi_j$, Theorem \ref{thmlcipush} implies that the right hand side equals 
\begin{eqnarray*}
\sum_{s'=0}^{\infty}\left(\frac{-\beta}{2}\right)^{s'}\tscS_{\lambda_j+s+s'}((U/D_{j-1} - D_{\gamma_j}^{\perp}/F^{\chi_j})^{\vee})(1_{Z_{j-1}}),
\end{eqnarray*}
and the claim follows from (\ref{eq1}).
\end{proof}
Set $R:=\CK^*(\OG^k(E))\otimes_{\ZZ}\ZZ[1/2]$ and let $\calL^R$ be the ring of formal Laurent series with indeterminates $t_1,\dots, t_{r}$ defined in Definition \ref{def fls}.
\begin{defn}
Define a homomorphism $\phi_1: \calL^{R}\otimes_{\ZZ}\ZZ[1/2] \to \CK_*(\OG^k(E))\otimes_{\ZZ}\ZZ[1/2]$ of graded $R$-modules by 
\[
\phi_1(t_1^{s_1}\cdots  t_r^{s_r})= \tscS_{s_1}((U -E/F^{\chi_1})^\vee) \circ\cdots\circ  \tscS_{s_r}((U -E/F^{\chi_r})^\vee)(1_{\OG^k(E)}).
\]
Similarly, for $j\geq 2$, define a homomorphism $\phi_{j}: \calL^{R,j}\otimes_{\ZZ}\ZZ[1/2] \to \CK_*(Z_{j-1})\otimes_{\ZZ}\ZZ[1/2]$ graded $R$-modules by
\[
\phi_j( t_1^{s_1}\cdots  t_r^{s_r})=  \tilde{\tau}_1^{s_1}\circ\cdots \circ\tilde{\tau}_{j-1}^{s_j}\circ\tscS_{s_j}((U -E/F^{\chi_j})^\vee)\circ \cdots  \circ\tscS_{s_r}((U -E/F^{\chi_r})^\vee)(1_{Z_{j-1}}).
\]
Note that for each $i$ such that $j\leq i \leq r$ and $\chi_i\geq 0$, we have 
\begin{equation}\label{remP}
\phi_j\left(\frac{t_i^{m}}{2+\beta t_i}\right) =\tscB_m^{(\chi_i)}(1_{Z_{j-1}}), \ \ \ m\in \ZZ.
\end{equation}
\end{defn}
Similarly to Lemma \ref{lemp-phiC} and Proposition \ref{propmainC},  starting from Lemma \ref{lempushB} we can show the following lemma and proposition.
\begin{lem}
We have
\[
\varpi_{j*}\circ  \tilde{\tau}_j^s(1_{Z_j})= \begin{cases}
\phi_j\left(  t_j^{\lambda_j+s}\dfrac{\prod_{i=1}^{j-1}( 1-\bar t_i/\bar t_j)}{\prod_{i=1}^{\gamma_j}(1 -  t_i/\bar t_j)} \right) & (\chi_j<0),\\
\phi_j\left(  \dfrac{t_j^{\lambda_j+s}}{2+\beta t_j}\dfrac{\prod_{i=1}^{j-1}( 1-\bar t_i/\bar t_j)}{\prod_{i=1}^{\gamma_j}(1 -  t_i/\bar t_j)}\right) & (0\leq \chi_j),
\end{cases}
\]
for all $s\geq 0$.
\end{lem}
\begin{prop}\label{prop1B} 
We have
\[
 \varpi_{1*}\circ\cdots\circ \varpi_{r*}(1_{Z_{\lambda}}) = \phi_1\left(t_1^{\lambda_1}\cdots t_r^{\lambda_r}
\prod_{1\leq i\leq r \atop{\chi_i\geq 0}}\frac{1}{2+\beta t_i}
\frac{\prod_{(i,j)\in \Delta_{r}}(1- \bar t_i/\bar t_j)}{\prod_{(i,j) \in C(\lambda)} (1- t_i/\bar t_j)}\right).
\]
\end{prop}
\subsection{Main theorem for type B}
Let us recall from Section \ref{secmainC} the following Laurent series in $\calL^{R}$:
\begin{eqnarray*}
F_{i,j}^I(t)= \sum_{p,q\in \ZZ \atop{p\geq 0, p+q\geq 0}} \ff_{pq}^{ij,I} t_i^pt_j^q; \ \ \ \ \ 
F_i^I(t)= \sum_{p\in \ZZ\atop{p \geq 0}} \ff_p^{i,I} t_i^p.
\end{eqnarray*}
\begin{thm}\label{mainB}
Let $\lambda \in \SP^k(n)$ of length $r$ with $\chi$ its characteristic index. In $\CK^*(\OG^k(E))\otimes_{\ZZ}\ZZ[1/2]$, the fundamental class of the degeneracy locus $\Omega_{\lambda}$ is given by 
\[
[\Omega_{\lambda}] = \sum_{I \subset D(\lambda)_r} \Pf\left(\sum_{p,q\in \ZZ \atop{p\geq 0, p+q\geq 0}} \ff_{pq}^{ij,I} \scB_{\lambda_i+d_i^I+p}^{(\chi_i)}\scB_{\lambda_j+d_j^I+q}^{(\chi_j)}\right)_{1\leq i<j\leq m },
\]
where  $m=r$ if $r$ is even and $m=r+1$ if $r$ is odd, and $\scB_{-i}^{(-n-1)}:=(-\beta)^i$ for $i\leq 0$. In particular, if $r$ is odd, then $(i,m)$-entry of the Pfaffian reduces to $\sum_{p\in \ZZ\atop{p \geq 0}} \ff_p^{i,I} \scB_{\lambda_i+d_i^I+p}^{(\chi_i)}$. 
\end{thm}
\begin{proof}
By Lemma \ref{lem1B} and Proposition \ref{prop1B}, we have 
\begin{equation*}
[\Omega_{\lambda}] = \phi_1 \left( t_1^{\lambda_1}\cdots t_r^{\lambda_r}\prod_{1\leq i \leq r \atop{\chi_i\geq 0}}\frac{1}{2+\beta t_i}\frac{\prod_{(i,j)\in \Delta_r}(1- \bar t_i/\bar t_j)}{\prod_{(i,j) \in C(\lambda)} (1- t_i/\bar t_j)}\right).
\end{equation*}
Then an application of Lemma \ref{lem4C} together with (\ref{remP}) proves the formula.  
\end{proof}

%
\section{Equivariant connective $K$-theory}\label{sec7}
In this section, we introduce torus equivariant connective $K$-theory following Krishna \cite{Krishna} and give Goresky--Kottwitz--MacPherson (GKM) type description for the equivariant connective $K$-theory of isotropic Grassmannians. For the rest of the paper we fix a nonnegative integer $k$. 
\subsection{Preliminaries}
Let $T_n$ be a standard algebraic torus $(\bbG_m)^n$. If $T_n$ acts on a smooth variety $X$, the $T_n$-equivariant connective $K$-theory $\CK^*_{T_n}(X)$ is a graded algebra over $\CK^*_{T_n}(pt)$, the $T_n$-equivariant connective $K$-theory of a point. If $E \to X$ is a $T_n$-equivariant vector bundle, the $i$-th $T_n$-equivariant Chern classes of $E$ is denoted also by $c_i(E)$ as in the rest of the paper.

First of all, we fix the identification of $\CK^*_{T_n}(pt)$ with a ring of graded formal power series. Let $\varepsilon_1,\dots, \varepsilon_n$ be the standard basis of the character group of $T_n$. Let $L_i$ be the one dimensional representation of $T_n$ with character $\varepsilon_i$. Let $b_1,b_2,\dots$ be an infinite sequence of indeterminants. Then we have the isomorphism
\begin{equation}\label{eq:CK(pt)}
\CK^*_{T_n}(pt) \to \QQ[\beta][[b_1,\dots,b_n]]_{\gr}; \ \ c_1(L_{i}) \mapsto b_i
\end{equation}
of graded algebras over $\QQ[[\beta]]_{\gr}$ (\cite[\S 2.6]{Krishna}). We set $\CK^*_{T_n}:=\CK^*_{T_n}(pt)$ for simplicity. 

Similarly we denote $\CK^*_{T_{\infty}}:=\QQ[\beta][[b]]_{\gr}:=\QQ[\beta][[b_1,b_2,\dots]]_{\gr}$. In the rest of the paper, we regard a $\CK^*_{T_n}$-algebra as a $\CK^*_{T_\infty}$-algebra via the projection $\CK^*_{T_\infty} \to \CK^*_{T_n}$ defined by $b_i=0$ for all $i>n$.
\begin{rem}
After specializing at $\beta=0$, $\QQ[\beta][[b_1,\dots,b_n]]_{\gr}$ becomes $\QQ[b_1,\dots,b_n]$. If we specialize at $\beta=-1$, one obtains $\QQ[[b_1,\dots, b_n]]$, which can be identified with the completion of the representation ring $R(T_n)$ with rational coefficients
\[
K_{T_n}({pt})\otimes_{\ZZ}\QQ= R(T_n)\otimes_{\ZZ}\QQ=\QQ[e^{\pm \varepsilon_1},\dots, e^{\pm \varepsilon_n}],
\]
where we naturally identify the $K$-theory class of $L_i$ with $e^{-\varepsilon_i}$ and $b_i$ corresponds to $1 - e^{\varepsilon_i}$ (\textit{cf.} Krishna \cite[Theorem 7.3]{Krishna0}).
\end{rem}
\subsection{Symplectic and odd orthogonal Grassmannians}\label{sec: CKT(SG)}
Recall that $\bbF$ is our base field. Let $E=E^{(n)}$ be a vector space $\bbF^{2n}$ or $\bbF^{2n+1}$ of dimension $2n$ or $2n+1$ respectively. We fix bases by
\[
\bbF^{2n}=\Span\{\vece_{\bar i}, \vece_i\:|\;1\leq i\leq n\},\ \ \ \ \ \ \ 
\bbF^{2n+1}=\Span\{\vece_{\bar i}, \vece_i\:|\;1\leq i\leq n\} \cup \{\vece_0\}.
\]
together with  the symplectic form and non-degenerate symmetric form
\[
\sum_{i=1}^n \vece_i^*\wedge \vece_{\bar i}^*, \ \ \  \ \ \  \vece_0^*\otimes \vece_0^*+\sum_{i=1}^n \vece_i^*\otimes \vece_{\bar i}^*
\]
respectively, where $\vece_i^*$ denotes the dual of $\vece_i$. We define the action of $T_n$ on $E$ as follows: $T_n$ acts on $\bbF \vece_i$ with weight $\varepsilon_i$ and on $\bbF\vece_{\bar i}$ with weight $-\varepsilon_i$. This identifies $T_n$ with maximal tori of $\Sp_{2n}(\bbF)$ and $O_{2n+1}(\bbF)$.

For each $\ell \in \{1,\dots, n\}$, define the subspaces of $E$ 
\[
F^{\ell} = \Span_{\bbF}\{\vece_n,\dots, \vece_{\ell+1}\}, \ \ \ F^{-\ell} = (F^0)^{\perp} \oplus \Span_{\bbF}\{\vece_{\bar1},\cdots, \vece_{\bar\ell}\}
\]
so that $c(E/F^0;u) = \prod_{i=1}^n(1+ b_i u)$ and 
\[
c(E/F^{\ell};u)= c(E/F^0;u)\prod_{i=1}^{\ell}(1 + b_iu), \ \ \ c(E/F^{-\ell};u)= c(E/F^0;u)\prod_{i=1}^{\ell}\frac{1}{1 + \bar b_iu}.
\]
Here let us observe that the action of $T_n$ on $\bbF \vece_0$ is trivial, so that $c((F^0)^{\perp}/F^0;u)=1$.

For $n \geq k$, let $\calG_n^k$ be the Grassmannians of $n-k$ dimensional isotropic subspaces in $E$, {\it i.e.}
\[
\calG_n^k := \begin{cases}
\SG^k(E) & \mbox{ if $E=\bbF^{2n}$ with the symplectic form},\\
\OG^k(E) & \mbox{ if $E=\bbF^{2n+1}$ with the non-degenerate symmetric form}.
\end{cases}
\]
\emph{We write $X=C$ for the symplectic case, and $X=B$ for the odd orthogonal case.} For each $\lambda\in \SP^{k}(n)$, the Schubert variety $\Omega_{\lambda}^X$ of $\calG_n^k$ by Definition \ref{dfOmegaC} (or  \ref{dfOmegaB}) is $T_n$-stable. Thus in $\CK^*_{T_n}(\calG_n^k)$ it defines the $T_n$-equivariant class $[\Omega_\lambda^X]_{T_n}$.  As a $\CK^*_{T_n}$-module, $\CK^*_{T_n}(\calG_n^k)$ is freely generated by $[\Omega_\lambda^X]_{T_n}, \lambda\in\SP^k(n)$. See \cite{Krishna}. 

Let $U$ be the tautological isotropic bundle of $\calG_n^k$. We denote the trivial bundle over $\calG_n^k$ with the fiber $F^i$ by the same symbol $F^i$. We define the classes $\scC_m^{(\ell)}$ in $\CK^*_{T_n}(\SG^k(\bbF^{2n}))$ and $\scB_m^{(\ell)}$ in $\CK^*_{T_n}(\OG^k(\bbF^{2n+1}))$ by using Definition \ref{dfCclass} and \ref{dfBclass} respectively where we regard the Chern (or Segre) classes as equivariant ones. Namely, for $m\in \ZZ$ and $\ell=-n,\cdots,n$, let
\[
\scC_m^{(\ell)} := \scS_{m}(U^{\vee}-(E/F^{\ell})^{\vee})
\]
and
\[
\scB_m^{(\ell)} := \begin{cases}
\scS_m((U- E/F^{\ell})^{\vee}) & (-n\leq \ell <0),\\
\displaystyle\frac{1}{2} \sum_{s\geq 0} \left(\displaystyle\frac{-\beta}{2}\right)^s\scS_{m+s}((U- E/F^{\ell})^{\vee}) & (0\leq \ell \leq n).
\end{cases}
\]
Theorem \ref{mainC} and \ref{mainB} hold in this equivariant setting. Indeed, let $BT_n$ be the classifying space of $T_n$ and $ET_n\rightarrow BT_n$ the universal bundle. Consider the bundle $ET_n\times_{T_n}E$ over $ET_n\times_{T_n}\calG_n^k$. We can apply Theorem \ref{mainC} or \ref{mainB} to every finite approximation of this bundle. Then the functoriality of Chern classes implies the claim. 
\subsection{GKM description}
It is well-known that the set $(\calG_n^k)^{T_n}$ of $T_n$-fixed  points in $\calG_n^k$ is bijective to $\SP^k(n)$. For each $\lambda \in \SP^k(n)$, let $e_\lambda$ denote the corresponding fixed point. Let $\Fun(\SP^{k}(n),\CK^*_{T_n})$ be the algebra of maps from $\SP^{k}(n)$ to $\CK^*_{T_n}$ where the algebra structure is given by the pointwise multiplication. Then we can identify $\CK^*_{T_n}((\calG_n^k)^{T_n})$ with $\Fun(\SP^{k}(n),\CK^*_{T_n})$ as graded $\CK^*_{T_n}$-algebras. For every inclusion  $\iota_n: (\calG_n^k)^{T_n} \inc \calG_n^k$, we can consider the following homomorphism of $\CK^*_{T_n}$-algebras given by pull-back:
\[
\iota_n^*: \CK^*_{T_n}(\calG_n^k)\to \Fun(\SP^{k}(n),\CK^*_{T_n}).
\]

Below, with the help of GKM theory, we describe $\CK^*_{T_n}(\calG_n^k)$ as the image of $\iota_n^*$. First we prepare some notations. Let $\Delta^+$ be the set of {\em positive roots} in $L:=\bigoplus_{i=1}^\infty \ZZ \varepsilon_i$ defined by 
\begin{eqnarray*}
&&\Delta^+ := \{ \varepsilon_i, 1\leq i \} \cup \{\varepsilon_j \pm \varepsilon_i\ |\ 1\leq i <  j\} \ \ \ \mbox{for type B and}\\
&&\Delta^+ := \{ 2\varepsilon_i, 1\leq i \} \cup \{\varepsilon_j \pm \varepsilon_i\ |\ 1\leq i <  j\} \ \ \ \mbox{for type C.}
\end{eqnarray*}
We define a map $e: L \to \CK^*_{T_\infty}$ by
\[
e(\varepsilon_i)=b_i,\ \  e(-\varepsilon_i)=\bar b_i, \ \  e(\alpha+\gamma)=e(\alpha)\oplus e(\gamma) \ \mbox{ and }  \ e(\alpha-\gamma)=e(\alpha)\ominus e(\gamma).
\]
Let $s_\alpha\in W_\infty$ be the simple reflection associated with the positive root $\alpha\in \Delta^+$. Note that $W_\infty$ acts naturally on the set $\SP^k$ via the bijection $\SP^k\cong W_\infty/W_{(k)}$ discussed in Section \ref{seccombBC}. Similarly, $W_n$ acts on $\SP^k(n)$ via the bijection $\SP^k(n)\cong W_n/W_{n,(k)}$. Let $\Delta^+_n:= \Delta^+\cap \ \Span_{\ZZ}\{\varepsilon_1,\dots, \varepsilon_n\}$. For each $\alpha \in \Delta^+_n$ one has $s_{\alpha} \in W_n$ and $e(\alpha)\in \CK^*_{T_n}$. 
\begin{defn}\label{dfGKMn}
Let $\frakK_n^{(k)}$ be the graded $\CK^*_{T_n}$-subalgebra of $\Fun(\SP^k(n), \CK^*_{T_n})$ defined as follows: a map $\psi: \SP^k(n) \to \CK^*_{T_n}$  is in $\frakK_n^{(k)}$ if and only if 
\[
\psi(s_{\alpha}\mu)-\psi(\mu) \in e(\alpha)\cdot \CK^*_{T_n} \ \ \mbox{ for all } \mu\in \SP^k(n) \ \mbox{ and }\  \alpha \in \Delta^+_n.
\]
\end{defn}
The next theorem holds by Corollary (3.20) in \cite{KostantKumar} ({\it cf.} Theorem 7.8 \cite{Krishna}).
\begin{thm}\label{thmGKM}
The map $\iota_n^*$ is injective and its image coincides with $\frakK_n^{(k)}$.
\end{thm}
\begin{rem}\label{remIndBC}
By the fact that $e(2\varepsilon_i) = b_i \oplus b_i = b_i(2+ \beta b_i)$ and since $2+\beta b_i$ is invertible in $\CK^*_{T_n}$, we can see that $\frakK_n^{(k)}$ is independent of the type B and C. Therefore we have $\CK^*_{T_n}(\SG^k(\bbF^{2n})) \cong \CK^*_{T_n}(\OG^k(\bbF^{2n+1}))$ as graded $\CK^*_{T_n}$-algebras with rational coefficients.
\end{rem}
The following proposition will be used in Section \ref{sec: GTheta}.
\begin{prop}[{\it cf.} Proposition 10.1 \cite{IkedaMihalceaNaruse}]\label{propUloc}
Let $\mu$ be a $k$-strict partition in $\SP^k(n)$ and $w$ its corresponding signed permutation in $W_n$. Consider the pullback $\iota_{\mu}^*: \CK^*_{T_n}(\calG_n^k)\to \CK^*_{T_n}$ of the inclusion $\iota_\mu: \{e_\mu\} \to\calG_n^k$. We have 
\[
\iota_{\mu}^*(c(U;u)) = \prod_{i=k+1}^n(1+ b_{w(i)} u),
\]
where we denote $b_{\bar i}:=\bar b_i$. 
\end{prop}
\subsection{Stability of Schubert classes}
Let $E^{(n)} \to E^{(n+1)}$ be the injective linear map defined by the inclusion of the basis elements. It induces an embedding $j_n: \calG_n^k \to \calG_{n+1}^k$, which is equivariant with respect to the corresponding inclusion $T_n \to T_{n+1}$. Consider its pullback
\[
j_n^*: \CK^*_{T_{n+1}}(\calG_{n+1}^k) \to \CK^*_{T_n}(\calG_n^k).
\]
Define
\[
\CK^*_{T_{\infty}}(\calG_{\infty}^k):= \bigoplus_{m\in \ZZ} \lim_{\longleftarrow \atop{n}} \CK^m_{T_n}(\calG_n^k),
\]
where we take the inverse limit with respect to $j_n^*$. Since we have
\begin{equation}\label{stabSch}
j_n^*[\Omega_{\lambda}^X]_{T_{n+1}} = 
\begin{cases}
[\Omega_{\lambda}^X]_{T_n} & \mbox{ if } \lambda\in \SP^k(n),\\
0 & \mbox{ if } \lambda\not\in \SP^k(n),
\end{cases}
\end{equation}
one obtains a unique element $[\Omega_{\lambda}^X]_T$ in $\CK^*_{T_{\infty}}(\calG_{\infty}^k)$ as a limit of the classes $[\Omega_{\lambda}^X]_{T_n}$. By the stability (\ref{stabSch}), together with the fact that the Schubert classes form a $\CK^*_{T_n}$-module basis, we can conclude the following.
\begin{lem}\label{lemLimBasis}
Any element $f$ of $\CK^*_{T_{\infty}}(\calG_{\infty}^k)$ can be expressed uniquely as a possibly infinite $\CK^*_{T_{\infty}}$-linear combination of the classes $[\Omega_{\lambda}^X]_T$:
\[
f = \sum_{\lambda\in \SP^k} c_{\lambda} [\Omega_{\lambda}^X]_T, \ \ \ \ \ c_{\lambda}\in \CK^*_{T_{\infty}}.
\]
\end{lem}
 
\section{The ring of double Grothendieck polynomials}\label{sec8}
In this section, we study an algebraic framework to introduce the functions that represent the Schubert classes for the $K$-theory of symplectic and odd orthogonal Grassmannians. They can be regarded as a generalization of  the ones developed in \cite{IkedaMihalceaNaruse} for double Schubert polynomials. We also follow Ikeda--Naruse \cite{IkedaNaruse} with a slight modification to be able to deal with connective $K$-theory.
\subsection{The ring $\hGG$ and its formal basis}
In this section, we define the ring $\hGG$ and show that $\GP$-functions form a formal basis of $\hGG$. Note that our $\hGG$ is a completion of the one defined in \cite{IkedaNaruse}. Let $x=(x_1,x_2,\ldots)$ be a sequence of  indeterminates and $\QQ[\beta][[x]]_{\gr}$ the ring of graded formal power series in $x_i$'s.
\begin{defn}\label{GGcond}
We denote by $\hGG_n$ the graded subring of $\QQ[\beta][[x_1,\ldots,x_n]]_{\gr}$ whose elements are the series $f(x)$ such that:
\begin{itemize}
\item[(1)] $f(x)$ is symmetric in $x_1,\ldots,x_n$.
\item[(2)] $f(t,\bar t,x_3,x_4,\dots,x_n)=f(0,0,x_3,x_4,\dots,x_n)$.
\end{itemize}
\end{defn}
These rings form a projective system with respect to the degree preserving homomorphism $\hGG_{n+1}\to \hGG_n$ given by $x_{n+1}=0$. Let us denote by $\hGG$ the graded projective limit of the projective system. We can identify $\hGG$ with the subring of $\QQ[\beta][[x_1,x_2,\dots]]_{\gr}$ defined by the conditions analogous to (1) and (2) above.
\begin{defn}
For each strict partition $\lambda=(\lambda_1,\dots,\lambda_r)$ of length $r \leq n$, we define 
\[
\GP_{\lambda}(x_1,\dots,x_n) = \frac{1}{(n-r)!} \sum_{w\in S_n} w\left[ x_1^{\lambda_1}\cdots x_r^{\lambda_r}\prod_{i=1}^r\prod_{j=i+1}^n\frac{x_i\oplus x_j}{x_i\ominus x_j} \right].
\]
\end{defn}
The polynomial $\GP_{\lambda}(x_1,\dots,x_n)$ is an element of $\hGG_n$ and in the projective limit, it defines an element $\GP_{\lambda}(x)$ in $\hGG$.

The next lemma is a slight modification of Theorem 3.1 \cite{IkedaNaruse}. Although in order to prove it we must work in the ring of graded formal power series, we leave the details to the reader since it is parallel to the original one.
\begin{lem}\label{lemfb}
A homogeneous element $f(x_1,\dots,x_n)$ in $\hGG_n$ of degree $r$ is uniquely expressed as a possibly infinite linear combination
\[
f(x_1,\dots, x_n) = \sum_{\lambda \in \SP_n} c_{\lambda} \GP_{\lambda}(x_1,\dots,x_n), \ \ \ c_{\lambda} \in \QQ[\beta]_{r-|\lambda|}.
\]
\end{lem}
Lemma \ref{lemfb} implies the following proposition.
\begin{prop}\label{propfb}
Any homogeneous element $f(x)$ of $\hGG$ with degree $m$ is uniquely expressed as a possibly infinite linear combination
\[
f(x) = \sum_{\lambda \in \SP} c_{\lambda} \GP_{\lambda}(x), \ \ \ c_{\lambda} \in \QQ[\beta]_{m-|\lambda|}.
\]
\end{prop}
\subsection{The ring $\calK_{\infty}$ and its GKM description}
For infinite sequences of variables $a=(a_1,a_2,\dots)$ and $b=(b_1,b_2,\dots)$, consider the rings
\begin{eqnarray*}
\calR_a&:=& \bigcup_{m=0}^{\infty} \QQ[\beta][[a_1,\dots, a_m]]_{\gr}, \ \ \ \calR_b:= \bigcup_{m=0}^{\infty} \QQ[\beta][[b_1,\dots,b_m]]_{\gr}.
\end{eqnarray*}
Define the $\calR_b$-algebra $\calK_{\infty}:=\hGG\otimes_{\QQ[\beta]}\calR_a \otimes_{\QQ[\beta]} \calR_b$.
\begin{rem}
The ring $\calK_{\infty}$ (resp. $\calK_{\infty}^{(k)}$) is the $K$-theoretic version of $\calR_{\infty}$ introduced in \cite{IkedaMihalceaNaruse} (resp. $\calR_\infty^{(k)}$ in \cite{IkedaMatsumura}). The corresponding \emph{double Grothendieck polynomials} constructed by Kirillov-Naruse \cite{KrNr} represent the $K$-theoretic equivariant Schubert classes. 
\end{rem}
\begin{defn}
We define the homomorphism of $\calR_b$-algebras
\[
\Phi_{\infty} : \calK_{\infty} \to \Fun(W_{\infty}, \calR_b); \ \ \ \ \Phi_{\infty}(f) := (v \mapsto \Phi_v(f))_{v \in W_{\infty}},
\]
where $\Phi_v: \calK_{\infty} \to \calR_b$ is the $\calR_b$-algebra homomorphism given by the substitution
\[
x_i \mapsto \begin{cases}  
b_{v(i)} & \mbox{ if $v(i)<0$}\\
0 & \mbox{ if $v(i)>0$}
\end{cases}
 \ \ \ \ \mbox{ and } \ \  a_i \mapsto \bar b_{v(i)}.
\]
\end{defn}
\begin{defn} 
Let $\frakK_{\infty}$ be the subalgebra of $\Fun(W_{\infty}, \calR_b)$ consisting of functions $\psi$ such that
\[
\psi(s_{\alpha} v) - \psi(v) \in e(\alpha)\cdot \calR_b, \ \ \mbox{ for all } v\in W_{\infty}\ \ \mbox{ and } \alpha \in \Delta^+.
\]
\end{defn}
As noted in Remark \ref{remIndBC}, the elements $2+\beta b_i$ are invertible in $\calR_b$ and hence the subalgebra $\frakK_{\infty}$ is independent of the types B and C.
\begin{lem}\label{lemfullsurj}
The image of $\Phi_{\infty}$ lies in $\frakK_{\infty}$.
\end{lem}
\begin{proof}
Let $b_v:=((b_v)_1,(b_v)_2,\dots)$ be the sequence defined by setting $(b_v)_i:=b_{v(i)}$ if $v(i)<0$ and $(b_v)_i:=0$ if $v(i)>0$. By the definition of $\calK_{\infty}$, it suffices to show that for any $F(x)\in \hGG$, 
\begin{eqnarray*}
&&F(b_{t_{ij}v}) - F(b_v) \in \lan  b_j\ominus b_i \ran \ \ \mbox{ for $j>i\geq 1$}\\
&&F(b_{s_{ij}v}) - F(b_v) \in \lan  b_j\oplus b_i \ran \ \ \mbox{ for $j>i\geq 1$}\\
&&F(b_{s_{ii}v}) - F(b_v) \in \lan  b_i \ran \ \ \mbox{ for $i\geq 1$},
\end{eqnarray*}
where $t_{ij}=s_{\varepsilon_j-\varepsilon_i}$, $s_{ij}=s_{\varepsilon_j+\varepsilon_i}$, $s_{ii}=s_{\varepsilon_i}$. These follow from an argument similar to the one in the proof of Lemma 7.1. in \cite{IkedaNaruse}.
\end{proof}
\begin{lem}\label{lemfullinj}
The map $\Phi_{\infty}$ is injective.
\end{lem}
\begin{proof}
By the definition of $\calK_{\infty}$ and Proposition \ref{propfb}, a homogeneous element $f$ of $\calK_{\infty}$ of degree $d$ can be uniquely written as
\begin{equation}\label{eqf1}
f=\sum_{\lambda\in \SP} c_{\lambda}(a;b) \GP_{\lambda}(x), \ \ \ c_{\lambda}(a;b)\in(\calR_a \otimes_{\QQ[\beta]} \calR_b)_{d - |\lambda|}.
\end{equation}
By the definitions of $\calK_{\infty}$ and $\calR_a \otimes_{\QQ[\beta]} \calR_b$, there exist $m,n$ such that for all $\lambda \in \SP$
\[
c_{\lambda}(a;b)=c_{\lambda}(a_1,\dots,a_n;b_1,\dots, b_m) \in \QQ[\beta][[a_1,\dots,a_n, b_1,\dots,b_m]]_{d-|\lambda|}.
\]
Suppose that $\Phi_{\infty}(f)=0$. Choose an integer $N \geq m+n+1$ and consider an element $v$ of $W_{\infty}$ given by
\[
v=(m+1,\dots,m+n,1,\dots,m, \overline{m+n+1},\dots,\overline{m+n+N}). 
\]
Then by applying $\Phi_v$ to (\ref{eqf1}) we obtain
\[
\Phi_v(f)=\sum_{\lambda\in \SP} c_{\lambda}(b_{m+1},\dots,b_{m+n};b_1,\dots, b_m) \GP_{\lambda}(b_{m+n+1},\dots,b_{m+n+N},0,0,\dots)=0.
\]
By Lemma \ref{lemfb}, we can conclude that $c_{\lambda}(b_{m+1},\dots,b_{m+n};b_1,\dots, b_n)=0$ for all $\lambda \in \SP$.
\end{proof}
 \subsection{The action of $W_{\infty}$ and the ring $\calK_{\infty}^{(k)}$}
For $i\geq 0$ we define operators $s_i^{a}$ on $\hGG\otimes_{\QQ[\beta]}\calR_a$ as follows: if $i\geq 1$, $s_i^{a}$ switches $a_i$ and $a_{i+1}$, and for $i=0$, 
\[
s^{a}_0 (f(x_1,x_2,\dots ;a_1,a_2, \dots)) := f(a_1, x_1,x_2,\dots;\bar a_1,a_2,\dots).
\]
These define actions of the Weyl group $W_{\infty}$ on $\hGG\otimes_{\QQ[\beta]}\calR_a$ and on $\calK_{\infty}$. 

On the other hand, we have  the following action of $W_{\infty}$ on $\Fun(W_{\infty}, \calR_b)$: let $s_i$ act on $\psi \in \Fun(W_{\infty}, \calR_b)$ by  $s_i^R(\psi)(v) := \psi(vs_i)$ for all $i\geq 0$.

With these actions, the algebraic localization map $\Phi_{\infty}$ is $W_{\infty}$-equivariant, {\it i.e.} $s_i^R\Phi_{\infty} = \Phi_{\infty} s_i^a$. The proof of this fact is similar to the one for Proposition 7.3 \cite{IkedaMihalceaNaruse}. 
\begin{defn}
Fix $k\geq 0$. Let $\calK_{\infty}^{(k)}$ be the subalgebra of $\calK_{\infty}$ invariant under the action of $W_{(k)}$. 
\end{defn}
\begin{rem}
It is clear that $\calK_{\infty}^{(k)} = (\hGG\otimes_{\QQ[\beta]}\calR_a)^{W_{(k)}}\otimes_{\QQ[\beta]}\calR_b$ where $(\hGG\otimes_{\QQ[\beta]}\calR_a)^{W_{(k)}}$ is the subring of $\hGG\otimes_{\QQ[\beta]}\calR_a$ invariant under the action of $W_{(k)}$.
\end{rem}
\begin{rem}
Since each element of $\calR_a$ involves only finitely many $a_i$'s, it follows that $(\hGG\otimes_{\QQ[\beta]}\calR_a)^{W_{(k)}}$ is contained in $\hGG[[a_1,\dots,a_k]]_{\gr}$ and we have $f \in (\hGG\otimes_{\QQ[\beta]}\calR_a)^{W_{(k)}}$ if and only if $s_i^{a} f = f$ for all $i=0,1,2,\dots,k-1$. 
\end{rem}
We can identify $\Fun(\SP^k, \calR_b)$ with the subalgebra $\Fun(W_{\infty}, \calR_b)^{W_{(k)}}$ of $\Fun(W_{\infty}, \calR_b)$ invariant under the action of $W_{(k)}$ as follows. For each $\lambda\in \SP^k$, let $w_{\lambda}$ be the $k$-Grassmannian element in $W_{\infty}$. For each $\psi \in \Fun(\SP^k, \calR_b)$, we can define an element $\Fun(W_{\infty}, \calR_b)^{W_{(k)}}$ by 
\[
v \mapsto \psi(w_{\lambda}) \ \ \mbox{if} \ \ v\in w_{\lambda}W_{(k)}.
\]
Under this identification, the localization map $\Phi_{\infty}$ induces the map $\Phi_{\infty}^{(k)}: \calK_{\infty}^{(k)} \to \Fun(\SP^k, \calR_b)$. 
Thus we have the following commutative diagram:
\[
\xymatrix{
\calK_{\infty} \ar[rr]_{\Phi_{\infty}\ \ \ } && \Fun(W_{\infty}, \calR_b)\\
\calK_{\infty}^{(k)} \ar[rr]_{\Phi_{\infty}^{(k)}\ \ \ } \ar[u] && \Fun(\SP^k, \calR_b). \ar[u]
}
\]
Observe that the $W_{\infty}$-action on $\Fun(W_{\infty}, \calR_b)$ preserves $\frakK_{\infty}$. Let $\frakK_{\infty}^{(k)}$ be the subring of $\frakK_{\infty}$ invariant under the action of $W_{(k)}$. Under the identification $\Fun(\SP^k, \calR_b) \cong \Fun(W_{\infty}, \calR_b)^{W_{(k)}}$,  we find that $\frakK^{(k)}_{\infty}$ is the subalgebra of $\Fun(\SP^{k}, \calR_b)$ consisting of functions $\psi$ such that
\[
\psi(s_{\alpha} \mu) - \psi(\mu) \in e(\alpha)\cdot \calR_b, \ \ \mbox{ for all } \mu\in \SP^{k}\ \ \mbox{ and } \alpha \in \Delta^+.
\]

The next proposition follows from Lemma \ref{lemfullsurj} and \ref{lemfullinj}.
\begin{prop}\label{universal localization}
The map $\Phi_{\infty}^{(k)}$ is injective and its image lies in $\frakK_{\infty}^{(k)}$.
\end{prop}
 
\section{Functions representing Schubert classes}\label{sec: GTheta}
In this section, we introduce the functions $\GC_{\lambda}$ and $\GB_{\lambda}$ that represent the Schubert classes in the equivariant connective $K$-theory of symplectic and odd orthogonal Grassmannians in type C and B respectively. Recall that we have fixed a nonnegative integer $k$ for the rest of the paper.
\subsection{The basic functions}
\begin{defn}\label{df:GTGH}
Define ${}_k\GC_m(x,a), {}_k\tGB_m(x,a) \in \hGG[[a_1,\dots,a_k]]_{\gr}$ by
\begin{eqnarray*}
&&{}_k\GC(x,a;u):=\sum_{m\in \ZZ} {}_k\GC_m(x,a) u^m =\frac{1}{1+\beta u^{-1}} \prod_{i=1}^{\infty} \frac{1 + (u+\beta) x_i}{1 + (u+\beta)\bar x_i}\prod_{i=1}^{k} (1 + (u+\beta) a_i),\\
&&{}_k\tGB\,(x,a;u):=\sum_{m\in \ZZ} {}_k\tGB_{m}(x,a) u^m = \frac{1}{2+\beta u^{-1}}{}_k\GC(x,a;u).
\end{eqnarray*}
In particular, we have
\[
{}_k\tGB_m(x)= \frac{1}{2} \sum_{s\geq 0} \left(\displaystyle\frac{-\beta}{2}\right)^s{}_k\GC_{m+s}(x).
\]
We set
\[
{}_k\GB_m(x,a) = \begin{cases}
{}_k\GC_m(x,a) & (m\leq k)\\
{}_k\tGB_m(x,a) & (m>k).
\end{cases}
\]
\end{defn}
\begin{lem}\label{lemGTinv} 
The functions ${}_k\GC_m(x,a)$ and ${}_k\tGB_m(x,a)$ are elements of $(\hGG\otimes_{\QQ[\beta]}\calR_a)^{W_{(k)}}$ for all $m\in \ZZ$.
\end{lem}
\begin{proof}
The polynomial ${}_k\GC_m(x,a)$ is invariant under the actions of $s_1^{a}, \dots,s_{k-1}^{a}$ since by definition it is symmetric in $a_1,\dots, a_k$. As in \cite[Proposition 5.1]{IkedaMatsumura}, we can check that $s_0^{a}$ preserves the generating function in Definition \ref{df:GTGH} and so does ${}_k\GC_m(x,a)$. Therefore the claim holds.
\end{proof}
\begin{defn}\label{df:fGTGH}
For each $\ell \in \ZZ$,  define ${}_k\GC_m^{(\ell)}(x,a|b)$ and ${}_k\GB_m^{\,(\ell)}(x,a|b)$ in $\calK_{\infty}^{(k)}$ by
\[
\sum_{m\in \ZZ} {}_k\GC_m^{(\ell)}(x,a|b) u^m 
=\begin{cases}
{}_k\GC(x,a;u) \displaystyle\prod_{i=1}^{|\ell|}\dfrac{1}{1 + (u+\beta)\bar b_i}& (\ell< 0),\\
{}_k\GC(x,a;u) \displaystyle\prod_{i=1}^{\ell}(1 + (u+\beta) b_i) & (\ell\geq 0),
\end{cases}
\]
and
\[
\sum_{m\in \ZZ} {}_k\GB_m^{\,(\ell)}(x,a|b) u^m 
=\begin{cases}
{}_k\GC(x,a;u)\displaystyle\prod_{i=1}^{|\ell|}\dfrac{1}{1 + (u+\beta)\bar b_i}& (\ell< 0),\\
{}_k\tGB\,(x,a;u)\displaystyle\prod_{i=1}^{\ell}(1 + (u+\beta) b_i) & (\ell \geq 0).
\end{cases}
\]
\end{defn}
\subsection{Factorial $\GC$ and $\GB$-functions for $k$-strict partitions}
Let $\lambda$ be a partition in $\SP^{k}$ of length $r$ and $\chi$ its characteristic index. Let $\calL^{\calK_{\infty}^{(k)}}$ be the ring of formal Laurent series in indeterminates $t_1,\dots, t_r$ defined in Definition \ref{def fls}. Consider the $\calK_{\infty}^{(k)}$-module homomorphism
\[
\phi_{\lambda}: \calL^{\calK_{\infty}^{(k)}} \to \calK_{\infty}^{(k)},
\]
defined by 
\[
\phi_{\lambda}\left(\sum_{\pmb{s}\in \ZZ^r} f_{\pmb{s}}\cdot t_1^{s_1}\cdots t_r^{s_r}\right)=\sum_{\pmb{s}\in \ZZ^r} f_{\pmb{s}} \cdot {}_k\GC_{s_1}^{(\chi_1)}\cdots {}_k\GC_{s_r}^{(\chi_r)}.
\]
It is well-defined in the view of the definition of $\calK_{\infty}^{(k)}$ as the subring of graded formal power series. Also note that we have 
\[
\phi_{\lambda}\left(\frac{t_i^{s_i}}{2+\beta t_i}\right) = {}_k\GB_{s_i}^{\,(\chi_i)}.
\]
\begin{defn}
For each $\lambda\in \SP^k$ of length $r$, define the functions
\begin{eqnarray*}
{}_k\GC_{\lambda}(x,a|b) 
&:=& \phi_{\lambda}\left( t_1^{\lambda_1}\cdots t_r^{\lambda_r}\frac{\prod_{(i,j) \in \Delta_r} (1- \bar t_i/\bar t_j) }{\prod_{(i,j) \in C(\lambda)}(1- t_i/\bar t_j)}\right),\\
{}_k\GB_{\lambda}(x,a|b) 
&:=& \phi_{\lambda}\left( t_1^{\lambda_1}\cdots t_r^{\lambda_r}\prod_{1\leq i\leq r \atop{\chi_i\geq 0}} \frac{1}{2+\beta t_i}\frac{\prod_{(i,j) \in \Delta_r} (1- \bar t_i/\bar t_j) }{\prod_{(i,j) \in C(\lambda)}(1- t_i/\bar t_j)}\right).
\end{eqnarray*}
If the dependancy on $k$ is clear from the context, we will suppress $k$, {\it i.e.} we will write $\GC_{\lambda}(x,a|b):={}_k\GC_{\lambda}(x,a|b)$ and $\GB_{\lambda}(x,a|b):={}_k\GB_{\lambda}(x,a|b)$.
\end{defn}
\begin{rem}\label{rem:GTminus}
A direct computation shows that ${}_k\GC_m(x,a) = (-\beta)^{-m}$ for each $m\leq 0$.
\end{rem}
It is clear from the definition and Remark \ref{rem:GTminus} that both of $\GC_{\lambda}(x,a|b)$ and  $\GB_{\lambda}(x,a|b)$ are elements of $\calK_{\infty}^{(k)}$. Let us stress that $\GC_{\lambda}$ and $\GB_{\lambda}$ depend on $k$, since $\lambda$ is considered as a $k$-strict partition in $\SP^k$. By the result of Section \ref{secSPf}, $\GC_{\lambda}(x,a|b)$ and $\GB_{\lambda}(x,a|b)$ respectively have the Pfaffian expressions similarly to Theorem \ref{mainC} and \ref{mainB}.
\begin{example}
For a partition $\lambda=(\lambda_1) \in \SP^k$ of length $1$, the corresponding characteristic index is $\chi=(\lambda_1-k-1)$. In this case, we have $\GC_{(\lambda_1)}(x,a|b)={}_k\GC_{\lambda_1}^{(\lambda_1-k-1)}(x,a|b)$ and 
$\GB_{(\lambda_1)}(x,a|b) = {}_k\GB_{\lambda_1}^{\,(\lambda_1-k-1)}(x,z|b)
$.
\end{example}
\subsection{The map $\Psi_{\infty}^{(k)}$ from $\calK_{\infty}^{(k)}$ to $\CK^*_{T_{\infty}}(\calG_{\infty}^k)$}
Together with the restriction, the map $\calR_b \to \CK^*_{T_n}$ defined by setting $b_i=0$ for all $i\geq n+1$ defines the homomorphism $\widetilde{p_n}: \Fun(\SP^{k},\calR_b) \to \Fun(\SP^{k}(n),\CK^*_{T_n})$. By restriction, we obtain a morphism $\frakK^{(k)}_{\infty} \to \frakK^{(k)}_n$ also denoted by $\widetilde{p_n}$. 

Theorem \ref{thmGKM} and Proposition \ref{universal localization} allow us to define the homomorphism  of $\calR_b$-algebras
\[
\Psi_n^{(k)}: \calK_{\infty}^{(k)} \to \CK^*_{T_n}(\calG_n^k)
\]
by the composition
\[
\xymatrix{
\calK_{\infty}^{(k)} \ar[rr]_{\Phi_\infty^{(k)}}&& \frakK^{(k)}_{\infty} \ar[rr]_{\widetilde p_n}&&\frakK^{(k)}_n \ar[rr]_{(\iota_n^*)^{-1}\ \ \ \ \ \ }^{\cong\ \ \ \ \ \ }& &  \CK^*_{T_n}(\calG_n^k).
}
\]
\begin{prop}\label{propvarpi}
The morphisms $\Psi_n^{(k)}$ induce injective homomorphism of graded $\calR_b$-algebras
\[
\Psi_{\infty}^{(k)}:\calK_{\infty}^{(k)}   \to\CK^*_{T_{\infty}}(\calG_{\infty}^k).
\]
\end{prop}
\begin{proof}
We have the following commutative diagrams
\begin{equation}\label{eqcommvarpi}
\xymatrix{
\frakK^{(k)}_{\infty} \ar[rr]_{\widetilde p_{n+1}}\ar[rrd]_{\widetilde p_n}&&\frakK^{(k)}_{n+1} \ar[d]_{p_n}\ar[rr]_{(\iota_{n+1}^*)^{-1}\ \ \ \ \ \ }^{\cong\ \ \ \ \ \ }& &  \CK^*_{T_{n+1}}(\calG_{n+1}^k)\ar[d]_{j_n^*}\\
&&\frakK^{(k)}_n \ar[rr]_{(\iota_n^*)^{-1}\ \ \ \ \ \ }^{\cong\ \ \ \ \ \ }& &  \CK^*_{T_n}(\calG_n^k)
}
\end{equation}
where $p_n$ is defined by setting $b_{n+1}=0$ and restricting the domain of maps from $\SP^k(n+1)$ to $\SP^k(n)$. This defines the commutative diagram
\begin{equation}\label{commdiagvarpi}
\xymatrix{
\calK_{\infty}^{(k)} \ar[rr]^{\Psi_{n+1}^{(k)}} \ar[rrd]_{\Psi_{n}^{(k)}}&& \CK^*_{T_{n+1}}(\calG_{n+1}^k)\ar[d]^{j_n^*}\\
&&\CK^*_{T_n}(\calG_n^k).
}
\end{equation}
and hence we have the morphism $\Psi_{\infty}^{(k)}$. It is easy to see that the morphism $\frakK^{(k)}_{\infty} \to \displaystyle\lim_{\longleftarrow\atop{n}} \frakK^{(k)}_n$ induced by $p_n$ is injective, where $\displaystyle\lim_{\longleftarrow\atop{n}}$ denotes the direct sum of the projective limits of each graded piece. By the diagram (\ref{eqcommvarpi}), we find that $(\iota_n^*)^{-1}$ induces an isomorphism $\displaystyle\lim_{\longleftarrow\atop{n}} \frakK^{(k)}_n \to \CK^*_{T_{\infty}}(\calG_{\infty}^k)$. Hence, together with the injectivity of $\Phi_\infty^{(k)}$ stated in Proposition \ref{universal localization}, we can conclude that $\Psi_{\infty}^{(k)}$ is injective.
\end{proof}
Recall that $\calG_n^k=\SG^k(\bbF^{2n})$ or $\OG^k(\bbF^{2n+1})$ and $\Omega_{\lambda}^X=\Omega_{\lambda}^C$ or $\Omega_{\lambda}^B$ for $X=C$ or $B$ respectively. Similarly we introduce the notations $\scX_m^{(\ell)}$, ${}_k\GX_m^{(\ell)}$ and ${}_k\GX_{\lambda}$ as in the table
\[
\begin{array}{c|c|c|c|c|c}
		& \calG_n^k & \scX_m^{(\ell)}&\Omega_{\lambda}^X & {}_k\GX_m^{(\ell)} & {}_k\GX_{\lambda} \\
\hline
X=C &\SG^k(\bbF^{2n})&\scC_m^{(\ell)}&\Omega_{\lambda}^C&{}_k\GC_m^{(\ell)}&{}_k\GC_{\lambda}\\
\hline
X=B &\OG^k(\bbF^{2n+1})& \scB_m^{(\ell)}&\Omega_{\lambda}^B&{}_k\GB_m^{\,(\ell)}&{}_k\GB_{\lambda}
\end{array}
\]
\begin{prop}\label{propGT=Seg} 
For $-n \leq \ell \leq n$ and $m\in \ZZ$, we have $\Psi_n^{(k)}({}_k\GX_m^{(\ell)}(x,a|b)) = \scX_m^{(\ell)}$.
\end{prop}
\begin{proof}
The claim is a generalization of \cite[Lemma 10.3]{IkedaMihalceaNaruse} and the proof is analogous. Indeed, it follows from the comparison of the localizations at $e_{\lambda}$. For example, if $\calG_n^k=\SG^k(\bbF^{2n})$, it suffices to show
\begin{equation}\label{eq:loc=loc}
\Phi_{\mu}\left(\sum_{m\in \ZZ}  {}_k\GC_m^{(\ell)} u^m\right)=\iota_{\mu}^*\left(\sum_{m\in \ZZ}  \scC_m^{(\ell)} u^m\right).
\end{equation}
One can check this identity as follows. 
We have
\begin{eqnarray*}
&&\Phi_{\mu}\left(\sum_{m\in \ZZ} {}_k\GC_m^{(\ell)} u^m\right)\\
&=&\begin{cases}
\dfrac{1}{1+\beta u^{-1}}  \prod_{i=1}^{s} \dfrac{1 + (u+\beta)\bar b_{\zeta_i}}{1 + (u+\beta)b_{\zeta_i}}  \prod_{i=1}^{k} (1 + (u+\beta) \bar b_{u_i}) \prod_{i=1}^{\ell}(1 + (u+\beta) b_i) & (\ell\geq 0),\\
\dfrac{1}{1+\beta u^{-1}}  \prod_{i=1}^{s} \dfrac{1 + (u+\beta)\bar b_{\zeta_i}}{1 + (u+\beta)b_{\zeta_i}}  \prod_{i=1}^{k} (1 + (u+\beta) \bar b_{u_i}) \prod_{i=1}^{|\ell|}\dfrac{1}{1 + (u+\beta)\bar b_i} & (\ell \leq 0).
\end{cases}
\end{eqnarray*}
On the other hand, if  $\bfz_{k+1},\dots, \bfz_{n}$ are the Chern roots of $U$, we have 
\begin{eqnarray*}
\sum_{m\in \ZZ} \scC_m^{(\ell)}   u^m 
&=&
\begin{cases}
\dfrac{1}{1+\beta u^{-1}} \dfrac{\prod_{i=1}^n(1+(u+\beta)\bar b_i)}{\prod_{i=k+1}^n (1+(u+\beta) \mathbf{z}_i)} \prod_{i=1}^{\ell}(1+(u+\beta)  b_i)& (\ell\geq 0),\\
\dfrac{1}{1+\beta u^{-1}}\dfrac{\prod_{i=1}^n(1+(u+\beta)\bar b_i)}{\prod_{i=k+1}^n (1+(u+\beta) \mathbf{z}_i)} \prod_{i=1}^{|\ell|}\dfrac{1}{1+(u+\beta) \bar b_i} & (\ell\leq 0).
\end{cases}
\end{eqnarray*} 
Then Proposition \ref{propUloc} proves (\ref{eq:loc=loc}) and hence the claim holds. The proof for the case $\calG_n^k=\OG^k(\bbF^{2n+1})$ is similar.
\end{proof}
\begin{thm}\label{thmGBGC}
For each $X=B,C$, we have
\begin{equation}\label{GT to Omega}
\Psi_n^{(k)}({}_k\GX_{\lambda}(x,a|b)) =
\begin{cases}
[\Omega_{\lambda}^X]_{T_n} & \mbox{if $\lambda \in \SP^k(n)$}\\
0 & \mbox{if $\lambda\not\in \SP^k(n)$.}
\end{cases}
\end{equation}
In particular, the homomorphism $\Psi_{\infty}^{(k)}:\calK_{\infty}^{(k)} \to\CK^*_{T_{\infty}}(\calG_{\infty}^k)$ sends ${}_k\GX_{\lambda}$ to the limit  $[\Omega_{\lambda}^X]_T$ of the Schubert classes.
\end{thm}
\begin{proof} 
If $\lambda\in \SP^k(n)$, then $[\Omega_{\lambda}^X]_{T_n}$ and ${}_k\GX_{\lambda}$ are given by the same formula except that $\phi$ replaces $ t_i^{m_i}$ by $\scX_{m_i}^{(\chi_i)}$ or respectively by ${}_k\GX_{m_i}^{(\chi_i)}$ respectively. Therefore (\ref{GT to Omega}) follows from Proposition \ref{propGT=Seg}. For the vanishing, it suffices to show the case when $\lambda\in \SP^k(n+1)\backslash\SP^k(n)$. By the commutative diagram (\ref{commdiagvarpi}), we have $\Psi_{n}^{(k)}(\GX_{\lambda}(x,a|b))  = j_n^*[\Omega_{\lambda}^X]_{T_{n+1}}$ which is $0$ by (\ref{stabSch}). 
\end{proof}
As a consequence of Theorem \ref{thmGBGC} we obtain the following fact.
\begin{cor}
The functions ${}_0\GC_{\lambda}(x|b)$ and ${}_0\GB_{\lambda}(x|b)$ respectively coincide with $\GQ_{\lambda}(x|b)$ and $\GP_{\lambda}(x|0,b)$ defined by Ikeda--Naruse in \cite{IkedaNaruse}. 
\end{cor}
By Lemma \ref{lemLimBasis} and Theorem \ref{thmGBGC} we deduce the following corollary.
\begin{cor}
Any element $f$ of $\calK_{\infty}^{(k)}$ can be expressed uniquely as a possibly infinite $\calR_b$-linear combination
\[
f = \sum_{\lambda \in \SP^k} c_{\lambda}^X(b) {}_k\GX_{\lambda}(x,a|b), \ \ \ \ c_{\lambda}^X(b) \in \calR_b.
\]
\end{cor}
\begin{rem}
One can prove that the functions ${}_k\GX_{\lambda}(x,a|b), \lambda\in \SP^k$ coincide with the double Grothendieck polynomials $\calG_w^X(a,b;x), w\in W_{\infty}^{(k)}$. The details will be discussed in our future work. 
\end{rem}

 
\vspace{10mm}

\textbf{Acknowledgements.} 
We would like to thank Leonardo Mihalcea and Anders Buch for their helpful comments on the earlier versions of the manuscript. A considerable part of this work developed while the first and third authors were affiliated to KAIST, which they would like to thank for the excellent working conditions.
Part of this work developed while the first author was affiliated to  POSTECH, which he would like to thank for the excellent working conditions. He would also like to gratefully acknowledge the support of the National Research Foundation of Korea (NRF) through the grants funded by the Korea government (MSIP) (ASARC, NRF-2007-0056093) and (MSIP)(No.2011-0030044).
The second author  is supported by Grant-in-Aid for Scientific Research (C) 24540032, 15K04832.
The third author is supported by 	Grant-in-Aid for Young Scientists (B) 16K17584.
The fourth author is supported by Grant-in-Aid for Scientific Research (C) 25400041, (B) 16H03921.

\bibliographystyle{acm}
\bibliography{references}   

\

\begin{small}
{\scshape
\noindent Thomas Hudson, Fachgruppe Mathematik
und Informatik, Bergische Universit\"{a}t Wuppertal, Gaußstrasse 20, 42119 Wuppertal, Germany
}
\end{small}

{\textit{email address}: \tt{hudson@math.uni-wuppertal.de}}



\

\begin{small}
{\scshape
\noindent  Takeshi Ikeda, Department of Applied Mathematics, Okayama University of Science, Okayama 700-0005, Japan
}
\end{small}

{\textit{email address}: \tt{ike@xmath.ous.ac.jp}}

\

\begin{small}
{\scshape
\noindent Tomoo Matsumura, Department of Applied Mathematics, Okayama University of Science, Okayama 700-0005, Japan
}
\end{small}

{\textit{email address}: \tt{matsumur@xmath.ous.ac.jp}}

\

\begin{small}
{\scshape
\noindent Hiroshi Naruse, Graduate School of Education, University of Yamanashi, Yamanashi 400-8510, Japan
}
\end{small}

{\textit{email address}: \tt{hnaruse@yamanashi.ac.jp}}

\end{document}